\documentclass[10pt,a4paper]{article}
\date{}

\usepackage{subcaption}
\usepackage[margin=.1\textwidth]{caption}
\usepackage[para]{bigfoot}

\usepackage{tikz}
\usetikzlibrary{shapes.geometric, arrows}
\usetikzlibrary{intersections}

\renewcommand{\geq}{\geqslant}
\renewcommand{\leq}{\leqslant}

\usepackage{amsthm}
\usepackage{amsmath}
\usepackage{amsfonts, amssymb, enumerate}
\usepackage{psfrag,graphicx}
\usepackage[hidelinks]{hyperref}
\hypersetup{breaklinks=true}
\usepackage{array}
\usepackage{color}
\usepackage{caption}

\newtheorem{prpstn}{Proposition}[section]
\newtheorem{thrm}[prpstn]{Theorem}
\newtheorem{rmrk}[prpstn]{Remark}
\newtheorem{lmm}[prpstn]{Lemma}
\newtheorem{crllr}[prpstn]{Corollary}
\newtheorem{dfntn}[prpstn]{Definition}
\theoremstyle{remark}
\newtheorem{xmpl}[prpstn]{Example}

\newcommand{\R}{{\mathbb{R}}}
\newcommand{\N}{{\mathbb{N}}}
\newcommand{\Z}{{\mathbb{Z}}}
\newcommand{\BV}{{\mathbf{BV}}}
\renewcommand{\L}[1]{\mathbf{L^{\pmb #1}}}
\newcommand{\C}[1]{\mathbf{C^{\pmb #1}}}
\newcommand{\Cc}[1]{\mathbf{C_c^{#1}}}
\newcommand{\Lloc}[1]{\mathbf{L_{loc}^{\pmb #1}}}
\newcommand{\rsp}{\mathcal{RS}_{\rm p}}
\newcommand{\rsv}{\mathcal{RS}_{\rm v}}
\newcommand{\rsh}{\mathcal{RS}_{\rm h}}
\newcommand{\rsc}{\mathcal{RS}_{\rm c}}

\newcommand{\rs}{\mathcal{RS}}
\renewcommand{\d}{{\rm{d}}}
\usepackage{mathtools}

\usepackage{calc}
\usepackage[font={small}]{caption}

\title{Coherence and flow-maximization of a one-way valve}

\author{Andrea\ Corli$^*$, Ulrich Razafison$^{**}$, Massimiliano\ D.\ Rosini$^{*,***}$
\\
\small\textit{$^{*}$Department of Mathematics and Computer Science, University of Ferrara}
\\
\small\textit{I-44121 Italy}
\\
\small\textit{$^{**}$Laboratoire de math\'ematiques, CNRS UMR 6623, Universit\'e de Franche-Comt\'e, 16 route de Gray}
\\
\small\textit{25030 Besan\c{c}on, France}
\\
\small\textit{$^{***}$Uniwersytet Marii Curie-Sk\l odowskiej, Plac Marii Curie-Sk\l odowskiej 1}
\\
\small\textit{20-031 Lublin, Poland}
}

\usepackage[utf8]{inputenc}
\usepackage{enumitem}

\numberwithin{equation}{section}

\let\originalleft\left
\let\originalright\right
\renewcommand{\left}{\mathopen{}\mathclose\bgroup\originalleft}
\renewcommand{\right}{\aftergroup\egroup\originalright}

\begin{document}

\allowdisplaybreaks
\maketitle
\begin{abstract}
We consider a mathematical model for the gas flow through a one-way valve and focus on two issues. First, we propose a way to eliminate the chattering (the fast switch on and off of the valve) by slightly modifying the design of the valve. This mathematically amounts to the construction of a coupling Riemann solver with a suitable stability property, namely, coherence. We provide a numerical comparison of the behavior of the two valves. Second, we analyze, both analytically and numerically, for several significative situations, the maximization of the flow through the modified valve according to a control parameter of the valve and time.


\end{abstract}

\section{Introduction}

This paper aims at improving the modeling of a gas flow through a one-way valve, thus carrying on the research on the same subject in \cite{CFFR1, CR1, CR2}. The motivation of these studies is to provide an {\em analytic} modeling of such flows, which we base on hyperbolic systems of conservation laws and their Riemann solvers. Two related issues are considered here. First, how to remove the chattering (the fast switch on and off) of a valve in correspondence of some threshold states \cite{Hos-Champneys, Ulanicki-Skworcow}, by slightly modifying the design of the valve. Second, the maximization of the flow through the new valve, according to a characteristic parameter of the valve and time.

Since we focus on the behavior of the valve, we make some simplifying assumptions on the flow. The gas flow takes place along two straight pipes having equal and constant cross-sections; the position along the pipes is denoted by $x\in\R$ and they are joint by a valve at $x=0$. Thus we neglect the wall deformation of the pipes under pressure loads. The flow is characterized by the mass density $\rho>0$ and the velocity $v$ of the gas; we assume that it is isothermal and so we take 
\[
p(\rho) \doteq a^2\rho
\] 
as pressure law, where the constant $a>0$ is the sound speed. The flow is governed by the Euler equations
\begin{equation}\label{e:system}
\begin{cases}
\partial_t\rho + \partial_x (\rho \, v) =0,
\\
\partial_t(\rho \, v) + \partial_x\left(\rho \, v^2 + p(\rho)\right) = 0,
\end{cases}
\end{equation}
where $t>0$ is the time.
The initial-value problem for system \eqref{e:system} when the initial data are constant, apart for a single jump, is called Riemann problem and is well understood \cite{LeVeque-book}. The solution is autosimilar and provided by a {\em Riemann Solver} $\rsp$; it consists of constant states separated by shock or rarefaction waves. The solver $\rsp$ satisfies several properties but in particular it is {\em coherent}; in a few words, this means that, for any solution $u\doteq(\rho,q)$ provided by $\rsp$, with $q\doteq\rho\,v$ being the momentum, the Riemann problem having for initial data the traces $u(t,x^\pm)$ and solved by $\rsp$ leads to a function having the same local behavior of $u$ in a neighborhood of $(t,x)$. This property can be understood as a sort of interior stability of $\rsp$.

\smallskip

The modeling, maximization, and control of gas flows through networks of pipes have been recently considered in several papers; we refer to \cite{Gugat-Herty_survey} for a comprehensive survey on the subject and just quote some relevant references. Among the first papers dealing with a rigorous mathematical modeling of gas flows in networks we quote \cite{Banda-Herty-Klar2, Banda-Herty-Klar1}. There, the authors use system \eqref{e:system} and require that only positive velocities are admitted; moreover, at each junction, either the pressure is continuous or the flow is subsonic. We refer to Proposition \ref{p:subsonic} below for a discussion of the latter issue. Then, they find a unique solution under the additional condition that the flow is maximal at each junction.  They also provide some simple numerical modelings of two valves, which have either maximal \cite{Banda-Herty-Klar2} or zero \cite{Banda-Herty-Klar1} flux on the outgoing pipe.
Optimization and control problems are considered in \cite{Banda-Herty, Herty-Gugat, Gugat-Herty-Schleper, Herty-compressors, Herty-Sachers} for gas flows with compressors; in \cite{GLMSSW2, GLMSSW1, Martin-Moller-Moritz} also valves are present, and the problem is solved by suitable discretisations of the modeling equations. We emphasize that in the latter papers the treatment of valves is very different from ours because it is based on the supply and demand functions.

\smallskip

In this paper we pursue the analysis started in \cite{CR1} and focus again on flow-control valve: roughly speaking, for a fixed flow value $q_*$, the valve keeps the flow equal to $q_*$ if possible, otherwise it closes. As in \cite{CFFR1, CR1, CR2}, the effect of a valve is reproduced by enforcing a coupling between the ingoing and the outgoing flow at $x=0$.
This is encoded by the so-called {\em coupling Riemann solver}, c-Riemann solver for short, that gives solutions to the Riemann problem at $x=0$ for \eqref{e:system} for the coupling problem induced by the valve. The analysis of coupled Riemann solvers has a long history, starting from the seminal paper \cite{Godlewski-Raviart}; we refer the reader to the recent article \cite{Boutin-Coquel-LeFloch_II} for general information and detailed references. We point out, however, that most of the papers in the literature either deal with scalar equations (we consider a system of two equations) or require stricter coupling condition than ours, for example the continuity of the traces at $x=0$ (in our case this will only hold for the momentum component of the solution). From a mathematical point of view, our modeling rather lies in the framework of constrained Riemann problems: see \cite{Colombo-Goatin} for scalar conservation laws, \cite{Garavello-Goatin} for a $2\times2$ system, and \cite{MR3965292} for recent advances. We also point out that, with respect to \cite{Banda-Herty-Klar2, Banda-Herty-Klar1}, we neither require the continuity of the pressure nor that flows are subsonic at the valve.

Different valves correspond to different c-Riemann solvers, see \cite{CFFR1, CR1, CR2} and Sections~\ref{s:prelim},~\ref{s:cova} for some examples. 
A key analytic feature of the modeling is the possible {\em incoherence} of the corresponding c-Riemann solver, which is related to chattering. Clearly, incoherence leads to the numerical instability of the solution, see for instance the central column in \figurename~\ref{f:comparison:RSh_RSv}.

\smallskip

In this paper we slightly modify the c-Riemann solver $\rsv$ introduced in \cite[Section~4]{CR1}; the motivation is that $\rsv$ is incoherent. Incoherent states for $\rsv$ are supersonic, indeed.
In most real gas-flows through pipe networks, supersonic states do not occur, and the reason is attributed to friction terms \cite{Banda-Herty-Klar2}, \cite[pages 45 and 49]{Taylor} and safety reasons. Supersonic flows {\em do} occur in particular circumstances, indeed, see for example \cite{Modesti-Pirozzoli} and references therein.
Moreover, from a mathematical point of view, the invariant domains for the Riemann solver $\rsp$ always contain supersonic states (see \cite{LeVeque-book}) and, as a consequence, these states can appear even if they are not present initially; the latter happens as well for $\rsv$, as we show in Proposition \ref{p:subsonic} {\em (ii)}. The issue is how to modify $\rsv$ to recover coherence.
The new proposed c-Riemann solver $\rsh$
\begin{itemize} 
\item is coherent; 
\item differs from $\rsv$ {\em only} for the states that lead $\rsv$ to lose coherence;
\item for incoherent initial data, it selects the unique solution that maximizes the flow through the valve among all c-Riemann solvers.
\end{itemize}
The last property deserves a comment. As it has been first pointed out in \cite{Holden-Risebro}, such a condition is understood as a sort of \lq\lq entropy\rq\rq\ condition  \cite{Banda-Herty-Klar2, Banda-Herty-Klar1, garavellopiccoli-book, Holden-Risebro}, because it singles out  solutions uniquely. It is interesting to observe that this property is suggested by the behavior of a valve with a positive reaction time, see Section \ref{s:cova}.

We provide some numerical simulations letting us conjecture that  $\rsh$ furnishes the solution obtained by applying $\rsv$ at every time step $\Delta t$ (where chattering occurs) as $\Delta t\to0$, see the last two columns in \figurename~\ref{f:comparison:RSh_RSv}. We currently miss of a general analytic proof of the latter statement; nevertheless, such simulations suggest that $\rsh$ reproduces the {\em final effects} on the gas flow of a chattering valve, without the inconvenient of the (numerical) instability caused by incoherence.
This suggests a different design of the valve corresponding to $\rsv$, replacing it with the valve corresponding to $\rsh$.
By the way, the study of the coherence of a constrained Riemann problem is surely of interest from the mathematical point of view.

\smallskip

Here follows an outline of the paper. In Section \ref{s:prelim} we first introduce our notation and quickly review some basic facts about system \eqref{e:system}. We emphasize that, there and in the following, {\em the Lax curves represented in the pictures are always exact} and not merely qualitative. Then we summarize the modeling of the flow through a valve, with a special focus to the above mentioned valve.  Section \ref{s:cova} provides the definition of the modified Riemann solver $\rsh$ and the proof of its coherence. 
In Section \ref{s:Ulrich} we first introduce the numerical scheme to be used in the following and the reasons of our choice; then we give some comparisons between exact and numerical solutions.
Section \ref{s:max} is the core of the paper. There, we first state the maximization problem under study; it depends both on the flow threshold $q_*$ and on the time horizon $T$. In same simple cases the solutions can be computed analytically, and they are compared to the numerical solutions to further validate the numerical scheme. Then we give some numerical simulations of more complicated situations, in particular dealing with the perturbation of the incoming flow of either a shock or a rarefaction. 

\section{Preliminary results and notation}\label{s:prelim}
In this section we first briefly recall the main facts about system \eqref{e:system} with the pressure law $p(\rho)=a^2\rho$, in particular for what concerns Lax curves and their properties. All of them are well known, see \cite{LeVeque-book} and \cite{Bressan-book, Dafermos-book} for general information, but this avoids us to systematically refer the reader to other books or papers. Then, we summarize the modeling of a gas flow through a one-way valve \cite{CR1}, which is located at $x=0$; we also provide some new results.  

\smallskip

We always deal with the conservative variables $u\doteq(\rho,q)$, where $q\doteq\rho \, v$ is the momentum, so that system \eqref{e:system} can be written as
\begin{equation}\label{e:systemq}
\begin{cases}
\partial_t\rho + \partial_xq=0,
\\
\partial_tq + \partial_x\left(\frac{q^2}{\rho} + a^2 \rho\right) = 0.
\end{cases}
\end{equation}
We denote $\Omega \doteq \{(\rho,q) \in \R^2 : \rho>0\}$. 
The eigenvalues of \eqref{e:systemq} are $\lambda_1(u) \doteq \frac{q}{\rho} - a$, and $\lambda_2(u) \doteq \frac{q}{\rho} + a$; system \eqref{e:systemq} is strictly hyperbolic in $\Omega$ and $\lambda_1,\lambda_2$ are genuinely nonlinear. 
The Riemann problem for \eqref{e:systemq} is the Cauchy problem with initial condition
\begin{equation}\label{e:Riemann}
u(0,x) = 
\begin{cases}
u_\ell &\hbox{ if } x<0,
\\
u_r &\hbox{ if } x\geq0,
\end{cases}
\end{equation}
where $u_\ell,u_r\in\Omega$ are constant states. Solutions to \eqref{e:systemq}, \eqref{e:Riemann} are meant in the weak sense as follows.

\begin{dfntn}
A function $u \in \C0([0,\infty);\Lloc1(\R;\Omega))$ is a \emph{weak solution} to the Riemann problem \eqref{e:systemq}, \eqref{e:Riemann} in $[0,\infty)\times\R$ if for any $\varphi \in \Cc\infty([0,\infty)\times\R;\R)$ we have
\begin{align*}
\int_0^\infty\int_{\R} \Bigl[ \rho \, \partial_t\varphi + q \, \partial_x\varphi \Bigr] \d x \,\d t
+ \rho_\ell \int_{-\infty}^0 \varphi(0,x) \,\d x + \rho_r \int_0^\infty \varphi(0,x) \,\d x &= 0,
\\
\int_0^\infty\int_{\R} \Bigl[ q \, \partial_t\varphi + \Bigl(\frac{q^2}{\rho^2}+a^2\Bigr) \rho \, \partial_x\varphi \Bigr] \d x \,\d t
+ q_\ell \int_{-\infty}^0 \varphi(0,x) \,\d x + q_r \int_0^\infty \varphi(0,x) \,\d x &= 0.
\end{align*}
\end{dfntn}

If $x=\gamma(t)$ is a smooth curve along which a weak solution $u$ is discontinuous, then the following Rankine-Hugoniot conditions must be satisfied, where $u^\pm(t) \doteq u(t,\gamma(t)^\pm)$ are the traces of $u$ along $x=\gamma(t)$:
\begin{align}\label{e:RH1}
&\left(\rho^+-\rho^-\right) \dot{\gamma} = q^+-q^-,\\
\label{e:RH2}
&\left(q^+-q^-\right) \dot{\gamma} = \left(\dfrac{(q^+)^2}{\rho^+} + a^2 \, \rho^+\right) - \left(\dfrac{(q^-)^2}{\rho^-} + a^2 \, \rho^-\right).
\end{align}

For $u_o\in\Omega$ we define $\mathcal{FL}_i^{u_o}, \mathcal{BL}_i^{u_o} : (0,\infty) \to \R$, $i\in\{1,2\}$, by
\begin{align*}
\mathcal{FL}_1^{u_o}(\rho) &\doteq
\begin{cases}
\mathcal{R}_1^{u_o}(\rho)&\hbox{ if }\rho \in (0,\rho_o],
\\
\mathcal{S}_1^{u_o}(\rho)&\hbox{ if }\rho \in (\rho_o,\infty),
\end{cases}&
\mathcal{FL}_2^{u_o}(\rho) &\doteq
\begin{cases}
\mathcal{S}_2^{u_o}(\rho)&\hbox{ if }\rho \in (0,\rho_o),
\\
\mathcal{R}_2^{u_o}(\rho)&\hbox{ if }\rho \in [\rho_o,\infty),
\end{cases}
\\
\mathcal{BL}_1^{u_o}(\rho) &\doteq
\begin{cases}
\mathcal{S}_1^{u_o}(\rho)&\hbox{ if }\rho \in (0,\rho_o),
\\
\mathcal{R}_1^{u_o}(\rho)&\hbox{ if }\rho \in [\rho_o,\infty),
\end{cases}&
\mathcal{BL}_2^{u_o}(\rho) &\doteq
\begin{cases}
\mathcal{R}_2^{u_o}(\rho)&\hbox{ if }\rho \in (0,\rho_o],
\\
\mathcal{S}_2^{u_o}(\rho)&\hbox{ if }\rho \in (\rho_o,\infty),
\end{cases}
\end{align*}
where $\mathcal{S}_i^{u_o}, \mathcal{R}_i^{u_o} : (0,\infty) \to \R$, $i\in\{1,2\}$, are defined by
\begin{align}\label{e:SRi}
\mathcal{S}_i^{u_o}(\rho) 
&\doteq \rho \left( \frac{q_o}{\rho_o} +(-1)^i \, a \left(\sqrt{\dfrac{\rho}{\rho_o}} - \sqrt{\dfrac{\rho_o}{\rho}}\right) \right),&
\mathcal{R}_i^{u_o}(\rho) 
&\doteq \rho \left( \frac{q_o}{\rho_o} +(-1)^i \, a \, \ln\left(\dfrac{\rho}{\rho_o}\right)\right).
\end{align}

The graphs of the functions $\mathcal{FL}_i^{u_o}$ and $\mathcal{BL}_i^{u_o}$ are the {\em forward} $\mathsf{FL}_i^{u_o}$ and {\em backward} $\mathsf{BL}_i^{u_o}$ Lax curves, respectively, of the $i$-th family through $u_o$, see \figurename~\ref{f:RVLaxcurves}.
Analogously, the shock $\mathsf{S}_i^{u_o}$ and rarefaction $\mathsf{R}_i^{u_o}$ curves through $u_o$ are the graphs of the functions $\mathcal{S}_i^{u_o}$ and $\mathcal{R}_i^{u_o}$. The shock speeds are $s_{1}^{u_o}(\rho)\doteq v_o - a\, \sqrt{\rho/\rho_o}$ and $s_{2}^{u_o}(\rho) \doteq v_o + a \, \sqrt{\rho/\rho_o}$.
A state $(\rho,q)\in\Omega$ is \emph{subsonic} if $|v| < a$ and \emph{supersonic} if $|v| > a$;  
the {\em sonic lines} are $q=\pm a \, \rho$. 

\begin{figure}[!htb]\centering
\begin{tikzpicture}[every node/.style={anchor=south west,inner sep=0pt},x=1mm, y=1mm]
\node at (0,0) {\includegraphics[width=45mm]{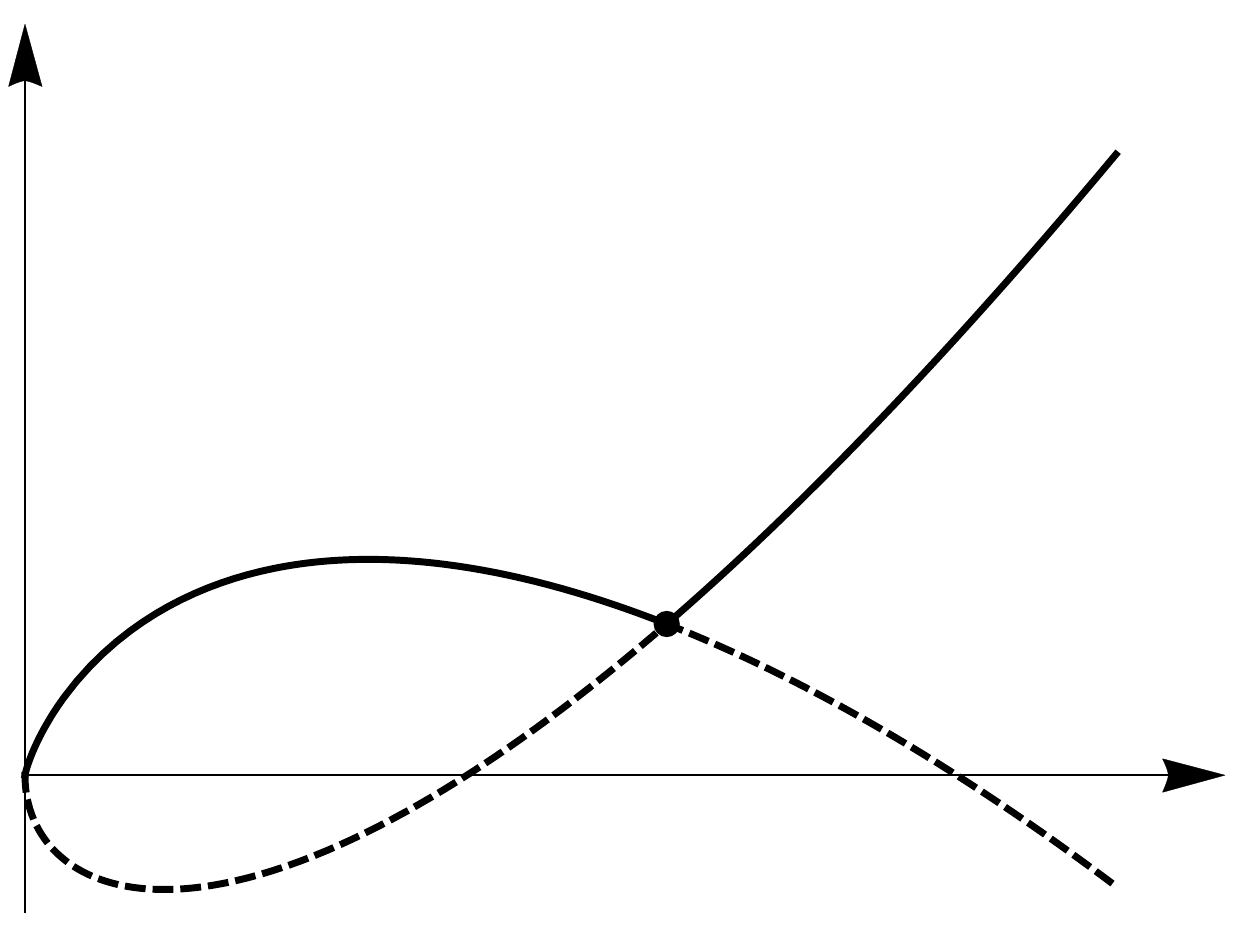}};
\node at (43,1) {\strut $\rho$};
\node at (-3,30) {\strut $q$};
\node at (10,30) {\strut $\mathsf{FL}_1^{u_o} \cup \mathsf{FL}_2^{u_o}$};
\node at (12,14) {\strut $\mathsf{R}_1^{u_o}$};
\node at (40,28) {\strut $\mathsf{R}_2^{u_o}$};
\node at (16,0) {\strut $\mathsf{S}_2^{u_o}$};
\node at (31,8) {\strut $\mathsf{S}_1^{u_o}$};
\node at (21,12) {\strut $u_o$};
\end{tikzpicture}
\quad
\begin{tikzpicture}[every node/.style={anchor=south west,inner sep=0pt},x=1mm, y=1mm]
\node at (0,0) {\includegraphics[width=45mm]{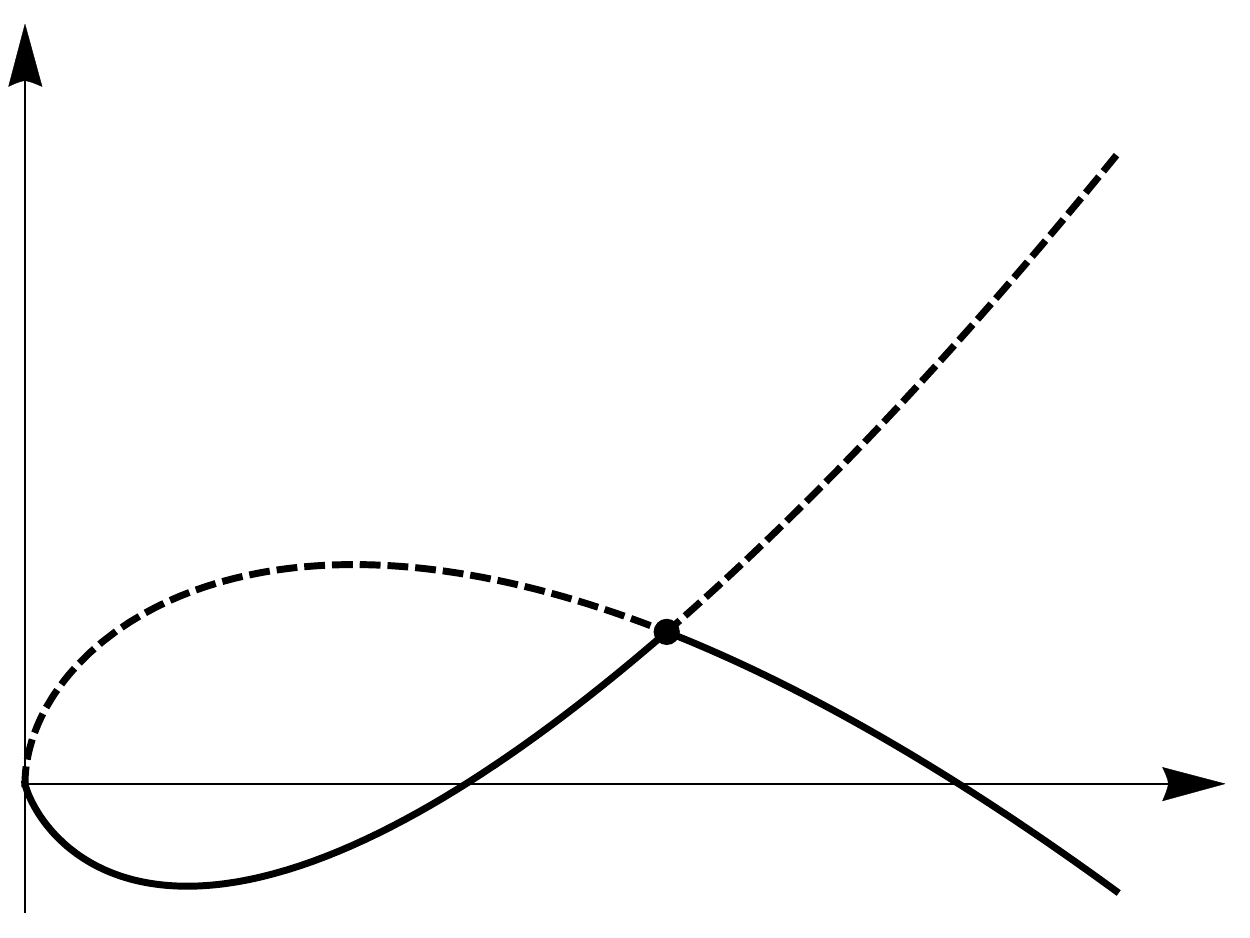}};
\node at (43,1) {\strut $\rho$};
\node at (-3,30) {\strut $q$};
\node at (10,30) {\strut $\mathsf{BL}_1^{u_o} \cup \mathsf{BL}_2^{u_o}$};
\node at (12,14) {\strut $\mathsf{S}_1^{u_o}$};
\node at (40,28) {\strut $\mathsf{S}_2^{u_o}$};
\node at (16,0) {\strut $\mathsf{R}_2^{u_o}$};
\node at (30,8) {\strut $\mathsf{R}_1^{u_o}$};
\node at (21.5,11.5) {\strut $u_o$};
\end{tikzpicture}
\quad
\begin{tikzpicture}[every node/.style={anchor=south west,inner sep=0pt},x=1mm, y=1mm]
\node at (0,0) {\includegraphics[width=45mm]{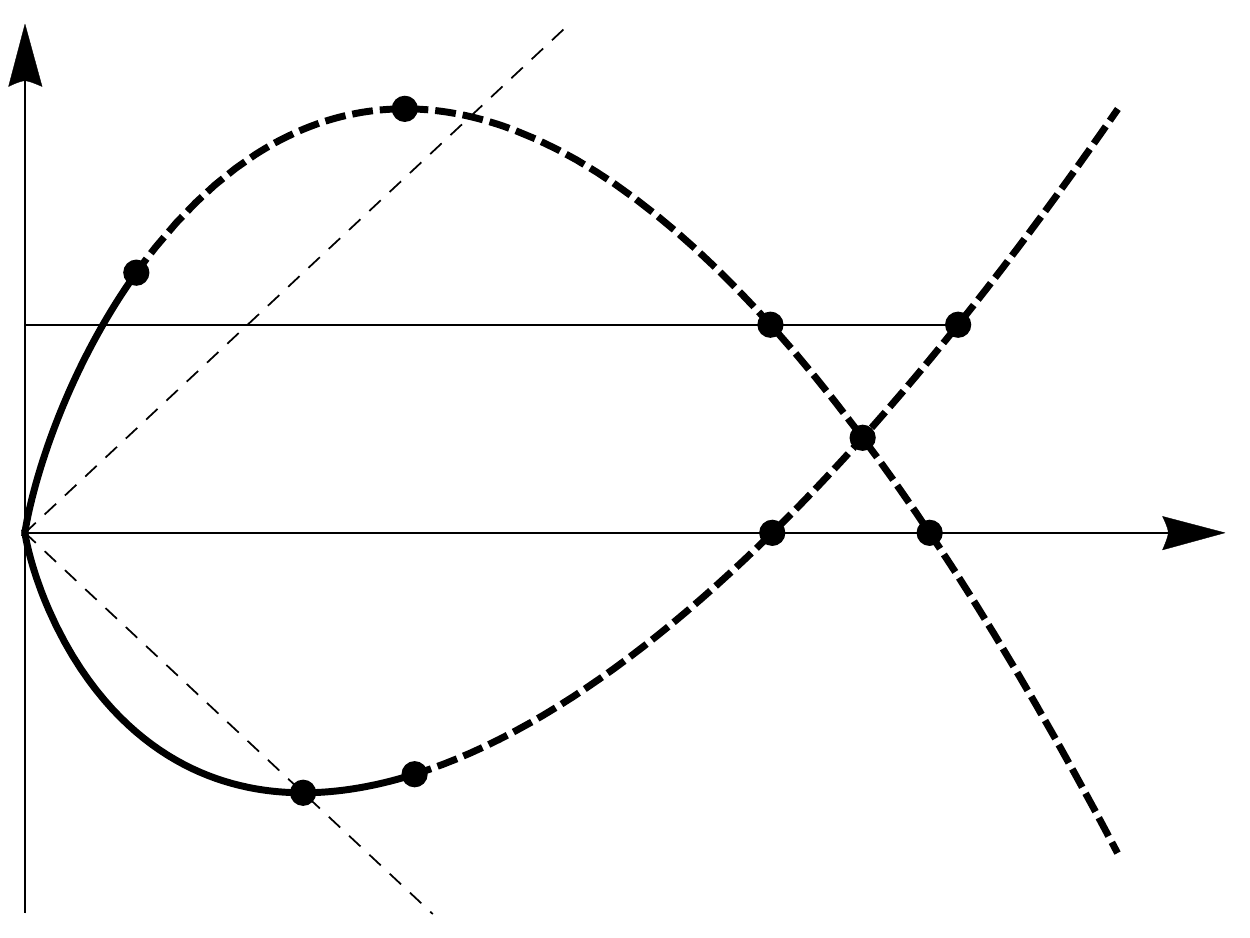}};
\node at (43,10) {\strut $\rho$};
\node at (-3,30) {\strut $q$};
\node at (40,29) {\strut $\mathsf{BL}_2^{u_r}$};
\node at (41,0) {\strut $\mathsf{FL}_1^{u_\ell}$};
\draw[latex-] (34,15) -- (41,17.5) node[right] {\strut $\hat{u}(0,u_\ell)$};
\node at (16,14.5) {\strut $\check{u}(0,u_r)$};
\node at (15,22) {\strut $\hat{u}(q_o,u_\ell)$};
\node at (-3,20) {\strut $q_o$};
\draw[latex-] (32,18.5) -- (41,21.5) node[right] {\strut $\tilde{u}(u_\ell,u_r)$};
\draw[latex-] (35.5,22.5) -- (41,25.5) node[right] {\strut $\check{u}(q_o,u_r)$};
\node at (15.5,2.5) {\strut $u_r$};
\node at (9,30) {\strut $\bar{u}(u_\ell)$};
\node at (1.5,24) {\strut $u_\ell$};
\node at (20,33) {\strut $q=a\,\rho$};
\node at (16.5,-2) {\strut $q=-a\,\rho$};
\end{tikzpicture}
\caption{Left: forward Lax curves. Center: backward Lax curves. Right: an illustration of the quantities in Definition~\ref{def:notation}.}
\label{f:RVLaxcurves}
\end{figure}
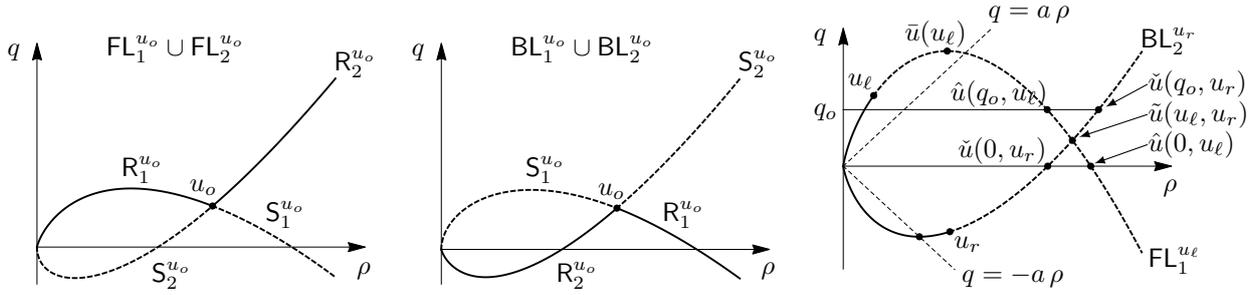

We introduce the following notations, see \figurename~\ref{f:RVLaxcurves} on right for an illustration.

\begin{dfntn}\label{def:notation} 
For $u_\ell,u_r\in\Omega$ we denote:
\begin{itemize}[leftmargin=*]\setlength{\itemsep}{0cm}\setlength\itemsep{0em}
\item 
$\bar{u}(u_\ell)$ is the element of $\mathsf{FL}_1^{u_\ell}$ with the maximum $q$-coordinate;
\item 
$\tilde{u}(u_\ell,u_r)$ is the (unique) element of $\mathsf{FL}_1^{u_\ell} \cap \mathsf{BL}_2^{u_r}$;
\item
$\hat{u}(q_o,u_\ell)$, for any $q_o \leq \bar{q}(u_\ell)$, is the intersection of $\mathsf{FL}_1^{u_\ell}$ and $q=q_o$ with the largest $\rho$-coordinate;
\item
$\check{u}(q_o,u_r)$, for any $q_o \geq 0$, is the intersection of $\mathsf{BL}_2^{u_r}$ and $q=q_o$ with the largest $\rho$-coordinate.
\end{itemize}
\end{dfntn}

Now, we briefly recall the modeling of a gas flow through a \emph{one-way} valve \cite{CR1}.
One-way valves are characterized by letting the flow occur (at $x=0$) in a single direction; we fix the positive one for definiteness. 

We denote by $\BV(\R;\Omega)$ the space of $\Omega$-valued functions with bounded variation. 
We define $\mathsf{D} \doteq \Omega \times \Omega$.
The Lax Riemann solver $\rsp \colon \mathsf{D}\to\BV(\R;\Omega)$, whose action is denoted for $(u_\ell,u_r)\in \mathsf{D} $ by $(t,x) \mapsto \rsp[u_\ell,u_r](x/t)$, provides the unique entropic solution to Riemann problem  \eqref{e:systemq}, \eqref{e:Riemann}, see \cite{LeVeque-book}.
We denote for brevity, when the dependence on initial data is clear,
\begin{align}\label{eq:BMTH}
&u_{\rm p} \doteq \rsp[u_\ell,u_r],&u_{\rm p}^\pm \doteq u_{\rm p}(0^\pm).
\end{align}

\begin{rmrk}\label{r:BHK} If the $1$-wave in $u_{\rm p}$ is a shock with positive speed, then we have $v_\ell>a$ by $\eqref{e:SRi}_1$. If the $1$-wave is a rarefaction, then it enters the region $x>0$ if and only if $v_r>a$; if also $a<v_\ell<v_r$, then the whole rarefaction enters the region $x>0$. An analogous remark holds for $2$-waves. As a consequence, subsonic states $u_\ell,u_r$ never produce waves moving to the same direction.
\end{rmrk}

Solutions to \eqref{e:systemq} for $x\ne0$ will always be given by $\rsp$; at $x=0$ we model the flow through the valve by a {\em c-Riemann solver} $\rs$ (\lq\lq c\rq\rq\ for coupling), as we are going to define.
First, for each $(u_\ell,u_r) \in \mathsf{D}$ we assign the flow $Q = Q(u_\ell,u_r) \in [ 0 , \overline{Q}(u_\ell) ]$ through the valve, where $\overline{Q} \colon \Omega\to\R$ is given by
\begin{equation}\label{e:AliceinChains}
\overline{Q}(u) \doteq
\begin{cases}
\bar{q}(u) = \frac{a \, \rho}{e} \, \exp\left(\frac{v}{a}\right)&\hbox{ if } v \leq a ,
\\
q&\hbox{ if } v > a .
\end{cases}
\end{equation}
We observe that $\overline{Q}\in\C1(\Omega)$.
The introduction of $\overline{Q}(u)$ is needed to select the values of the flow across the valve, in order that the Riemann solver $\rsc$ defined below provides at most {\em one} wave (a $1$-rarefaction or a $1$-shock) on the left of the valve; see Remark~\ref{r:rem4} \ref{itemmm} below.

\begin{dfntn}\label{d:RSV} 
Let $Q \colon \mathsf{D} \to \R$ be such that $Q(u_\ell,u_r) \in [ 0 , \overline{Q}(u_\ell) ]$ for every $(u_\ell,u_r)\in\mathsf{D}$. 
The corresponding c-Riemann solver $\rsc \colon \mathsf{D} \to \BV(\R;\Omega)$ is defined by
\begin{equation}\label{e:rsc}
\rsc[u_\ell,u_r](\xi) \doteq
\begin{cases}
\rsp\left[u_\ell,\hat{u}\left(Q(u_\ell,u_r),u_\ell\right)\right](\xi)&\hbox{ if }\xi<0,\\[5pt]
\rsp\left[\check{u}\left(Q(u_\ell,u_r),u_r\right),u_r\right](\xi)&\hbox{ if }\xi\geq0.
\end{cases}
\end{equation}
\end{dfntn}

Analogously to \eqref{eq:BMTH} we denote
$u_{\rm c} \doteq \rsc[u_\ell,u_r]$ and $u_{\rm c}^\pm \doteq u_{\rm c}(0^\pm)$.

\begin{rmrk}\label{r:rem4}
We now give several explanations of the previous definition and introduce some notations.

\begin{enumerate}[label={\em(\roman*)},leftmargin=*,nolistsep]\setlength{\itemsep}{0cm}\setlength\itemsep{0em}

\item 
A c-Riemann solver $\rsc$ is characterized by $Q$. For brevity we omit the dependence on $Q$.

\item
The conservation of the mass (which corresponds to the first Rankine-Hugoniot condition \eqref{e:RH1}) must hold at $x=0$. This condition  is automatically satisfied by $u_{\rm c}$: if $u_{\rm c}$ has a {\em stationary} discontinuity at $x=0$, then $\dot \gamma=0$ but $q_{\rm c}(0^-) = Q(u_\ell,u_r) = q_{\rm c}(0^+)$ because of the definitions of $\check u$ and $\hat u$, and so \eqref{e:RH1} holds.
On the contrary, the conservation of momentum is lost at $x=0$, in general; hence the second Rankine-Hugoniot condition \eqref{e:RH2}, which encodes this property, cannot be required. Indeed, definition \eqref{e:rsc} does not imply \eqref{e:RH2}. As a consequence, $u_{\rm c}$ may fail to be a weak solution of \eqref{e:systemq} at $x=0$.

\item
We say that for $(u_\ell,u_r)\in\mathsf{D}$ the valve is {\em closed} if $Q(u_\ell,u_r)=0$ and \emph{open} if $Q(u_\ell,u_r) \neq 0$. 
\item
By Definition~\ref{def:notation} and \eqref{e:AliceinChains} we deduce that $\overline{Q}(u_\ell) \geq q_\ell$ and $0 < \overline{Q}(u_\ell) \leq \bar q(u_\ell)$.
So, if $Q \in [ 0 , \overline{Q}(u_\ell) ]$, then $\hat u(Q,u_\ell)$ and $\check u(Q,u_r)$ are well defined.
We denote for short
\begin{align*}
\hat{u} = \hat{u}(u_\ell,u_r) \doteq \hat{u}\left(Q\left(u_\ell,u_r\right),u_\ell\right),
\qquad
\check{u} = \check{u}(u_\ell,u_r) \doteq \check{u}\left(Q\left(u_\ell,u_r\right),u_r\right).
\end{align*}
By Definition~\ref{def:notation} we deduce $\hat{\rho} \geq \bar{\rho}(u_\ell)$ and $\hat{q} = Q\left(u_\ell,u_r\right) = \check{q}$. 
\item\label{itemmm}
The states $\hat{u} = \hat{u}(u_\ell,u_r)$ and $\check{u} = \check{u}(u_\ell,u_r)$ are well defined by assuming $Q(u_\ell,u_r)\in [0,\bar q(u_\ell)]$.
However, it is easy to check that the stricter condition $Q(u_\ell,u_r) \leq \overline{Q}(u_\ell)$ required in  Definition \ref{d:RSV} is needed in order that $\xi\mapsto\rsp[u_\ell,\hat{u}](\xi) \in \mathsf{FL}_1^{u_\ell}$ represents a single wave with {\em negative} ($\leq0$) speed, and then let \eqref{e:rsc} make sense for $\xi<0$.

Analogously, condition $Q(u_\ell,u_r) \geq 0$ ensures that $\xi\mapsto\rsp[\check{u},u_r](\xi) \in \mathsf{FL}_2^{\check{u}}$ represents a single wave with {\em positive} ($\geq0$) speed (so that \eqref{e:rsc} makes sense for $\xi\geq0$) and it is needed in order that $q(\rsc[u_\ell,u_r](0)) = q(\rsp[u_\ell,\hat{u}](0)) = q(\rsp[\check{u},u_r](0)) \geq 0$.
\end{enumerate}
\end{rmrk}

We now discuss the occurrence of subsonic states for the solver $\rsc$.

\begin{prpstn}\label{p:subsonic} Let $(u_\ell,u_r)\in\mathsf{D}$ and $u_{\rm c} \doteq \rsc[u_\ell,u_r]$. Then:

\begin{enumerate}[label={{(\roman*)}},leftmargin=*,nolistsep]\setlength{\itemsep}{0cm}\setlength\itemsep{0em}
\item The restriction to $\xi<0$ of $u_{\rm c}$ attains supersonic values if and only if $u_\ell$ is supersonic.
\item The restriction to $\xi\geq0$ of $u_{\rm c}$ may attain supersonic values even if neither $u_\ell$ nor $u_r$ are supersonic.
\end{enumerate}
\end{prpstn}

\begin{proof}
First, we prove {\em (i)}. If $u_\ell$ is subsonic, i.e.\ $|v_\ell| \leq a$, then $\overline{Q}(u_\ell) = \bar{q}(u_\ell)$ and $\bar{u}(u_\ell)$ is a sonic state by \cite[Lemma~2.6]{CR1}. Therefore for any $\xi<0$
\[
u_{\rm c}(\xi) = \rsp\left[u_\ell,\hat{u}\left(Q(u_\ell,u_r),u_\ell\right)\right](\xi) \in \left\{u \in \mathsf{FL}_1^{u_\ell} : |v|\leq a\right\},
\] 
because by assumption $Q(u_\ell,u_r) \in [ 0 , \overline{Q}(u_\ell) ]$.
The converse is trivial: if $u_\ell$ is supersonic, then $u_{\rm c}$ attains a supersonic value at least at $u_\ell$.

About {\em (ii)}, it is sufficient to choose $Q(u_\ell,u_r)$ sufficiently large in order that $\check{u}\left(Q(u_\ell,u_r),u_r\right)$ is supersonic.
\end{proof}

We can now give the definition of a {\em coherent} coupling Riemann solver.

\begin{dfntn}
A c-Riemann solver $\rsc \colon \mathsf{D} \to \BV(\R;\Omega)$ is {\em coherent} at $(u_\ell,u_r)\in\mathsf{D}$ if the traces $u_{\rm c}^\pm\doteq\rsc[u_\ell,u_r](0^\pm)$ satisfy
\begin{equation}\label{chv}
\rsc[u^-,u^+](\xi) = 
\begin{cases}
u_{\rm c}^-&\hbox{ if } \xi<0,\\
u_{\rm c}^+&\hbox{ if } \xi\geq0.
\end{cases}
\end{equation}
The \emph{coherence domain} $\mathsf{CH}$ of $\rsc$ is the set of all pairs $(u_\ell,u_r)\in\mathsf{D}$ where $\rsc$ is coherent. The set $\mathsf{CH}^{\scriptscriptstyle\complement}\doteq\mathsf{D} \setminus \mathsf{CH}$ is the \emph{incoherence domain}.
\end{dfntn}

A c-Riemann solver $\rsc$ is coherent at an initial datum $(u_\ell, u_r) \in \mathsf{D}$ if the ordered pair of the traces of the corresponding solution $\left(\rsc[u_\ell, u_r](0^-), \rsc[u_\ell, u_r](0^+)\right)$ is, in a sense, a fixed point of $\rsc$. Hence, coherence may be thought as a stability property. On the contrary, the incoherence of a c-Riemann solver is understood as modeling the chattering of a valve and may yield analytical and numerical instabilities, see for instance the central column in \figurename~\ref{f:comparison:RSh_RSv}.

The Riemann solver $\rsp$ is coherent in $\mathsf{D}$ \cite[Proposition~2.5]{CFFR1}. On the contrary, coherence may fail for $\rsc$ because of the presence of a valve.
Indeed this is the case for the c-Riemann solver $\rsv$ introduced in \cite{CR1} and that we are going to briefly recall.
Fix $q_*>0$, then $\rsv$ corresponds to the valve that keeps the flow at $x=0$ equal to $q_*$ if possible, otherwise it closes.
This motivates the way $Q=Q_{\rm v}$ is defined in \eqref{e:Blackfield} below.

\begin{dfntn}
\label{d:Blackfield}
We denote by $\rsv$ the c-Riemann solver corresponding to 
\begin{equation}\label{e:Blackfield}
Q_{\rm v}(u_\ell) \doteq 
\begin{cases}
q_*&\hbox{ if } \overline{Q}(u_\ell)\geq q_*,
\\
0&\hbox{ if }\overline{Q}(u_\ell)<q_*.
\end{cases}
\end{equation}
\end{dfntn}
Notice that $Q_{\rm v}$ in \eqref{e:Blackfield} only depends on $u_\ell$ (and $q_*$, but for the moment we keep it fixed) and not on $u_r$. We denote $u_{\rm v} \doteq \rs[u_\ell,u_r]$ and $u_{\rm v}^\pm \doteq u_{\rm v}(0^\pm)$.
The function $Q_{\rm v}$ is discontinuous along some curve in $\Omega$; we explicitly find such a curve in the following Lemma~\ref{l:closedvalve}. 

We denote
\begin{align}\label{e:usa}
&u_*^a \doteq (\rho_*^a,q_*) \doteq (q_*/a,q_*),&
&u_*^0 \doteq (\rho_*^0,0) \doteq (e\,q_*/a,0),
\end{align}
see \figurename~\ref{f:closed}. Notice that $u_*^a$ is the intersection of the line $\{ u \in \Omega : q=q_*\}$ with the sonic line $\{ u \in \Omega : v = a\}$. 
Moreover, $u_*^0$ is the unique intersection of the curve $\mathsf{BL}_1^{u_*^a}$ with the line $\{ u \in \Omega : q=0\}$.

\begin{figure}[!htb]\centering
\begin{tikzpicture}[every node/.style={anchor=south west,inner sep=0pt},x=1mm, y=.5mm]
\node at (4,4) {\includegraphics[width=50mm]{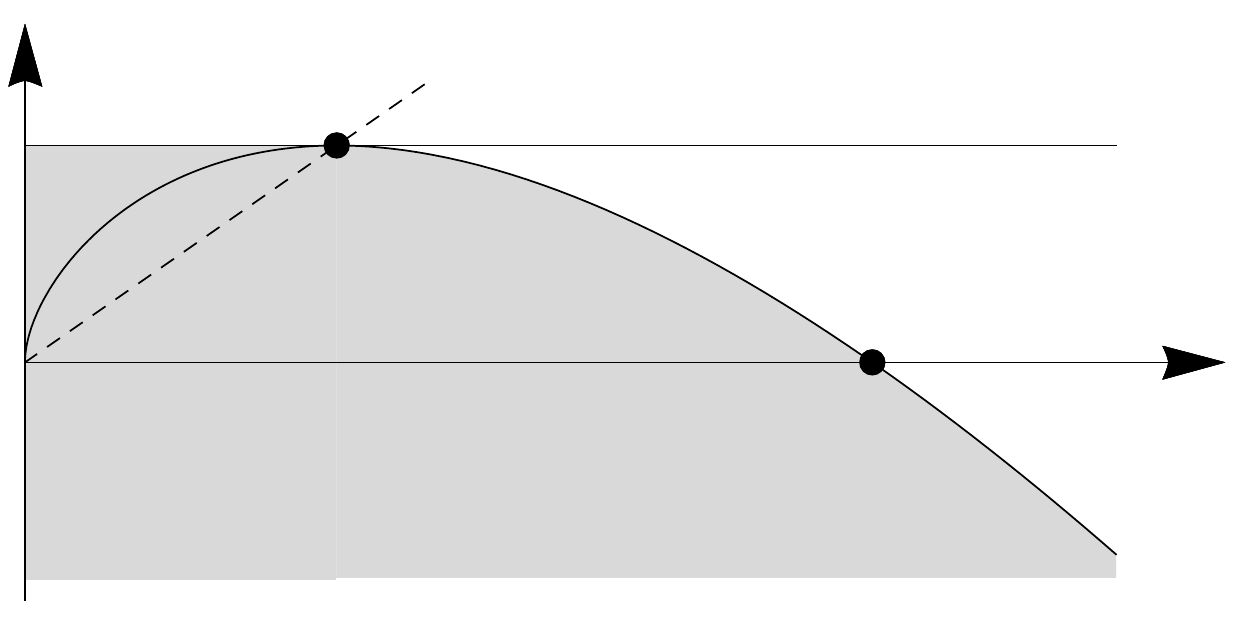}};
\node at (52,16) {\strut $\rho$};
\node at (0,48) {\strut $q$};
\node at (22,48) {\strut $q=a\,\rho$};
\node at (0,38) {\strut $q_*$};
\node at (50,5) {\strut $\mathsf{BL}_1^{u_*^a}$};
\node at (14,43) {\strut $u_*^a$};
\node at (39,25) {\strut $u_*^0$};
\node at (22,15) {\strut $\mathsf{C}_\ell$};
\end{tikzpicture}
\caption{The shaded region represents the set $\mathsf{C}_\ell$ of left states $u_\ell$ such that the valve corresponding to \eqref{e:Blackfield} is closed.}
\label{f:closed}
\end{figure}
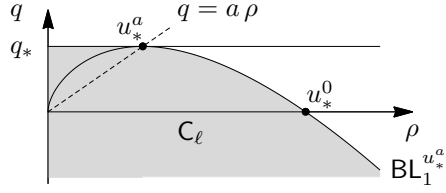

The following lemma characterizes the states for which the valve is closed; see \figurename~\ref{f:closed}. The function $Q_{\rm v}$ is then discontinuous along the upper boundary of the set $\mathsf{C}_\ell$.

\begin{lmm}\label{l:closedvalve}
The valve is closed if and only if one of the following equivalent conditions is satisfied:
\begin{enumerate}[label={{(\roman*)}},leftmargin=*,nolistsep]\setlength{\itemsep}{0cm}\setlength\itemsep{0em}

\item
$Q_{\rm v}(u_\ell)=0$;

\item\label{ii}
$\overline{Q}(u_\ell)<q_*$;

\item
$u_\ell$ belongs to the set
\begin{equation}
\mathsf{C}_\ell \doteq 
\bigl( (0,\rho_*^a] \times (-\infty,q_*) \bigr)
\cup
\bigl\{u \in \Omega : 
\rho > \rho_*^a,\ q < \mathcal{R}_1^{u_*^a}(\rho) \bigr\}.
\label{e:CL}
\end{equation}
\end{enumerate}
\end{lmm}

We now introduce the states $u_*^{\sup}$ and $u_*^{\rm sub}$, see \figurename~\ref{f:CH2} on the left. Notice that $\mathsf{BL}_1^{u_*^0}=\mathsf{S}_1^{u_*^0}$ if $q\geq0$.
\begin{figure}[!htb]\centering
\begin{tikzpicture}[every node/.style={anchor=south west,inner sep=0pt},x=1mm, y=.5mm]
\node at (4,4) {\includegraphics[width=60mm]{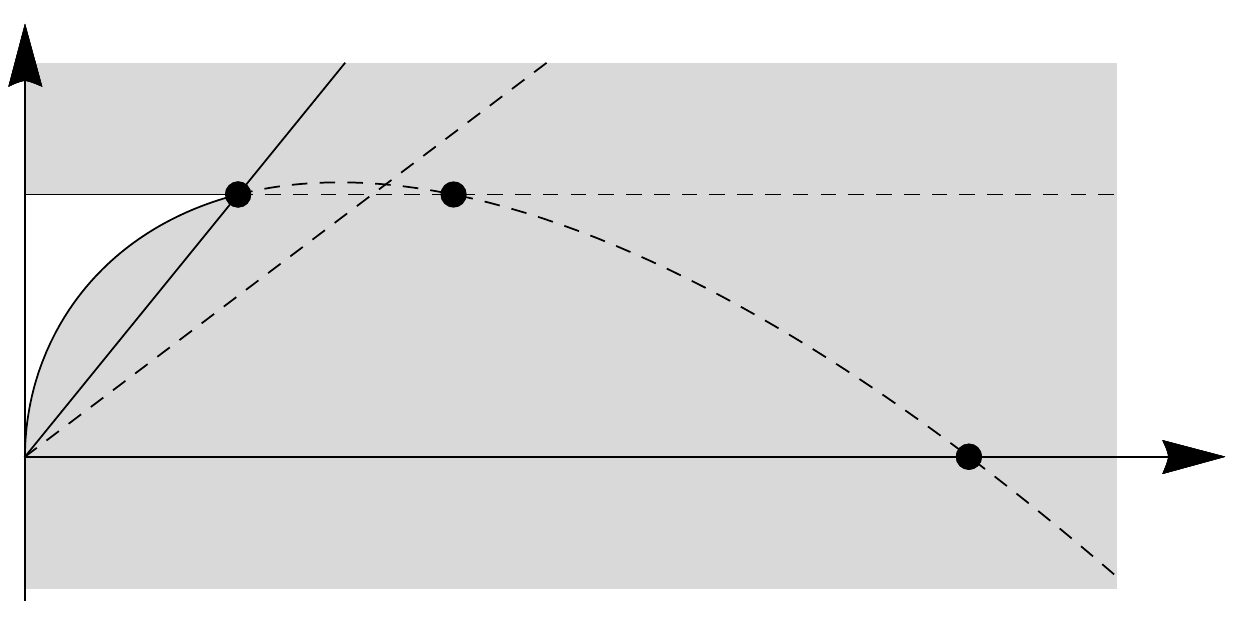}};
\node at (62,10) {\strut $\rho$};
\node at (0,58) {\strut $q$};
\node at (0,44) {\strut $q_*$};
\node at (51,20) {\strut $u_*^0$};
\node at (41,33) {\strut $\mathsf{BL}_1^{u_*^0}$};
\node at (10,46) {\strut $u_*^{\sup}$};
\node at (27,46) {\strut $u_*^{\rm sub}$};
\node at (14,58) {\strut $v=v_*^{\sup}$};
\node at (31,58) {\strut $v=a$};
\end{tikzpicture}
\qquad
\begin{tikzpicture}[every node/.style={anchor=south west,inner sep=0pt},x=1mm, y=.5mm]
\node at (4,4) {\includegraphics[width=60mm]{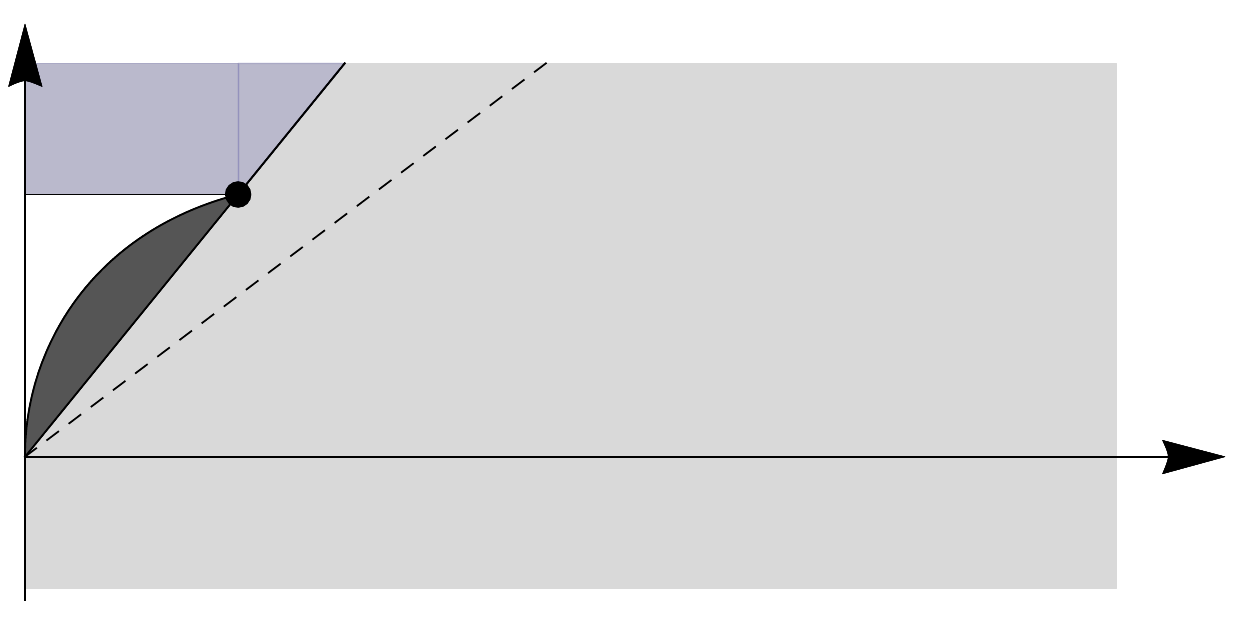}};
\node at (62,10) {\strut $\rho$};
\node at (0,58) {\strut $q$};
\node at (0,44) {\strut $q_*$};
\node at (6,23) {\strut \rotatebox{60}{\color{white}$\mathsf{CH}_{\ell,2}$}};
\node at (35,32) {\strut $\mathsf{CH}_{\ell,1}$};
\node at (7,48) {\strut $\mathsf{CH}_{\ell,3}$};
\node at (14,58) {\strut $v=v_*^{\sup}$};
\node at (31,58) {\strut $v=a$};
\draw[-latex] (3,38) node[left] {$\mathsf{CH}_\ell^{\scriptscriptstyle\complement}$} -- (7,40);
\end{tikzpicture}
\caption{Left: the shaded region represents the coherence domain $\mathsf{CH}_\ell$, the white region the incoherence domain $\mathsf{CH}^{\scriptscriptstyle\complement}_\ell$. Right: the decomposition of $\mathsf{CH}_\ell$ into the subsets $\mathsf{CH}_{\ell,1}$, $\mathsf{CH}_{\ell,2}$, $\mathsf{CH}_{\ell,3}$ given in \eqref{e:CH}.}
\label{f:CH2}
\end{figure}
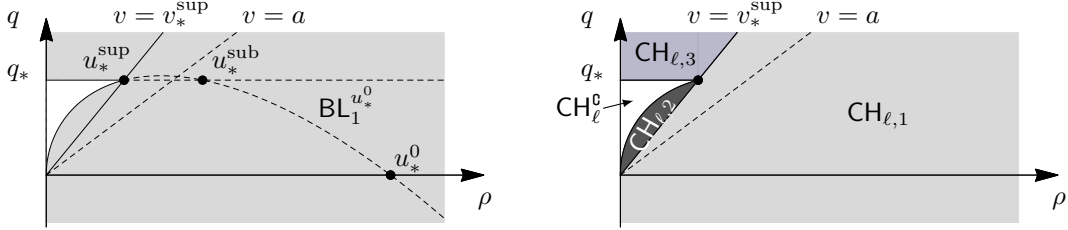
It is easy to see that the curve $\mathsf{BL}_1^{u_*^0}$ intersects the line $q=q_*$ at the two points 
$u_*^{\sup} \doteq (\rho_*^{\sup},q_*)$ and $u_*^{\rm sub} \doteq (\rho_*^{\rm sub},q_*)$.
The state $u_*^{\sup}$ is supersonic, the state $u_*^{\rm sub}$ is subsonic with constant speeds
\begin{align}\label{e:DevinTownsend}
&v_*^{\sup} \approx 1.63 \cdot a,&
&v_*^{\rm sub} \approx 0.81 \cdot a.
\end{align}

The next theorem characterizes the incoherence domain $\mathsf{CH}^{\scriptscriptstyle\complement}$ of $\rsv$.
Since $Q_{\rm v}$ only depends on the upstream states, it is clear that 
\[
\mathsf{CH}=\mathsf{CH}_\ell\times \Omega, \qquad \mathsf{CH}^{\scriptscriptstyle\complement}=\mathsf{CH}_\ell^{\scriptscriptstyle\complement}\times \Omega,
\] 
where both $\mathsf{CH}_\ell \subseteq \Omega$ and $ \mathsf{CH}_\ell^{\scriptscriptstyle\complement} \doteq \Omega \setminus \mathsf{CH}_\ell$ only contain left states $u_\ell$. 

\begin{thrm}[Incoherence]
The incoherence domain of $\rsv$ is $\mathsf{CH}^{\scriptscriptstyle\complement} =\mathsf{CH}_\ell^{\scriptscriptstyle\complement}\times \Omega$, where 
\begin{align}\label{e:CHellC}
\mathsf{CH}_{\ell}^{\scriptscriptstyle\complement} &= 
\left\{
u \in \Omega : 
v>v_*^{\sup},\ \mathcal{S}_1^{u_*^0}(\rho)\leq q < q_*\right\}.
\end{align}
\end{thrm}
\noindent
We refer to \figurename~\ref{f:CH2} on the left for a representation of $\mathsf{CH}_{\ell}^{\scriptscriptstyle\complement}$. 
By \eqref{e:CHellC} we deduce $\mathsf{CH}_\ell = \mathsf{CH}_{\ell,1}\cup\mathsf{CH}_{\ell,2}\cup\mathsf{CH}_{\ell,3}$ where, see \figurename~\ref{f:CH2} on the right,
\begin{align}\label{e:CH}
\mathsf{CH}_{\ell,1} &\doteq \bigl\{u\in \Omega : v\leq v_*^{\sup}\bigr\},
&
\mathsf{CH}_{\ell,2} & \doteq \bigl\{u\in \Omega : v> v_*^{\sup},\ q < \mathcal{S}_1^{u_*^0}(\rho) \bigr\},
&
\mathsf{CH}_{\ell,3} & \doteq \bigl\{u\in \Omega : v> v_*^{\sup},\ q\geq q_*\bigr\}.
\end{align}

Notice that $\mathsf{CH}_{\ell,1}$ is independent of $q_*$ by the definition \eqref{e:DevinTownsend} of $v_*^{\sup}$. We now show that if $\rsv$ is not coherent at $(u_\ell,u_r)$ then the valve is closed.

\begin{crllr}\label{c:SnarkyPuppy}
We have $\mathsf{CH}_\ell^{\scriptscriptstyle\complement} \subset \mathsf{C}_\ell$ and $\rsv[\mathsf{CH}_\ell^{\scriptscriptstyle\complement},\Omega](0^-) \subseteq \Omega \setminus \mathsf{C}_\ell \subset \mathsf{CH}_\ell$.
\end{crllr}

In other words, Corollary~\ref{c:SnarkyPuppy} means that if $(u_\ell,u_r) \in \mathsf{CH}^{\scriptscriptstyle\complement}= \mathsf{CH}_\ell^{\scriptscriptstyle\complement}\times\Omega$, then in the solution $u_{\rm v} = \rsv[u_\ell,u_r]$ the valve is closed, while in the solution $\rsv[u_{\rm v}^-,u_{\rm v}^+]$ the valve is open and then $(u_{\rm v}^-,u_{\rm v}^+) \in \mathsf{CH} = \mathsf{CH}_\ell \times \Omega$; as a consequence by \eqref{chv} we have $\rsv[u_{\rm v}^-,u_{\rm v}^+](0^\pm) = u_{\rm v}^\pm$.

\section{The coherent c-Riemann solver \texorpdfstring{$\rsh$}{}}
\label{s:cova}

A drawback of incoherence is that it leads to instabilities. For instance, 
numerical solutions obtained by exploiting $\rsv$ at each time step may substantially differ from the {\em exact} solution $u_{\rm v}$ in correspondence of incoherent initial data, see for instance the first two columns in \figurename~\ref{f:comparison:RSh_RSv}. This difficulty motivates the design of a new valve, which reproduces the behavior of the valve modeled in Definition \ref{d:Blackfield} for coherent initial data but that gives rise to a coherent solver.

We introduce such a valve in Definition~\ref{d:rsv0} through its Riemann solver $\rsh$ (\lq\lq h\rq\rq\ for co{\em h}erent). The solver $\rsh$, roughly speaking, is uniquely determined by the following conditions:

\begin{enumerate}[label={{(\Roman*)}},leftmargin=*]

\item $\rsh$ is coherent in the whole of $\mathsf{D} \doteq \Omega\times\Omega$;

\item $\rsh$ coincides with $\rsv$ in the coherence region $\mathsf{CH}$ of $\rsv$;

\item if $(u_\ell,u_r) \in \mathsf{CH}^{\scriptscriptstyle\complement}$, then $\rsh[u_\ell,u_r]$ maximizes the flow across $x=0$, that is $q(\rsh[u_\ell,u_r](0)) \geq q(\rsc[u_\ell,u_r](0))$ for any c-Riemann solver $\rsc$.
\end{enumerate}
As we commented in the Introduction, condition (III) resembles an entropy condition. It has been already exploited in the framework of gas networks, see for instance \cite[(28)]{Banda-Herty-Klar2}, \cite[(15a)]{Banda-Herty-Klar1}.

Because of (II), the issue is then how to define $\rsh$ in $\mathsf{CH}^{\scriptscriptstyle\complement}$. A hint comes from \cite[\S6]{CR2}, where a valve with a reaction time and based on $\rsv$ is considered. For $(u_\ell,u_r) \in \mathsf{CH}^{\scriptscriptstyle\complement}$, a solution is constructed there by applying a front-tracking algorithm. Rather surprisingly, the reaction time leads to the periodic appearance of a flow $q_\ell$ at $x=0$, even if $q_\ell$ differs from both $0$ and $q_*$.

We are then led to prescribe a new value $Q=Q_{\rm h}$ of the flow at $x=0$, which equals $q_\ell$ in the incoherent region $\mathsf{CH}^{\scriptscriptstyle\complement} = \mathsf{CH}_\ell^{\scriptscriptstyle\complement} \times \Omega$ of $\rsv$, see \eqref{e:CHellC}, and coincides with $Q_{\rm v}$ in $\mathsf{CH}$, as stated in the following definition.

\begin{dfntn}
\label{d:rsv0}
We denote by $\rsh$ the c-Riemann solver corresponding to 
\begin{equation}\label{e:rsv0}
Q_{\rm h}(u_\ell) \doteq 
\begin{cases}
q_*&\hbox{ if } \overline{Q}(u_\ell)\geq q_*,
\\
0&\hbox{ if }\overline{Q}(u_\ell)<q_* \hbox{ and } u_\ell \in \mathsf{CH}_\ell,
\\
q_\ell&\hbox{ if }u_\ell \in \mathsf{CH}_\ell^{\scriptscriptstyle\complement},
\end{cases}
\end{equation}
where $\overline{Q}$ is defined as in \eqref{e:AliceinChains} and $\mathsf{CH}_\ell^{\scriptscriptstyle\complement} \times \Omega$ is the incoherence domain of $\rsv$, see \eqref{e:CHellC}.
\end{dfntn}

\begin{rmrk}
Observe that if $\overline{Q}(u_\ell)\geq q_*$ then $u_\ell \in \mathsf{CH}_\ell$, or equivalently, if $u_\ell \in \mathsf{CH}_\ell^{\scriptscriptstyle\complement}$ then $\overline{Q}(u_\ell)<q_*$.
\end{rmrk}

Numerical simulations based on $\rsh$ reproduce the {\em same} effect on the gas flow of $\rsv$ after the chattering, at least in the case considered in \figurename~\ref{f:comparison:RSh_RSv}, see the last two columns. We currently miss of a rigorous proof of this fact. 

\begin{figure}[!htb]\centering
\begin{subfigure}[(t)]{50mm}
\begin{tikzpicture}[every node/.style={anchor=south west,inner sep=0pt},x=1mm, y=1mm]
\node at (0,0) {\includegraphics[width=50mm,trim=72.5mm 195.5mm 73.5mm 46mm, clip]{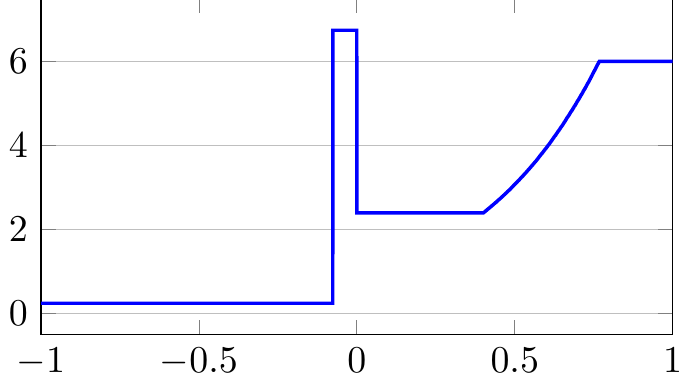}};
\end{tikzpicture}
\caption*{$x\mapsto\rho_\Delta(0.2,x)$}
\end{subfigure}
\hspace{5mm}
\begin{subfigure}[(t)]{50mm}
\begin{tikzpicture}[every node/.style={anchor=south west,inner sep=0pt},x=1mm, y=1mm]
\node at (0,0) {\includegraphics[width=50mm,trim=72.5mm 195.5mm 73.5mm 46mm, clip]{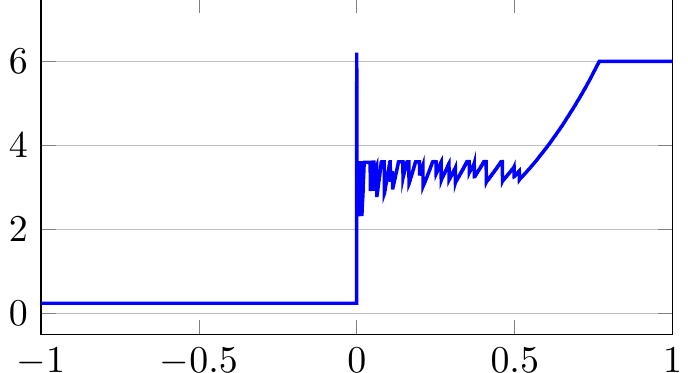}};
\end{tikzpicture}
\caption*{$x\mapsto\rho_\Delta(0.2,x)$}
\end{subfigure}
\hspace{5mm}
\begin{subfigure}[(t)]{50mm}
\begin{tikzpicture}[every node/.style={anchor=south west,inner sep=0pt},x=1mm, y=1mm]
\node at (0,0) {\includegraphics[width=50mm,trim=72.5mm 195.5mm 73.5mm 46mm, clip]{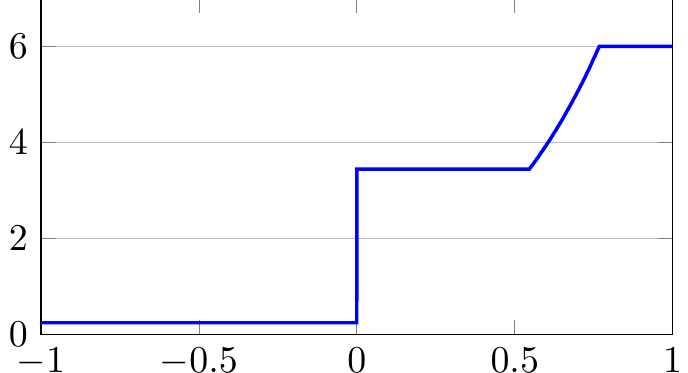}};
\end{tikzpicture}
\caption*{$x\mapsto\rho_\Delta(0.2,x)$}
\end{subfigure}
\\[10pt]
\begin{subfigure}[(t)]{50mm}
\begin{tikzpicture}[every node/.style={anchor=south west,inner sep=0pt},x=1mm, y=1mm]
\node at (0,0) {\includegraphics[width=50mm,trim=72.5mm 195.5mm 73.5mm 46mm, clip]{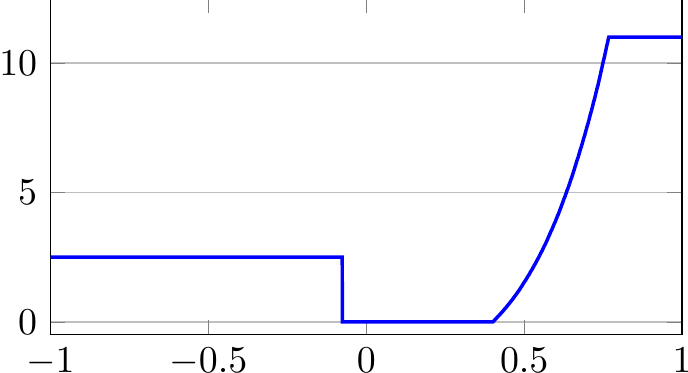}};
\end{tikzpicture}
\caption*{$x\mapsto q_\Delta(0.2,x)$}
\end{subfigure}
\hspace{5mm}
\begin{subfigure}[(t)]{50mm}
\begin{tikzpicture}[every node/.style={anchor=south west,inner sep=0pt},x=1mm, y=1mm]
\node at (0,0) {\includegraphics[width=50mm,trim=72.5mm 195.5mm 73.5mm 46mm, clip]{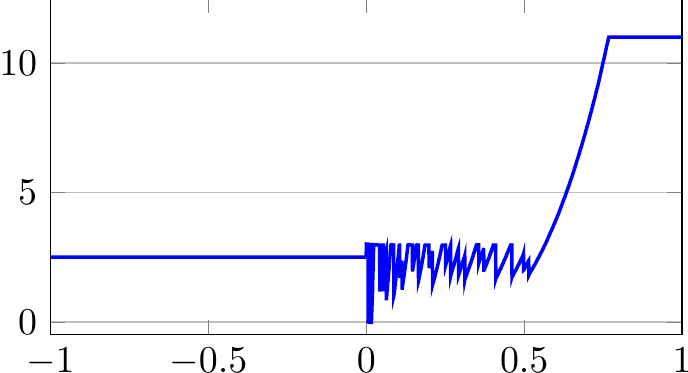}};
\end{tikzpicture}
\caption*{$x\mapsto q_\Delta(0.2,x)$}
\end{subfigure}
\hspace{5mm}
\begin{subfigure}[(t)]{50mm}
\begin{tikzpicture}[every node/.style={anchor=south west,inner sep=0pt},x=1mm, y=1mm]
\node at (0,0) {\includegraphics[width=50mm,trim=72.5mm 195.5mm 73.5mm 46mm, clip]{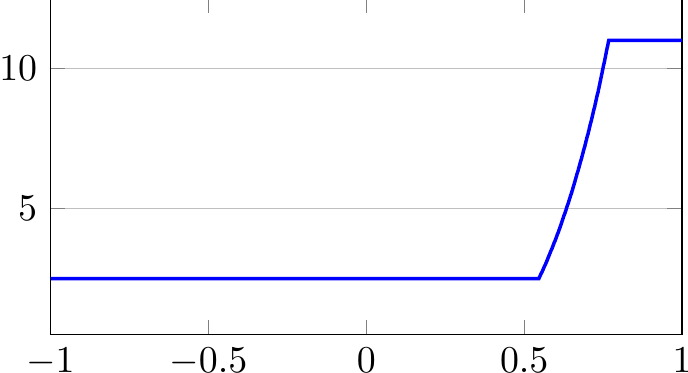}};
\end{tikzpicture}
\caption*{$x\mapsto q_\Delta(0.2,x)$}
\end{subfigure}
\caption{Different numerical simulations for $u(t,x) \doteq \rsv[u_\ell,u_r](x/t)$. Left column: we computed $Q_{\rm v}(u_\ell)$ and kept it as flow through the valve for any time. Center column: we applied $\rsv$ at $x=0$ at each time step, namely, we updated the flow through the valve according to the left traces of the solution computed at each time step. Right column: $u(t,x) \doteq \rsh[u_\ell,u_r](x/t)$.
Here $\rho_\ell=0.25$, $q_\ell=2.5$, $\rho_r=6$, $q_r=11$, $a=2$ and $q_*=3$, so that $u_\ell \in \mathsf{CH}_\ell^{\scriptscriptstyle\complement}$.}
\label{f:comparison:RSh_RSv}
\end{figure}

By Corollary~\ref{c:SnarkyPuppy} we have $\mathsf{CH}_\ell^{\scriptscriptstyle\complement} \subset \mathsf{C}_\ell$, hence $\mathsf{C}_\ell^{\scriptscriptstyle\complement} = \mathsf{CH}_\ell \setminus \mathsf{C}_\ell$.
This, and Lemma~\ref{l:closedvalve} \ref{ii}, implies
\begin{align}
\label{e:KingCrimson}
&\left\{ u_\ell \in \Omega : \overline{Q}(u_\ell) < q_* \right\} = \mathsf{C}_\ell,&
&\left\{ u_\ell \in \Omega : \overline{Q}(u_\ell)\geq q_* \right\} = \mathsf{CH}_\ell \setminus \mathsf{C}_\ell,
\end{align}
and therefore
\begin{equation}\label{e:rsv00}
Q_{\rm h}(u_\ell) =
\begin{cases}
q_*&\hbox{ if } u_\ell \in \mathsf{CH}_\ell \setminus \mathsf{C}_\ell,
\\
0&\hbox{ if } u_\ell \in \mathsf{CH}_\ell \cap \mathsf{C}_\ell,
\\
q_\ell&\hbox{ if } u_\ell \in \mathsf{CH}_\ell^{\scriptscriptstyle\complement}.
\end{cases}
\end{equation}

Now, we collect the main properties of $\rsh$. About (III), we notice that it is a consequence of the explicit definition \eqref{e:rsv0} and not, as in \cite{Banda-Herty-Klar2, Banda-Herty-Klar1}, an implicit consequence of a maximization process.

\begin{prpstn}
\label{p:argh}
For any $u_\ell,u_r\in\Omega$ the following holds:
\begin{enumerate}[label={{(\roman*)}},leftmargin=*,nolistsep]\setlength{\itemsep}{0cm}\setlength\itemsep{0em}

\item if $(u_\ell,u_r) \in \mathsf{CH}_\ell\times\Omega$, then
\begin{equation}
\label{e:prop1}
\rsh[u_\ell,u_r] \,\equiv\, \rsv[u_\ell,u_r];
\end{equation}

\item if $(u_\ell,u_r) \in \mathsf{CH}_\ell^{\scriptscriptstyle\complement}\times\Omega$, then
\begin{align}
\label{e:prop2}
&\rsh[u_\ell,u_r](\xi) =
\begin{cases}
u_\ell&\hbox{ if }\xi <0,\\
\rsp[\check{u}(q_\ell,u_r),u_r](\xi)&\hbox{ if }\xi\geq0,
\end{cases}
\intertext{and for every c-Riemann solver $\rsc$ we have}
\label{e:prop3}
&q( \rsc[u_\ell,u_r] )(0) \leq q( \rsh[u_\ell,u_r] )(0) \in (0,q_*);
\end{align}

\item $\rsh$ is coherent in $\mathsf{D}$.
\end{enumerate}
\end{prpstn}

\begin{proof}
About {\em (i)}, formula \eqref{e:prop1} directly follows from Definition~\ref{d:rsv0} and by comparing \eqref{e:Blackfield} with \eqref{e:rsv0}.

We now prove {\em (ii)}. About \eqref{e:prop2}, if $(u_\ell,u_r) \in \mathsf{CH}_\ell^{\scriptscriptstyle\complement} \times \Omega$, then by \eqref{e:rsc} and \eqref{e:rsv0} we deduce that
\[
\rsh[u_\ell,u_r](\xi) =
\begin{cases}
\rsp\left[u_\ell,\hat{u}\left(q_\ell,u_\ell\right)\right](\xi) & \hbox{ if } \xi<0,
\\
\rsp[\check{u}(q_\ell,u_r),u_r](\xi)&\hbox{ if } \xi\geq0. 
\end{cases}
\]
Then, we have two possibilities: either $\hat{u}\left(q_\ell,u_\ell\right) = u_\ell$ and so $\rsp\left[u_\ell,\hat{u}\left(q_\ell,u_\ell\right)\right] \equiv u_\ell$, or else $\hat{u}\left(q_\ell,u_\ell\right)\neq u_\ell$ and so $\rsp\left[u_\ell,\hat{u}\left(q_\ell,u_\ell\right)\right]$ consists of constant states $u_\ell$ and $\hat{u}\left(q_\ell,u_\ell\right)$ separated by a stationary shock.
In both cases \eqref{e:prop2} immediately follows.
To prove \eqref{e:prop3}, we first recall \eqref{e:rsc}, \eqref{e:rsv00} and observe that if $(u_\ell,u_r) \in \mathsf{CH}_\ell^{\scriptscriptstyle\complement}\times\Omega$ then
\[q( \rsh[u_\ell,u_r] )(0) = Q_{\rm h}(u_\ell) = q_\ell.
\]
By the definition \eqref{e:CHellC} of $\mathsf{CH}_\ell^{\scriptscriptstyle\complement}$ we have that both $q_\ell \in (0, q_*)$ and $v_\ell>a$; thus by \eqref{e:AliceinChains} we have
\[
q_\ell=\overline{Q}(u_\ell) = \max_{u\in\Omega} q( \rsp[u_\ell,u] )(0) \in (0, q_*).
\]

At last, to prove {\em (iii)}, notice that from \eqref{e:prop2} we have for any $(u_\ell,u_r) \in \mathsf{CH}_\ell^{\scriptscriptstyle\complement}\times\Omega$ that
\[\left(\rsh[u_\ell,u_r](0^-) , \rsh[u_\ell,u_r](0^+)\right) 
= \left(u_\ell , \check{u}(q_\ell,u_r)\right)
\in \mathsf{CH}_\ell^{\scriptscriptstyle\complement}\times\Omega.
\]
Moreover $\check{u}(q_\ell,\check{u}(q_\ell,u_r)) = \check{u}(q_\ell,u_r)$ because $q_\ell>0$, and then $\rsh$ is coherent in $\mathsf{D}$.
\end{proof}

Proposition \ref{p:argh} proofs the properties of $\rsh$ listed at the beginning of this section; in particular, formula \eqref{e:prop3} is the maximization of the flow at $x=0$.

\section{Numerical approximation of a c-Riemann solver \texorpdfstring{$\rsc$}{}}
\label{s:Ulrich}
In this section we introduce the numerical scheme to be used in the following and show some simulations to show its reliability in dealing cases where the valve is involved.

\subsection{Description of the numerical scheme}

In this subsection, we describe the scheme used to approximate the solutions provided by a given c-Riemann solver $\rsc$. It is based on the Random Choice Method (RCM), which was
introduced in~\cite{Glimm_CPAM_1965} in order to prove the existence of solutions to systems of
non-linear hyperbolic conservation laws. It has then been adapted and used in~\cite{Chorin_JCP_1976} as a numerical scheme. 
We also quote~\cite{Toro_book_1997} and references therein, for the description of the method as a numerical scheme to be implemented.\\

\noindent Let $\Delta x$ and $\Delta t$ be the constant space and time steps, respectively. We introduce the points $x_{j+1/2} \doteq j \, \Delta x$, the cells $K_j \doteq [x_{j-1/2},x_{j+1/2})$
and the cell centers $x_j \doteq (j-1/2) \, \Delta x$ for $j\in\mathbb{Z}$. We denote by $j_c$ the index such that $x_{j_c+1/2}$ is the location of the valve. Define 
$N \doteq \lfloor T/\Delta t\rfloor$ and, for $n\in\Z\cap[0,N]$, introduce the time discretization $t^n \doteq n \, \Delta t$. We denote by $u_\Delta$ the approximate solution that we assume 
to be constant in each cell $K_j$:
\[u_\Delta(t,x) \doteq u_j^n\in\R^2,\quad (t,x)\in[t^n,t^{n+1})\times K_j.\]
Next, we denote by $\hat{u}_\Delta(Q(u_\ell,u_r),u_\ell)$ and $\check{u}_\Delta(Q(u_\ell,u_r),u_r)$ the numerical approximations of 
$\hat{u}(Q(u_\ell,u_r),u_\ell)$ and $\check{u}_\Delta(Q(u_\ell,u_r),u_r)$, respectively.\\

\noindent The main goal is now to compute $u_j^n$ for any $n\in\N\cap[0,N]$ and $j\in\mathbb{Z}$. We first define  
\[u_j^0 \doteq \dfrac{1}{\Delta x}\int_{K_j} u(0,x) \,\d x.\]
Now for a fixed $n\in\Z\cap[0,N]$, assume that $u_j^n$ is given and for any $j\in\mathbb{Z}$. We use the following procedure to compute $u_j^{n+1}$:
\begin{itemize}[leftmargin=*]
\item [$\bullet$] We pick up randomly or quasi-randomly a number $\theta^n\in[0,1]$. Here, as in Colella~\cite{Collela_SIAJSSC_1982} (see also~\cite{Toro_book_1997}), 
we consider the van der Corput random sequence $(\theta^n)$ defined by 
\[\theta^n \doteq \sum_{k=0}^m i_k \, 2^{-(k+1)},\]
where
\[n\doteq\sum_{k=0}^m i_k \, 2^k,\qquad i_k\in\{0,1\},\]
denotes the binary expansion of the integer $n$.
\item [$\bullet$] The updated solution is then computed as follows, for $j\notin\{j_c,j_{c+1}\}$,
\[u_j^{n+1} \doteq
\begin{cases}
\rsp\left[u_{j-1}^n,u_j^n\right]\left( \theta^n\Delta x/\Delta t \right)
&\hbox{ if }0\leq \theta^n\leq \tfrac{1}{2},\\[2pt]
\rsp\left[u_{j}^n,u_{j+1}^n\right]\left( (\theta^n-1)\Delta x/\Delta t \right)
&\hbox{ if }\tfrac{1}{2}\leq \theta^n\leq 1,\\
\end{cases}\]
and, for $j\in\{j_c,j_{c+1}\}$, 
\begin{align*}
u_{j_c}^{n+1}& \doteq
\begin{cases}
\rsp\left[u_{j_c-1}^n,u_{j_c}^n\right]\left( \theta^n\Delta x/\Delta t \right)
&\hbox{ if }0\leq \theta^n\leq \tfrac{1}{2},\\[2pt]
\rsp\left[u_{j_c}^n,\hat{u}_\Delta\left(Q(u_{j_c-1}^n,u_{j_c}^n),u_{j_c-1}^n\right)\right]\left( (\theta^n-1)\Delta x/\Delta t \right)
&\hbox{ if }\tfrac{1}{2}\leq \theta^n\leq 1,\\
\end{cases}
\\
u_{j_c+1}^{n+1}& \doteq
\begin{cases}
\rsp\left[\check{u}_\Delta\left(Q(u_{j_c}^n,u_{j_c+1}^n),u_{j_c+1}^n\right),u_{j_c+1}^n\right]\left( \theta^n\Delta x/\Delta t \right)
&\hbox{ if }0\leq \theta^n\leq \tfrac{1}{2},\\[2pt]
\rsp\left[u_{j_c+1}^n,u_{j_c+2}^n\right]\left( (\theta^n-1)\Delta x/\Delta t \right)
&\hbox{ if }\tfrac{1}{2}\leq \theta^n\leq 1.\\
\end{cases}
\end{align*}
\end{itemize}

\noindent Let us note that, as usual, the time steps are chosen with respect to the CFL condition, that is, 
\[\Delta t = \frac{C_{\rm cfl} \, \Delta x}{\underset{j\in\mathbb{Z}}\max\,\underset{i\in\{1,2\}}\max|\lambda_i(u_j^n)|},\]
where the CFL coefficient $C_{\rm cfl}$ satisfies $0\leq C_{\rm cfl}\leq \frac{1}{2}$. For all the simulations of this paper, we always take $C_{\rm cfl}=0.45$. 

\subsection{Numerical simulations}

\noindent We use the scheme to compute numerical solutions of some cases involving different configurations of the valve,
and we compare them with exact solutions when available. We define the following relative $\L1$-error
\[e_{\L1}^t(\Delta x) \doteq \dfrac{\|u_\Delta(t,\cdot)-u(t,\cdot)\|_{\L1(I)}}{\|u(t,\cdot)\|_{\L1(I)}},\]
where $I\subset\R$ is the computational domain. In the remaining part of this subsection, we take
\begin{align*}
&I=[-1,1],&
&a=2,&
&q_*=3&
&\text{and the final time }&
&T=0.2.
\end{align*}

\noindent In the first two examples, we consider the case where $u_\ell\in \mathsf{CH}_\ell$.

\begin{xmpl}
\label{ex:1}
We take $u_\ell=(6, 1)  \in \mathsf{CH}_\ell\setminus\mathsf{C}_\ell$ and $u_r=(1,-1)$.
This corresponds to the case when the valve realizes the flow $q^*$.
In \figurename~\ref{fig.errors.valve}~(a), we show the numerical convergence of the scheme and this result also shows that the order of convergence is approximately~$1$.
Moreover, we can see in Figures~\ref{fig.solutions.valve.open} that the numerical solution is in a good agreement with the exact one. 
\end{xmpl}

\begin{xmpl}
\label{ex:2}
We take now $u_\ell=(2,2) \in \mathsf{CH}_\ell\cap\mathsf{C}_\ell$ and $u_r=(3,4)$.
This corresponds to the case when the valve is closed.
As for the previous example, we can see in \figurename~\ref{fig.errors.valve}~(b), the numerical convergence of the scheme 
and that the order of convergence is also approximately~$1$.
Moreover, \figurename~\ref{fig.solutions.valve.closed} shows the good agreement between the numerical and the exact solutions. 
\end{xmpl}

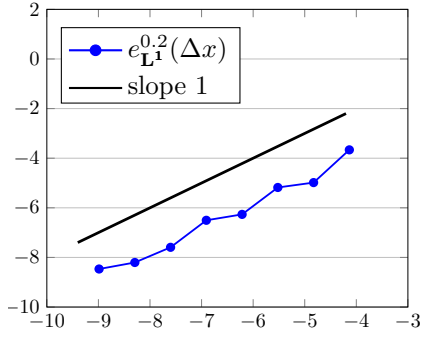
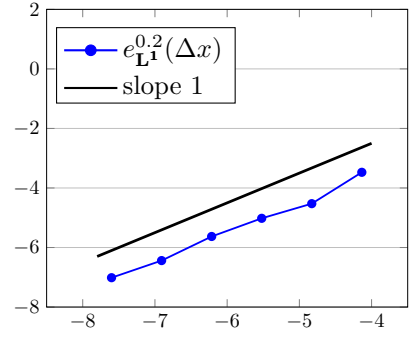
\begin{figure}[!htb]\centering
\begin{subfigure}[(t)]{54mm}
\begin{tikzpicture}[every node/.style={anchor=south west,inner sep=0pt},x=1mm, y=1mm]
\node at (0,0) {\includegraphics[width=54mm,trim=69mm 170mm 69mm 46mm, clip]{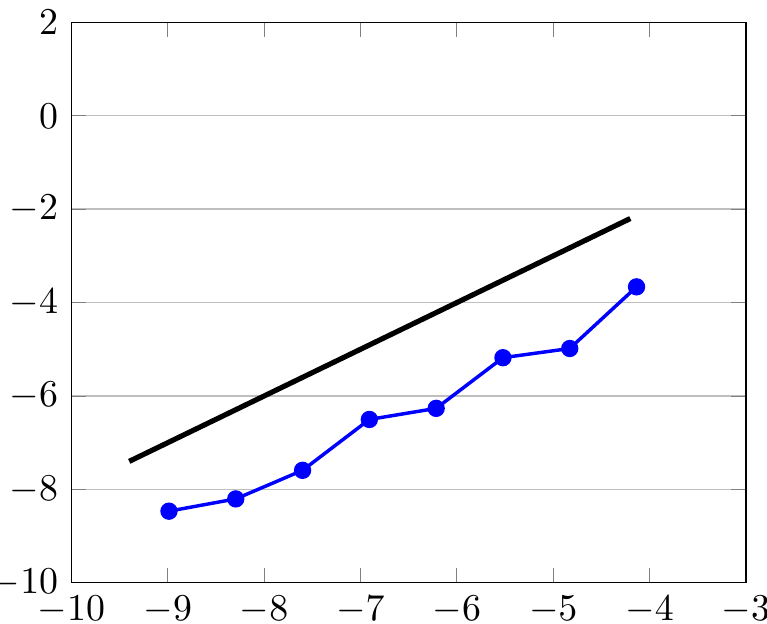}};
\draw[black, fill=white] (7,30) rectangle (30.5,40);
\draw[thick, blue] (8,37) -- (15,37) node[right, black] {\strut$\,e_{\L1}^{0.2}(\Delta x)$};
\draw[blue, fill=blue] (11.5,37) circle (2pt);
\draw[thick, black] (8,32) -- (15,32) node[right, black] {\strut\,slope $1$};
\end{tikzpicture}
\caption*{(a) Example~\ref{ex:1}}
\end{subfigure}
\hspace{30mm}
\begin{subfigure}[(t)]{54mm}
\begin{tikzpicture}[every node/.style={anchor=south west,inner sep=0pt},x=1mm, y=1mm]
\node at (0,0) {\includegraphics[width=54mm,trim=69mm 170mm 69mm 46mm, clip]{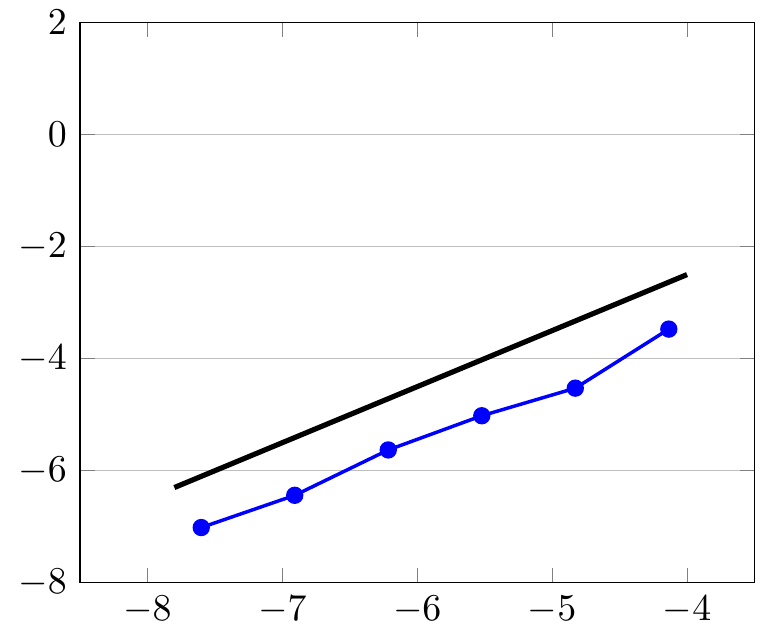}};
\draw[black, fill=white] (7,30) rectangle (30.5,40);
\draw[thick, blue] (8,37) -- (15,37) node[right, black] {\strut$\,e_{\L1}^{0.2}(\Delta x)$};
\draw[blue, fill=blue] (11.5,37) circle (2pt);
\draw[thick, black] (8,32) -- (15,32) node[right, black] {\strut\,slope $1$};
\end{tikzpicture}
\caption*{(b) Example~\ref{ex:2}}
\end{subfigure}
\caption{Relative $\L1$-errors in log/log scale for Examples~\ref{ex:1} and~\ref{ex:2}.}
\label{fig.errors.valve}
\end{figure}

\begin{figure}[!htb]\centering
\begin{tikzpicture}[every node/.style={anchor=south west,inner sep=0pt},x=1mm, y=1mm]
\node at (0,0) {\includegraphics[height=50mm,trim=62.5mm 156.5mm 63mm 47mm, clip]{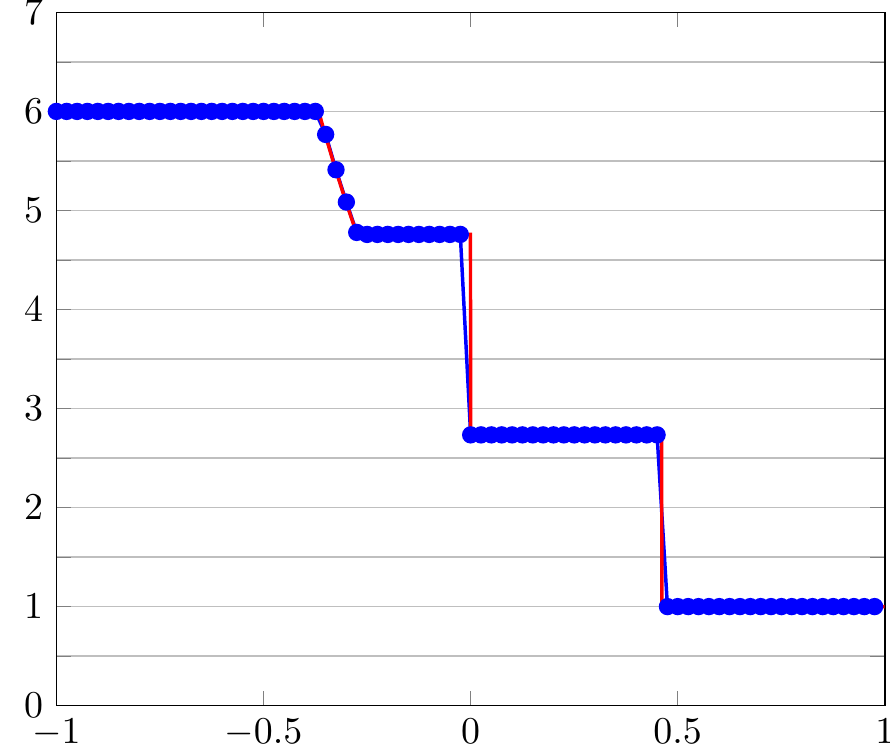}};

\begin{scope}[shift={(32.25,37)}]
\draw[black, fill=white] (0,0) rectangle (26,10);
\draw[thick, blue] (1,7) -- (8,7) node[right, black] {\strut$\,\rho_{\Delta}(0.2,x)$};
\draw[blue, fill=blue] (4.5,7) circle (2pt);
\draw[thick, red] (1,2) -- (8,2) node[right, black] {\strut$\,\rho(0.2,x)$};
\end{scope}
\end{tikzpicture}
\hspace{30mm}
\begin{tikzpicture}[every node/.style={anchor=south west,inner sep=0pt},x=1mm, y=1mm]
\node at (0,0) {\includegraphics[height=50mm,trim=62.5mm 156.5mm 63mm 47mm, clip]{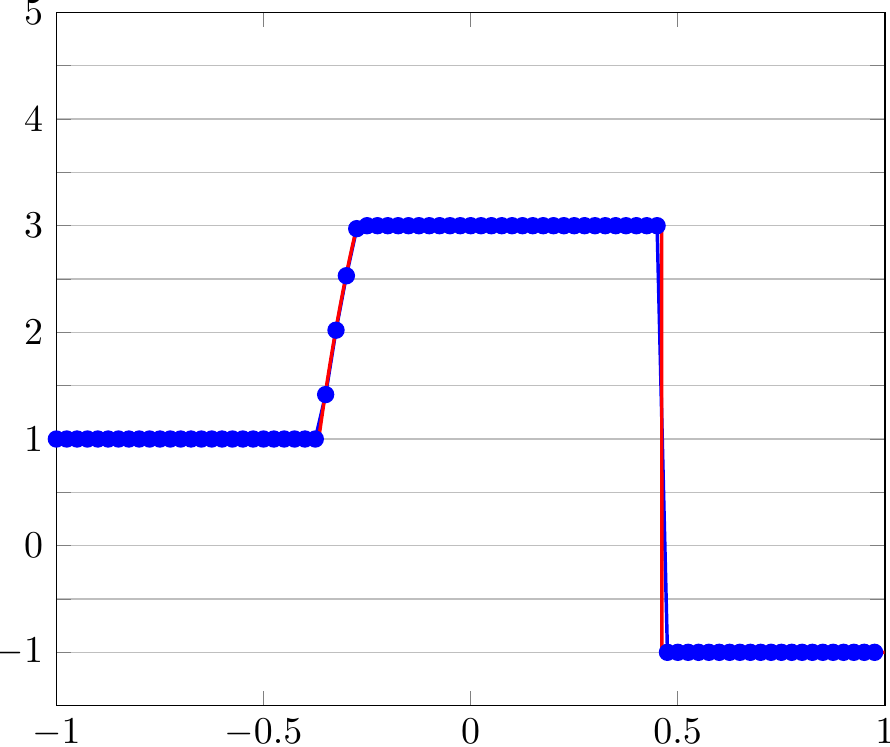}};

\begin{scope}[shift={(5,37)}]
\draw[black, fill=white] (0,0) rectangle (26,10);
\draw[thick, blue] (1,7) -- (8,7) node[right, black] {\strut$\,q_{\Delta}(0.2,x)$};
\draw[blue, fill=blue] (4.5,7) circle (2pt);
\draw[thick, red] (1,2) -- (8,2) node[right, black] {\strut$\,q(0.2,x)$};
\end{scope}
\end{tikzpicture}
\caption{Exact and numerical solutions for Example~\ref{ex:1} with $\Delta x=5\times10^{-4}$.}
\label{fig.solutions.valve.open}
\end{figure}

\begin{figure}[!htb]\centering
\begin{tikzpicture}[every node/.style={anchor=south west,inner sep=0pt},x=1mm, y=1mm]
\node at (0,0) {\includegraphics[height=50mm,trim=62.5mm 156.5mm 63mm 47mm, clip]{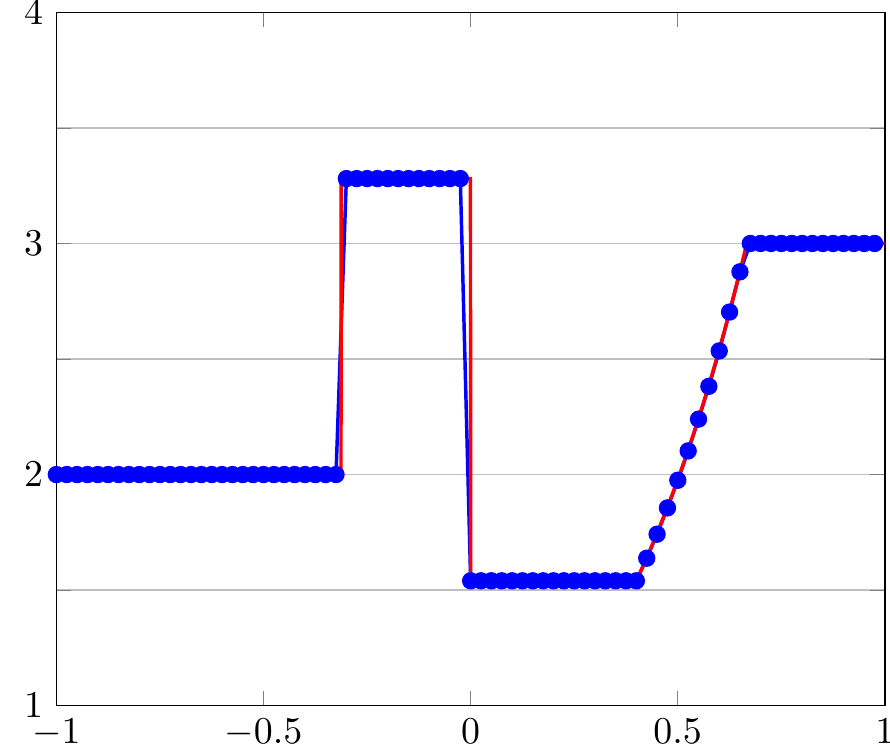}};

\begin{scope}[shift={(32.25,37)}]
\draw[black, fill=white] (0,0) rectangle (26,10);
\draw[thick, blue] (1,7) -- (8,7) node[right, black] {\strut$\,\rho_{\Delta}(0.2,x)$};
\draw[blue, fill=blue] (4.5,7) circle (2pt);
\draw[thick, red] (1,2) -- (8,2) node[right, black] {\strut$\,\rho(0.2,x)$};
\end{scope}
\end{tikzpicture}
\hspace{30mm}
\begin{tikzpicture}[every node/.style={anchor=south west,inner sep=0pt},x=1mm, y=1mm]
\node at (0,0) {\includegraphics[height=50mm,trim=62.5mm 156.5mm 63mm 47mm, clip]{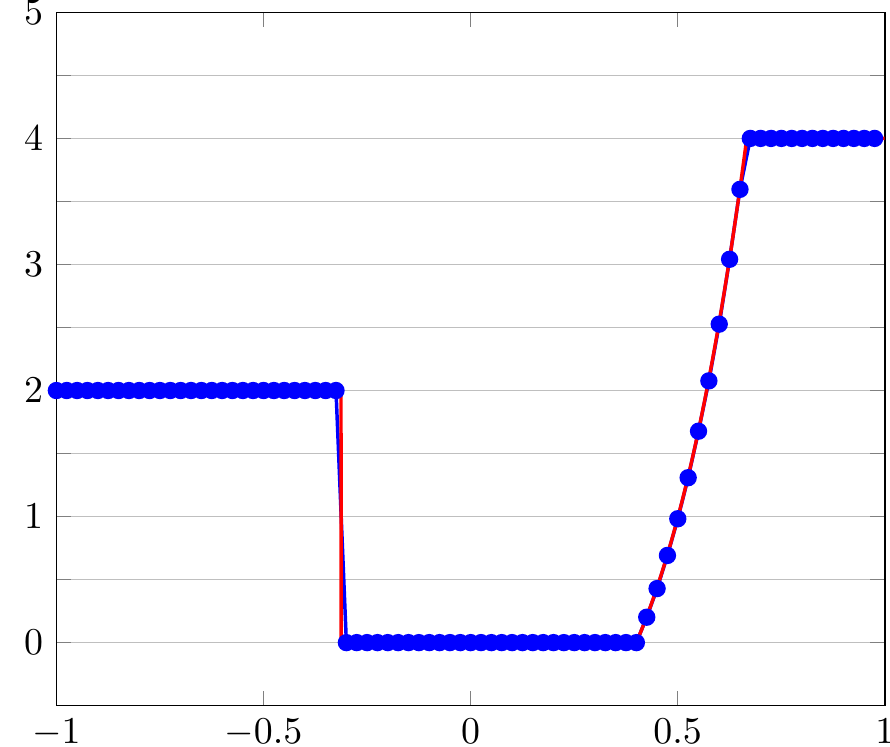}};

\begin{scope}[shift={(5,37)}]
\draw[black, fill=white] (0,0) rectangle (26,10);
\draw[thick, blue] (1,7) -- (8,7) node[right, black] {\strut$\,q_{\Delta}(0.2,x)$};
\draw[blue, fill=blue] (4.5,7) circle (2pt);
\draw[thick, red] (1,2) -- (8,2) node[right, black] {\strut$\,q(0.2,x)$};
\end{scope}
\end{tikzpicture}
\caption{Exact and numerical solutions for Example~\ref{ex:2} with $\Delta x=5\times10^{-4}$}
\label{fig.solutions.valve.closed}
\end{figure}

\begin{rmrk}
We now provide the main motivation to the choice of the numerical scheme we use in this paper: 
it lies in the fact that the RCM approximates well single discontinuities. 
By definition, the state $\check{u} = \check{u}\left(Q_{\rm h}(u_\ell),u_r\right)$ is given by the intersection of $\mathsf{BL}_2^{u_r}$ and $q=Q_{\rm h}(u_\ell)$; hence, the solution $\rsp[\check{u},u_r]$ has a single wave in $\xi>0$, namely a $2$-wave.
It is generically impossible, from a numerical point of view, to catch exact values on a curve.
For this reason we consider a numerical approximation $\check{u}_\Delta$ of $\check{u}$.
If $\check{u}_\Delta$ is subsonic, then $\rsp[\check{u}_\Delta,u_r]$ has only a $2$-wave in $\xi>0$ by Remark~\ref{r:BHK}. 
On the contrary, if $\check{u}_\Delta$ is supersonic, then $\rsp[\check{u}_\Delta,u_r]$ can well have a $1$-wave followed by a $2$-wave in $\xi>0$, because Remark~\ref{r:BHK} does not hold any more. 
Thus essentially any numerical approximation of $\rsp[\check{u},u_r]$ different from RCM, based on standard finite-volume methods (such as the Godunov scheme) 
has a $1$-wave followed by a $2$-wave in $\xi>0$ if $\check{u}_\Delta$ is supersonic, 
by the stability of the scheme.
The RCM avoids this sever drawback.
\end{rmrk}

\section{Maximization of the flow}
\label{s:max}
In this section we use the solver $\rsh$ to treat a maximization problem, by looking whenever possible to explicit solutions. 
Since we let the flow-threshold parameter $q_*$ vary, we use in the following the explicit notation $Q_{\rm h}^{q_*}$, $\mathsf{C}_\ell^{q_*}$,
$\mathsf{CH}_\ell^{q_*,\scriptscriptstyle\complement}$ for $Q_{\rm h}$, $\mathsf{C}_\ell$, $\mathsf{CH}_\ell^{\scriptscriptstyle\complement}$  
given by \eqref{e:rsv0}, \eqref{e:CL}, \eqref{e:CHellC}, respectively. As a consequence we denote by $\rsh^{q_*}$ the c-Riemann solver corresponding to $Q_{\rm h}^{q_*}$.

We fix a time horizon $T>0$ and an initial datum attaining the values $u_i, u_\ell, u_r \in \Omega$ ($i$ for \lq\lq ingoing\rq\rq) for $x$ belonging to $(-\infty, -1)$, $[-1,0)$, $[0,\infty)$, respectively; we only let $q_*$ vary. 
For any $q_*\geq0$, we denote by $u^{q_*}(t,x) \doteq (\rho^{q_*}(t,x),q^{q_*}(t,x))$ the solution corresponding to the initial condition
\begin{equation}\label{e:ic3}
u^{q_*}(0,x)  = \begin{cases}
u_i&\hbox{ if }x<-1,\\
u_\ell&\hbox{ if }-1\leq x<0,\\
u_r&\hbox{ if }x\geq0,
\end{cases}
\end{equation}
and constructed by applying $\rsh^{q_*}$ at $x=0$ and $\rsp$ elsewhere.
The choice of the initial datum as in \eqref{e:ic3} represents a Riemann problem at the valve position with a perturbation on the left. The choice of the point $-1$ is for simplicity: a different value only leads to a rescaling.

Assume for the moment that for any $q_*\geq0$ the corresponding solution $u^{q_*}$ is unique and well defined up to a fixed time $T$. 
We then study the maximization problem
\begin{equation}\label{e:opt1}
\max_{q_*>0} \mathfrak{Q}(q_*,T) \qquad \hbox{ with }\qquad \mathfrak{Q}(q_*,T) \doteq \frac{1}{T} \int_0^T q^{q_*}(t,0) \,\d t,
\end{equation}
for the average flow $\mathfrak{Q}(q_*,T)$.
We point out that the above assumption of existence of solutions is not trivial, because of the possibility of blow up in finite time \cite{MR1849664,  Bressan-NoBVbounds, MR1752421}; furthermore, to solve \eqref{e:opt1} we should also need qualitative properties of the solutions. As a consequence, analytic results for maximization problem \eqref{e:opt1} can hardly be proved in a general setting.
For this reason, in the following Subsections~\ref{ss:1} and~\ref{ss:2} we focus on some particular cases where {\em analytical} results are available.
These results will be crucial benchmarks for the numerical simulations in Subsections~\ref{sub:n1} and~\ref{sub:n2}, 
which regard an example which doesn't fit in the analytical results obtained in the preceding subsections. The last Subsection \ref{ss:3} contains a further case study which is treated only numerically.
For all the numerical simulations performed in the sequel, we always take $\Delta x=5\times10^{-4}$.

\subsection{The case \texorpdfstring{$u_i=u_\ell$}{}}
\label{ss:1}
In the case $u_i=u_\ell$, problem \eqref{e:opt1} only concerns solutions to a fixed Riemann problem at $x=0$; in particular, $q^{q_*}(t,0)=Q_{\rm h}^{q_*}(u_\ell)$ does not depend on $t$. We recall that the set $\mathsf{CH}_{\ell,1}$ does not depend on $q_*$, see \eqref{e:CH}.

\begin{prpstn}
Consider the maximization problem \eqref{e:opt1} in the case $u_i=u_\ell$. Then for any $T>0$ we have
\[
\max_{q_*>0}\mathfrak{Q}(q_*,T) = \overline{Q}(u_\ell),
\]
and a maximizer is $q_*=\overline{Q}(u_\ell)$. 
Moreover, the maximizer is unique if and only if $u_\ell \in \mathsf{CH}_{\ell,1}$.
\end{prpstn}

\begin{proof} 
In the case $u_i=u_\ell$, problem \eqref{e:opt1} reduces to maximize $q_*\mapsto Q_{\rm h}^{q_*}(u_\ell)$ because $q^{q_*}(t,0) = Q_{\rm h}^{q_*}(u_\ell)$ for any $t>0$ and therefore $\mathfrak{Q}(q_*,T) = Q_{\rm h}^{q_*}(u_\ell)$ for any $T>0$. 

\begin{figure}[!htb]\centering
\resizebox{.9\linewidth}{!}{
\begin{tikzpicture}[>=latex, x=10mm, y=9mm, semithick]
\draw[->] (0,0) --  (6.1,0) node[below]{\strut$q_*$} coordinate (x axis);
\draw[->] (0,0) -- (0,4) node[left]{$Q_{\rm h}^{q_*}(u_\ell)$} coordinate (y axis);
\draw[dashed] (0,3)node[left]{$\overline{Q}(u_\ell)$} -- (3,3) -- (3,0) node[below]{\strut$\overline{Q}(u_\ell)$};
\draw[ultra thick] (0,0) -- (3,3) node[midway,sloped,above]{$u_\ell\in\mathsf{CH}_{\ell,1} \setminus \mathsf{C}_\ell^{q_*}$};
\draw[fill=black] (0,0) circle (2pt);
\draw[fill=black] (3,3) circle (2pt);

\draw[ultra thick] (3,0) -- (6,0) node[midway,above]{$u_\ell \in \mathsf{CH}_{\ell,1} \cap \mathsf{C}_\ell^{q_*}$};
\draw[fill=white] (3,0) circle (2pt);
\end{tikzpicture}
\begin{tikzpicture}[every node/.style={anchor=south west,inner sep=0pt},x=1mm, y=.5mm]
\node at (4,4) {\includegraphics[width=90mm]{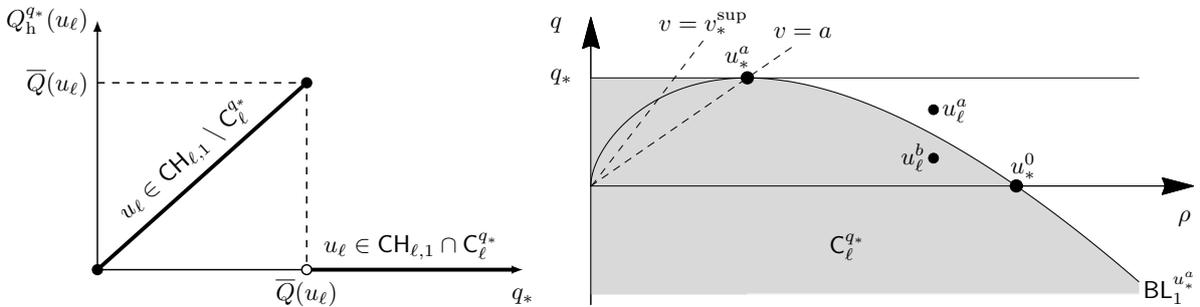}};
\node at (90,29) {\strut $\rho$};
\node at (0,85) {\strut $q$};
\node at (32,82) {\strut $v=a$};
\node at (0,70) {\strut $q_*$};
\node at (85,7) {\strut $\mathsf{BL}_1^{u_*^a}$};
\node at (25,75) {\strut $u_*^a$};
\node at (66,43) {\strut $u_*^0$};
\node at (40,20) {\strut $\mathsf{C}_\ell^{q_*}$};
\draw[dashed] (6,42) -- (22,84);
\node[above] at (22,84) {\strut $v=v_*^{\sup}$};
\draw[fill=black] (55,50) circle (2pt) node[left] {\strut $u_\ell^b\ $};
\draw[fill=black] (55,64) circle (2pt) node[right] {\strut $\ u_\ell^a$};
\end{tikzpicture}
}
\caption{\label{f:f2}{Left: Plot of $q_*\mapsto Q_{\rm h}^{q_*}(u_\ell)$ with $u_\ell \in \mathsf{CH}_{\ell,1}$ fixed, see \eqref{e:Ghost1}.
Right: Two states in $\mathsf{CH}_{\ell,1}$, with $u_\ell^a \in \mathsf{CH}_{\ell,1} \setminus \mathsf{C}_\ell^{q_*}$ and $u_\ell^b \in\mathsf{CH}_{\ell,1} \cap \mathsf{C}_\ell^{q_*}$.
The shaded region represents the set $\mathsf{C}_\ell^{q_*}$.
}}
\end{figure}

\begin{figure}[!htb]\centering
\resizebox{.9\linewidth}{!}{
\begin{tikzpicture}[>=latex, x=10mm, y=9mm, semithick]
\draw[->] (0,0) --  (8.1,0) node[below]{\strut$q_*$} coordinate (x axis);
\draw[->] (0,0) -- (0,4) node[left]{$Q_{\rm h}^{q_*}(u_\ell)$} coordinate (y axis);
\draw[dashed] (0,3) node[left]{$q_\ell$} -- (3,3) -- (3,0) node[below]{\strut$q_\ell = \overline{Q}(u_\ell)$};
\draw[ultra thick] (0,0) -- (3,3) node[midway,sloped,above]{$u_\ell\in\mathsf{CH}_{\ell,3}^{q_*}$};
\draw[ultra thick] (3,3) -- (5,3) node[midway,above]{$u_\ell \in \mathsf{CH}_\ell^{q_*,\scriptscriptstyle\complement}$};
\draw[dashed] (5,0) node[below]{\strut$\mathring{q}(u_\ell)$} -- (5,3);
\draw[fill=black] (0,0) circle (2pt);
\draw[fill=black] (3,3) circle (2pt);
\draw[fill=black] (5,3) circle (2pt);

\draw[ultra thick] (5,0) -- (8,0) node[midway,above]{$u_\ell \in \mathsf{CH}_{\ell,2}^{q_*}$};
\draw[fill=white] (5,0) circle (2pt);
\end{tikzpicture}
\begin{tikzpicture}[every node/.style={anchor=south west,inner sep=0pt},x=1mm, y=.5mm]
\node at (4,4) {\includegraphics[width=90mm]{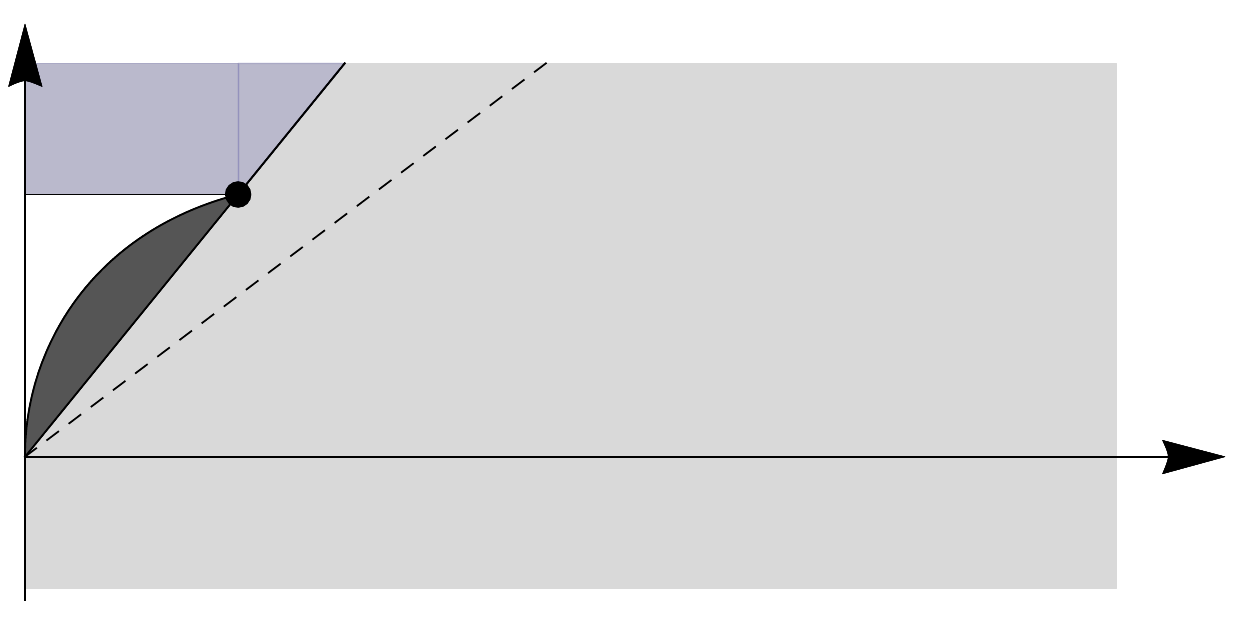}};
\node at (90,13) {\strut $\rho$};
\node at (2,85) {\strut $q$};
\node at (2,62) {\strut $q_*$};
\draw[white,fill=white] (10,45) circle (2pt) node {\strut\color{white} $u_\ell^c$};
\draw[fill=black] (11,62) circle (2pt) node[below left] {\strut $u_\ell^b\,$};
\draw[fill=black] (10,75) circle (2pt) node[right] {\strut $\ u_\ell^a$};
\draw[-latex] (2,45) node[left] {\strut $\mathsf{CH}_{\ell,2}^{q_*}$} --++ (5,0);
\draw[-latex] (2,56) node[left] {\strut $\mathsf{CH}_\ell^{q_*,\scriptscriptstyle\complement}$} --++ (3,0);
\node at (50,40) {\strut $\mathsf{CH}_{\ell,1}$};
\draw[-latex] (2,75) node[left] {\strut $\mathsf{CH}_{\ell,3}^{q_*}$} --++ (3,0);
\node at (25,85) {\strut $v=v_*^{\sup}$};
\node at (40,85) {\strut $v=a$};
\end{tikzpicture}
}
\caption{\label{f:f1}{Left: Plot of $q_*\mapsto Q_{\rm h}^{q_*}(u_\ell)$ with $u_\ell \in \mathsf{CH}_{\ell,1}^{\scriptscriptstyle\complement}$ fixed, see \eqref{e:Ghost2}. 
Right: Three possible elements of $\mathsf{CH}_{\ell,1}^{\scriptscriptstyle\complement}$, with $u_\ell^a \in \mathsf{CH}_{\ell,3}^{q_*}$, $u_\ell^b \in\mathsf{CH}_\ell^{q_*,\scriptscriptstyle\complement}$ and $u_\ell^c \in\mathsf{CH}_{\ell,2}^{q_*}$.
}}
\end{figure}

Consider first the case $u_\ell \in \mathsf{CH}_{\ell,1}$, see \figurename~\ref{f:f2} on the right.
By \eqref{e:KingCrimson} and \eqref{e:rsv00} we have
\begin{equation}
\label{e:Ghost1}
Q_{\rm h}^{q_*}(u_\ell) = \begin{cases}
q_*&\hbox{ if }q_* \in [0,\overline{Q}(u_\ell)],\\
0&\hbox{ if }q_* >\overline{Q}(u_\ell).
\end{cases}
\end{equation}
The plot of $q_*\mapsto Q^{q_*}(u_\ell)$ is represented in \figurename~\ref{f:f2} on the left.

Assume now that $u_\ell \in \mathsf{CH}_{\ell,1}^{\scriptscriptstyle\complement}$, see \figurename~\ref{f:f1} on the right.
We denote
\[\mathring{q}(u) \doteq \bar{q}\left(\hat{u}(0,u)\right) = \dfrac{\rho}{4\,a\,e} \left[ \sqrt{v^2 + 4 \, a^2} + v \right]^2.\]
Observe that in this case $\overline{Q}(u_\ell)=q_\ell$ by \eqref{e:AliceinChains}.
We use again \eqref{e:rsv00} to deduce the following:

\begin{itemize}

\item if $u_\ell \in \mathsf{CH}_{\ell,3}^{q_*} \subset \mathsf{CH}_\ell \setminus \mathsf{C}_\ell$, then we have $Q_{\rm h}^{q_*}(u_\ell) = q_*$ and $q_*\leq \overline{Q}(u_\ell)$ by $\eqref{e:KingCrimson}_2$;  

\item if $u_\ell \in \mathsf{CH}_\ell^{q_*,\scriptscriptstyle\complement}$, then we have $Q_{\rm h}^{q_*}(u_\ell)=q_\ell$ and $q_*\in(\overline{Q}(u_\ell),\mathring{q}(u_\ell)]$ by \cite[(5.1)]{CR1}; 

\item if $u_\ell \in \mathsf{CH}_{\ell,2}^{q_*} \subset \mathsf{CH}_\ell \cap \mathsf{C}_\ell$, then we have $Q_{\rm h}^{q_*}(u_\ell)=0$ and $q_*> \overline{Q}(u_\ell)$ by $\eqref{e:KingCrimson}_1$.

\end{itemize}
Therefore we deduce that 
\begin{equation}
\label{e:Ghost2}
Q_{\rm h}^{q_*}(u_\ell) = \begin{cases}
q_*&\hbox{ if }q_* \in [0,\overline{Q}(u_\ell)],\\
q_\ell&\hbox{ if }q_* \in (\overline{Q}(u_\ell),\mathring{q}(u_\ell)],\\
0&\hbox{ if }q_* >\overline{Q}(u_\ell).
\end{cases}
\end{equation}
See \figurename~\ref{f:f1} for the graph of $q_*\mapsto Q^{q_*}_{\rm h}(u_\ell)$ in this case. This concludes the proof.
\end{proof}

In \figurename~\ref{f:numCASE1} we show our numerical simulations corresponding to $a=2$, $u_r = (1,-1)$ and
\begin{align}\label{e:data1}
&\hbox{left:}&
&u_i = u_\ell = (2,2),&
&\overline{Q}(u_\ell) = \frac{4}{\sqrt{e}} \approx 2.43,
\\\label{e:data2}
&\hbox{right:}&
&u_i = u_\ell = \left(\frac{1}{4},\frac{5}{2}\right),&
&\overline{Q}(u_\ell) = \frac{5}{2},&
&\hspace{-10mm}
\mathring{q}(u_\ell) = \frac{(10+2 \sqrt{29})^2}{32 e} \approx 4.96.
\end{align}
We notice a very good match with the analytic results, see \eqref{e:Ghost1} and \eqref{e:Ghost2}. The slight deviation from the expected value $0$ (for $q_*$ approximately larger than $5$) in \figurename~\ref{f:numCASE1} on the right is only due to numerical rounding errors.

\begin{figure}[!htb]\centering
\begin{subfigure}[t]{0.35\textwidth}\centering
\resizebox{\linewidth}{!}{
\begin{tikzpicture}[every node/.style={anchor=south west,inner sep=0pt},x=1mm, y=1mm]
\node at (0,0) {\includegraphics[width=60mm,trim=46.5mm 158mm 40mm 47mm, clip]{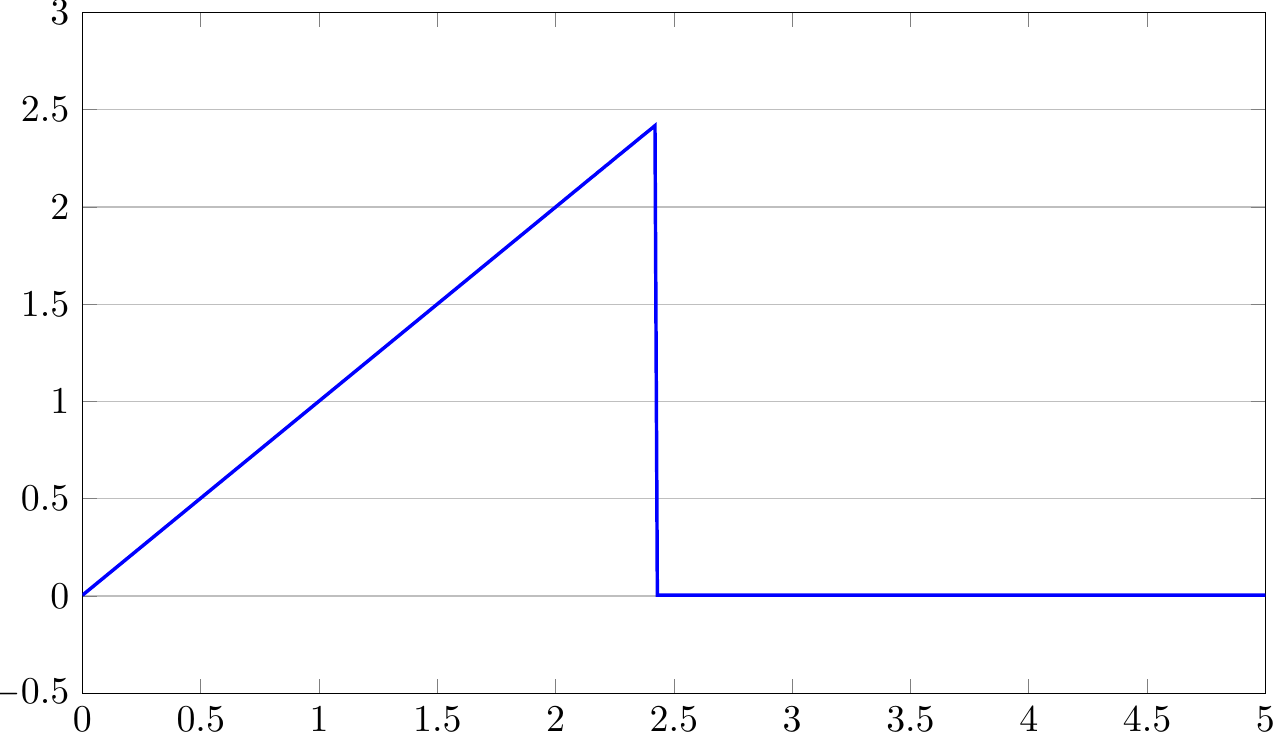}};
\end{tikzpicture}}
\caption*{$q_* \mapsto Q_{\rm h}^{q_*}(u_\ell)$}
\end{subfigure}
\hspace{5mm}
\begin{subfigure}[t]{0.35\textwidth}\centering
\resizebox{\linewidth}{!}{
\begin{tikzpicture}[every node/.style={anchor=south west,inner sep=0pt},x=1mm, y=1mm]
\node at (0,0) {\includegraphics[width=60mm,trim=46.5mm 158mm 40mm 47mm, clip]{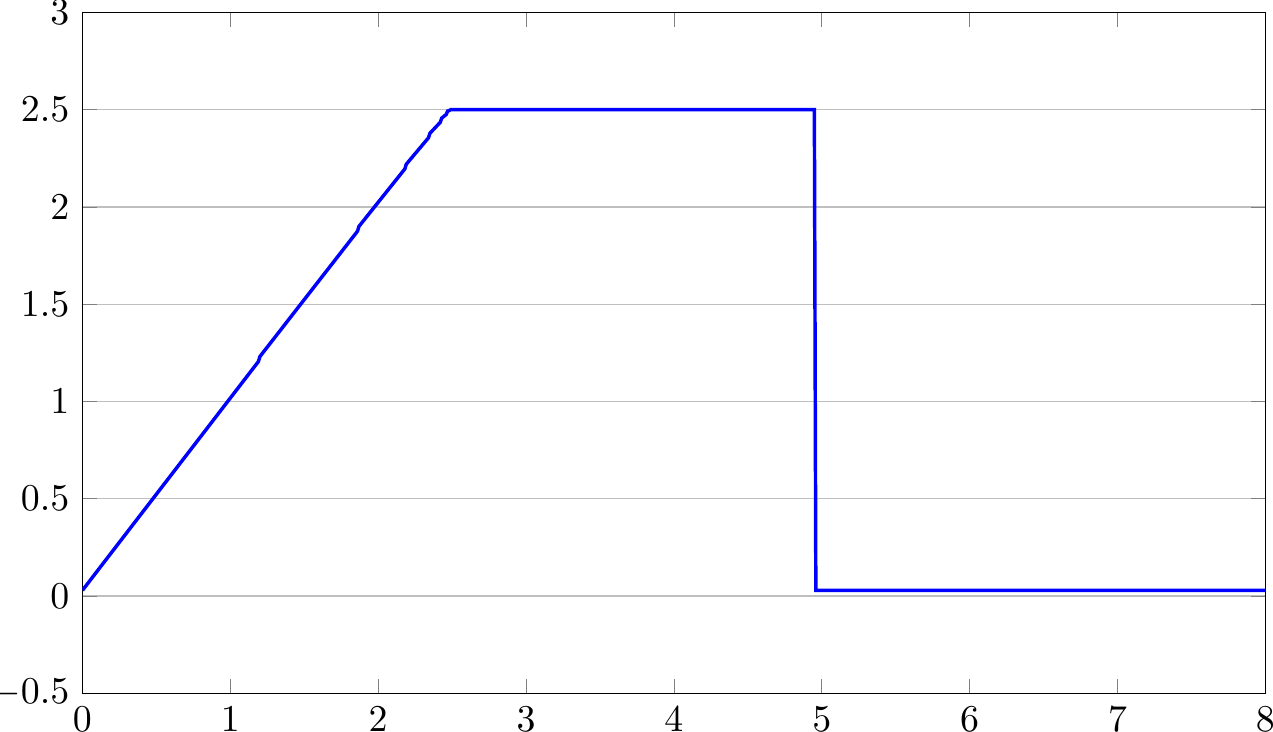}};
\end{tikzpicture}}
\caption*{$q_* \mapsto Q_{\rm h}^{q_*}(u_\ell)$}
\end{subfigure}
\caption{\label{f:numCASE1}{Numerical simulations to the maximization problem \eqref{e:opt1} with $u_i=u_\ell$;  the values of the parameters are as in \eqref{e:data1} and \eqref{e:data2}, respectively.}}
\end{figure}

\subsection{The case \texorpdfstring{$\rsp[u_i,u_\ell]$ is a $2$-shock}{}}
\label{ss:2}

In this subsection we show how to construct, for small times, an explicit solution to the Cauchy problem for system \eqref{e:systemq} with an initial datum as in \eqref{e:ic3}.
We apply $\rsp$, for $x\ne0$, and $\rsh^{q_*}$, at $x=0$, at each discontinuity of the initial datum and at each wave interaction.
As we mentioned above, we focus on a single explicit example; other cases can be handled similarly. We point out that the special case $u_\ell=u_r$ has the advantage of reducing the number of initial parameters; however, the property $u(t,0^-)=u(t,0^+)$ is not preserved when $q_*$ varies, because $\rsh^{q_*}[u_\ell,u_\ell]$ may lead to solutions that do not have such property. For this reason, we do not treat explicitly this special case. 

We assume that at $t=0$ there is no flow on the right of the valve. On the left, instead, we have a supersonic perturbation $u_i$ which is separated from the state $u_\ell$ by a $2$-shock wave moving toward the valve. For simplicity we assume that the state $u_\ell$ is sonic. More precisely, see \figurename~\ref{f:optimization00}, we assume
\begin{align}\label{e:conditions}
&u_\ell \in \mathsf{FL}_2^{u_i},&
&v_r=0<v_\ell=a<v_i,&
&q_r=0 < q_\ell < \tilde{q}\left(u_i,\hat{u}(0,u_\ell)\right) < q_i < \bar{q}\left( \tilde{u}\left(u_i,\hat{u}(0,u_\ell)\right) \right).
\end{align}
\begin{figure}[!htb]\centering
\resizebox{.3\linewidth}{!}{
\begin{tikzpicture}[every node/.style={anchor=south west,inner sep=0pt},x=1mm, y=1mm]
\node at (4,4) {\includegraphics[width=60mm]{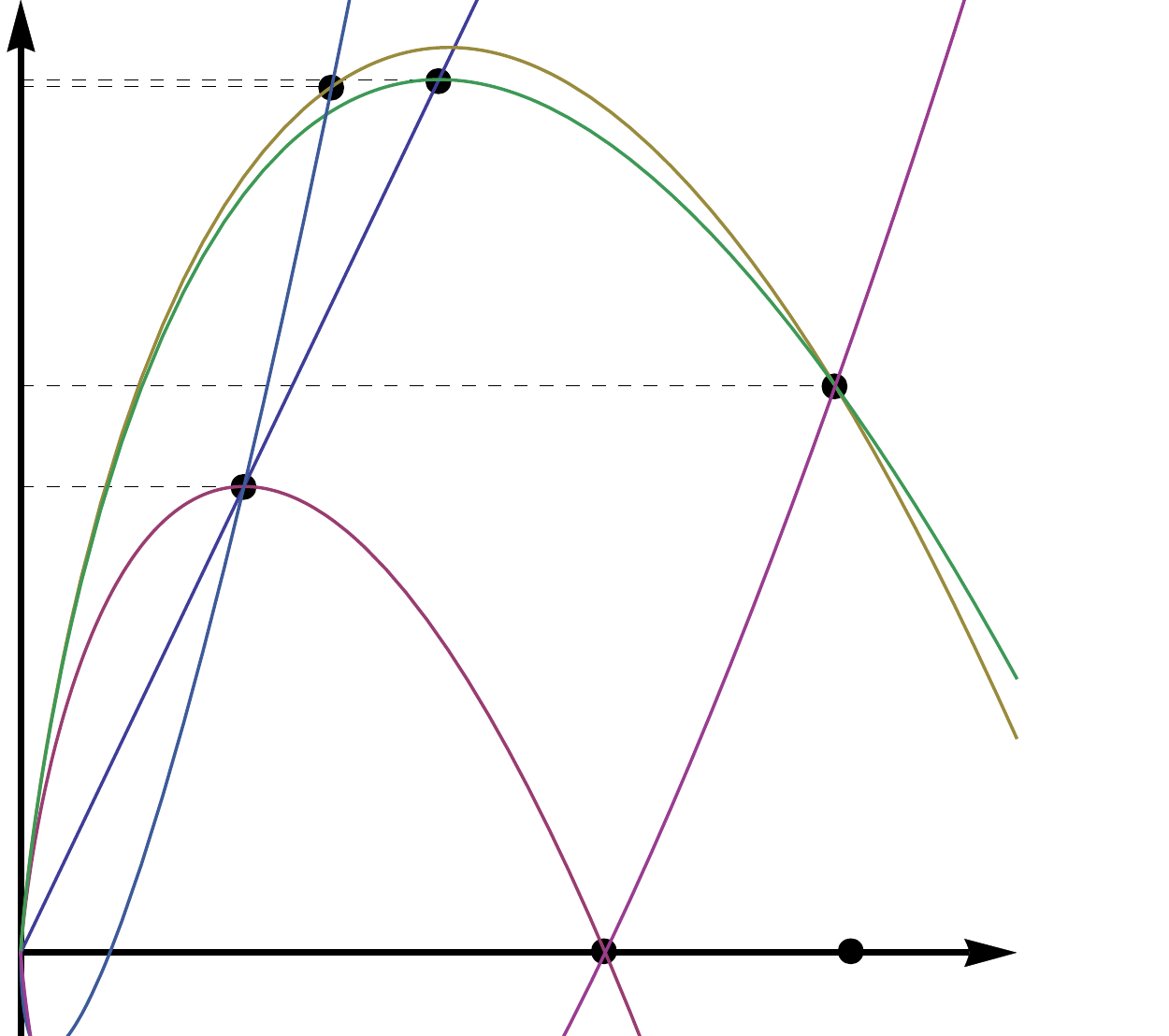}};
\node at (55,3) {\strut $\rho$};
\node at (47,8) {\strut $u_r$};
\node at (0,55) {\strut $q$};
\node[left] at (4,32) {\strut $q_\ell$};
\node[left] at (4,37) {\strut $\tilde{q}$};
\draw[latex-] (4.5,52) -- (0,45) node[left]{\strut $q_i$};
\draw[latex-] (4.5,53) -- (0,53) node[left]{\strut $\bar{q}$};
\node at (37,8) {\strut $\hat{u}$};
\node at (27,48) {\strut $\bar{u}$};
\node at (18,31) {\strut $u_\ell$};
\node at (16,52.5) {\strut $u_i$};
\node at (49,35) {\strut $\tilde{u}$};
\node at (25,57) {\strut $v=a$};
\node at (6,-1) {\strut $\mathsf{FL}_2^{u_i}$};
\node at (26,-1) {\strut $\mathsf{BL}_2^{\hat{u}}$};
\node[right] at (56,25) {\strut $\mathsf{FL}_1^{\tilde{u}_1}$};
\node[right] at (56,17) {\strut $\mathsf{FL}_1^{u_i}$};
\node at (37,-1) {\strut $\mathsf{FL}_1^{u_\ell}$};
\end{tikzpicture}}
\caption{The states $u_i$, $u_\ell$ and $u_r$ given in \eqref{e:data02-A}, \eqref{e:data02-B} satisfy \eqref{e:conditions}. We denote $\hat{u} \doteq \hat{u}(0,u_\ell)$, $\tilde{u} \doteq \tilde{u}\left(u_i,\hat{u}(0,u_\ell)\right)$ and $\bar{u} \doteq \bar{u}\left( \tilde{u}\left(u_i,\hat{u}(0,u_\ell)\right) \right)$.
}
\label{f:optimization00}
\end{figure}

\subsubsection{The explicit solution for small times}\label{subsub:las} 
We now construct an {\em exact} and {\em explicit} solution to the initial-value problem \eqref{e:system}, \eqref{e:ic3} for small times, under the assumptions in \eqref{e:conditions}. Since an interaction involving a rarefaction wave is complicate to handle explicitly, we stop the construction when such interactions occur.
We distinguish four cases; we emphasize that in the following pictures also the interaction patterns in the space $(x,t)$ (point coordinates and slopes) are {\em exact} and not merely representative. We refer to \figurename~\ref{f:optimization}.

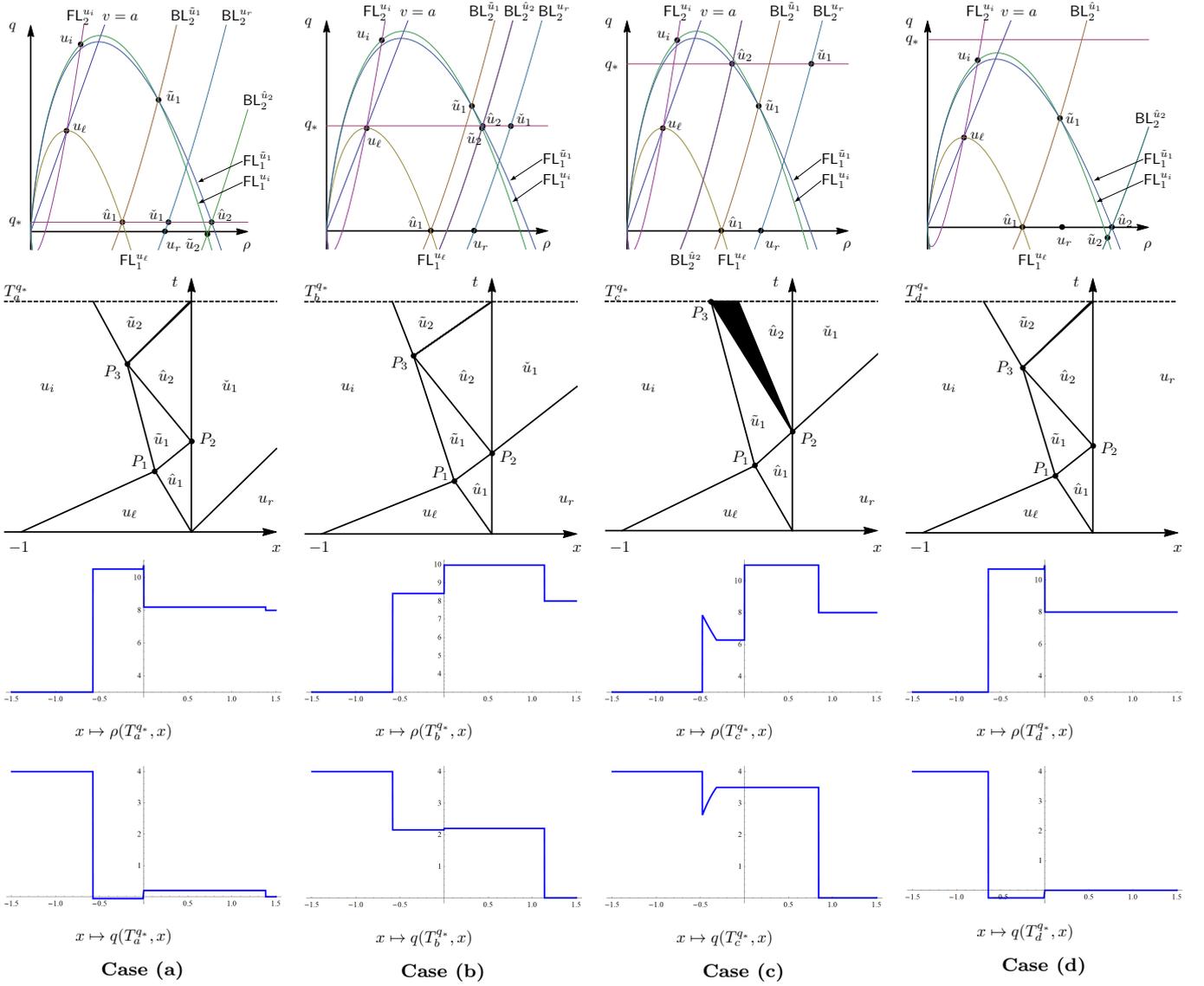
\begin{figure}[!htb]\centering
\begin{subfigure}[t]{0.235\textwidth}
\centering
\resizebox{\linewidth}{!}{
\begin{tikzpicture}[every node/.style={anchor=south west,inner sep=0pt},x=1mm, y=1mm]
\node at (4,4) {\includegraphics[width=60mm]{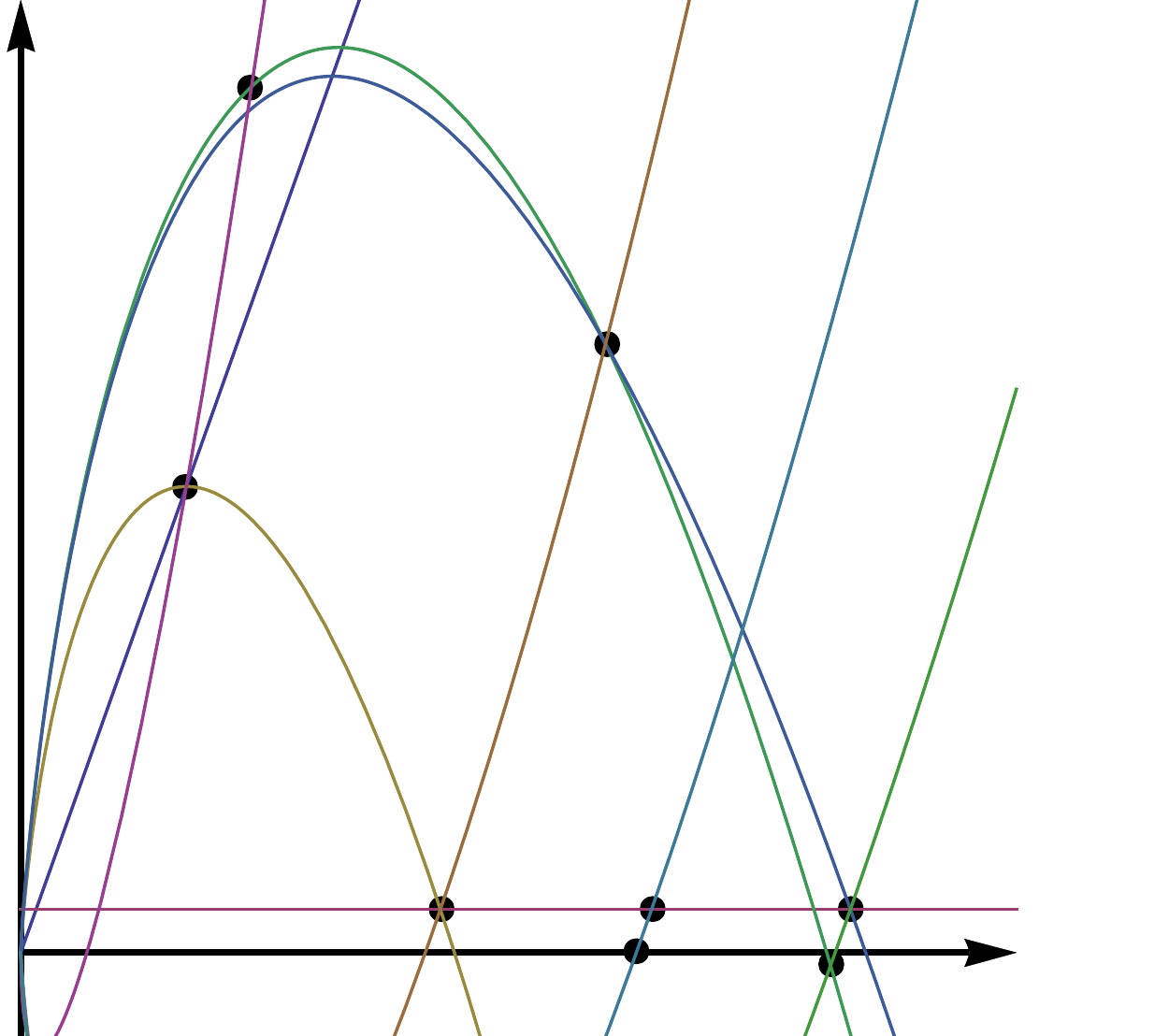}};
\node at (55,3) {\strut $\rho$};
\node at (41.5,3) {\strut $\tilde{u}_2$};
\node at (37,3) {\strut $u_r$};
\node at (0,55) {\strut $q$};
\node at (0,9) {\strut $q_*$};
\node at (49,10) {\strut $\hat{u}_2$};
\node at (33,10) {\strut $\check{u}_1$};
\node at (22,10) {\strut $\hat{u}_1$};
\node at (15,31) {\strut $u_\ell$};
\node at (12,52) {\strut $u_i$};
\node at (37,38) {\strut $\tilde{u}_1$};
\node at (22,57) {\strut $v=a$};
\node at (13.5,57) {\strut $\mathsf{FL}_2^{u_i}$};
\node at (38,57) {\strut $\mathsf{BL}_2^{\tilde{u}_1}$};
\node at (50,57) {\strut $\mathsf{BL}_2^{u_r}$};
\node at (55,37) {\strut $\mathsf{BL}_2^{\hat{u}_2}$};
\draw[latex-] (44.4,20) -- (55,25) node[right] {\strut $\mathsf{FL}_1^{\tilde{u}_1}$};
\draw[latex-] (44.5,15) -- (55,20) node[right] {\strut $\mathsf{FL}_1^{u_i}$};
\node at (26,-1) {\strut $\mathsf{FL}_1^{u_\ell}$};
\end{tikzpicture}}
\\
\resizebox{\linewidth}{!}{\begin{tikzpicture}[every node/.style={anchor=south west,inner sep=0pt},x=1mm, y=1mm]
\node at (4,4) {\includegraphics[width=60mm]{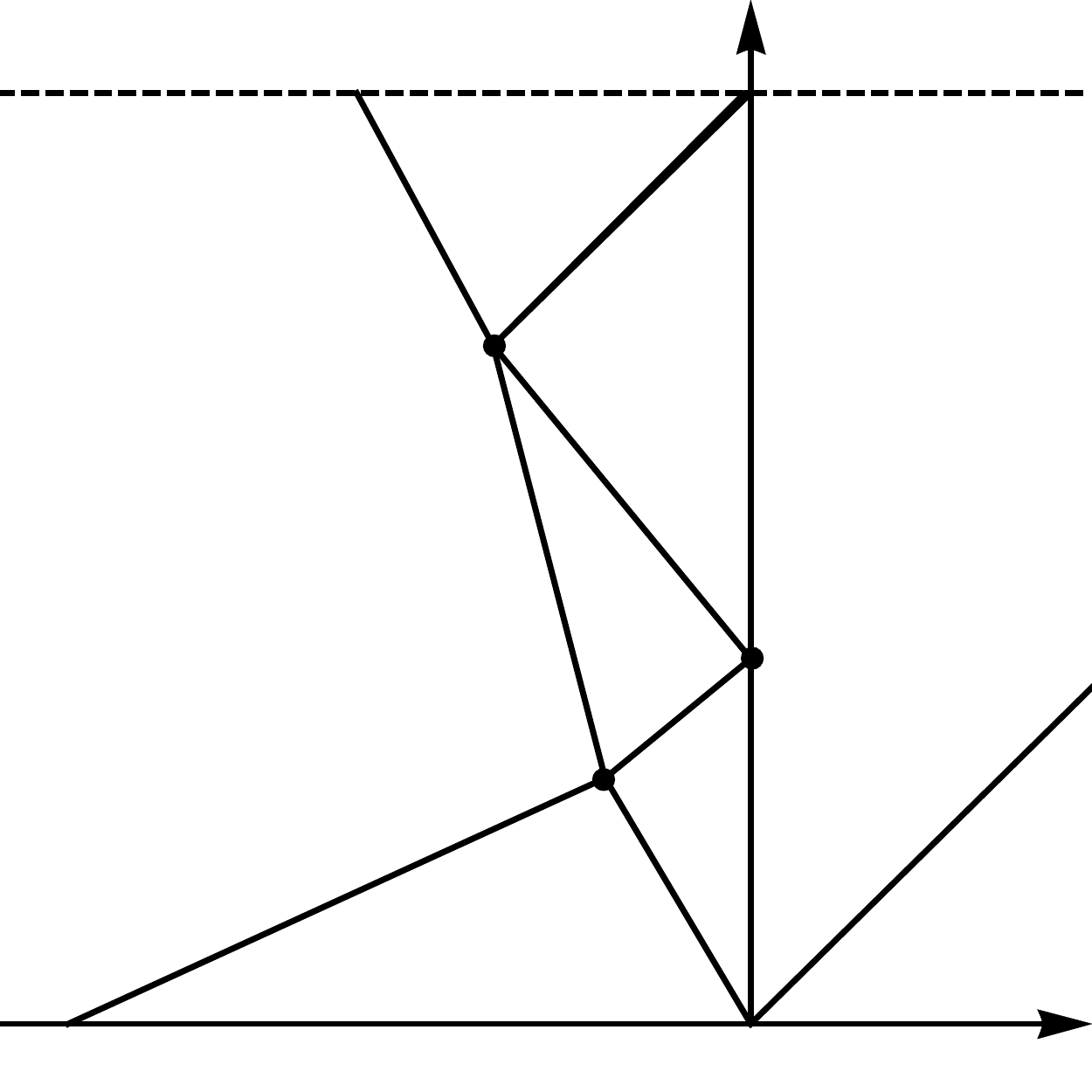}};
\node[right] at (4,61) {\strut $T_a^{q_*}$};
\node at (63,2) {\strut $x$};
\node at (5,2) {\strut $-1$};
\node at (31,52) {\strut $\tilde{u}_2$};
\node at (60,13) {\strut $u_r$};
\node at (41,61) {\strut $t$};
\node at (38,40) {\strut $\hat{u}_2$};
\node at (52,38) {\strut $\check{u}_1$};
\node at (40,17) {\strut $\hat{u}_1$};
\node at (30,10) {\strut $u_\ell$};
\node at (12,38) {\strut $u_i$};
\node at (37,26) {\strut $\tilde{u}_1$};
\node at (32,21) {\strut $P_1$};
\node at (47,26) {\strut $P_2$};
\node at (26,41) {\strut $P_3$};
\end{tikzpicture}}
\\
\resizebox{\linewidth}{!}{\begin{tikzpicture}[every node/.style={anchor=south west,inner sep=0pt},x=1mm, y=1mm]
\node at (4,4) {\includegraphics[width=60mm]{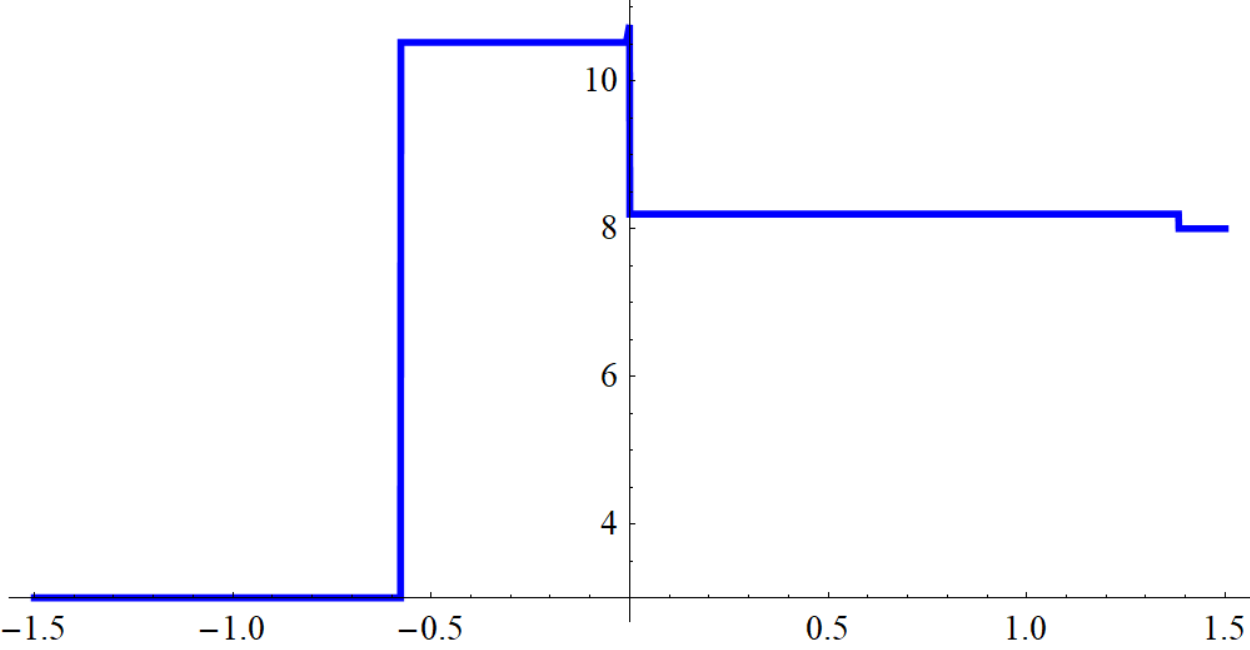}};
\node[below] at (30,0) {\strut $x\mapsto \rho(T_a^{q_*},x)$};
\end{tikzpicture}}
\\[10pt]
\resizebox{\linewidth}{!}{\begin{tikzpicture}[every node/.style={anchor=south west,inner sep=0pt},x=1mm, y=1mm]
\node at (4,4) {\includegraphics[width=60mm]{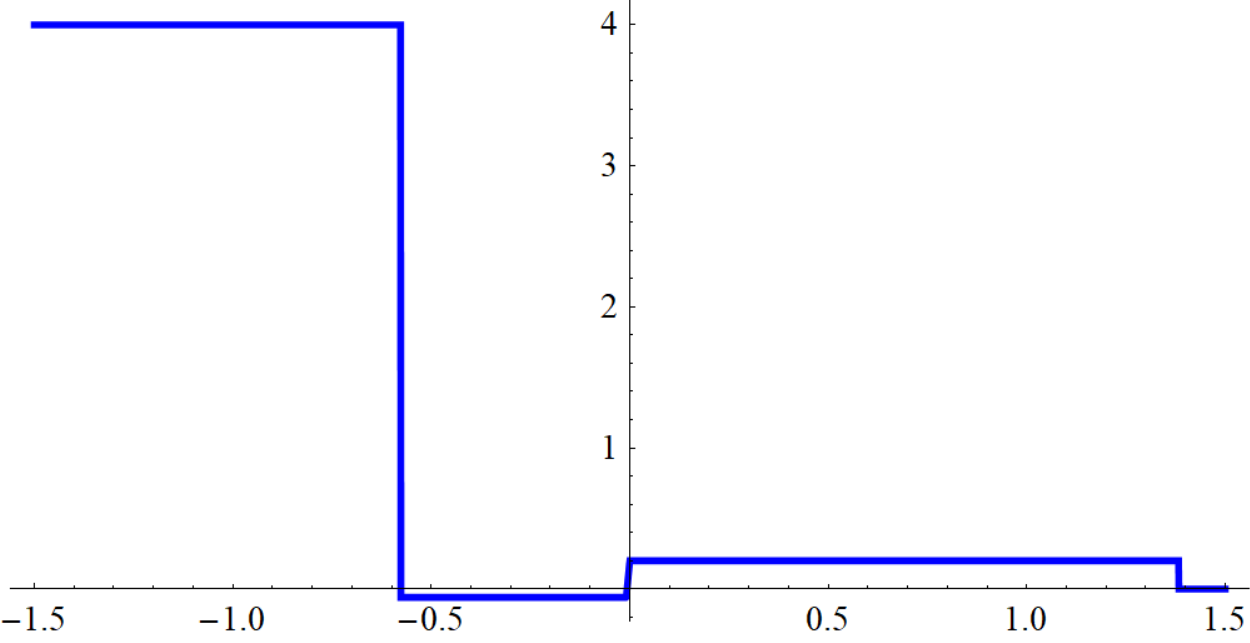}};
\node[below] at (30,0) {\strut $x\mapsto q(T_a^{q_*},x)$};
\end{tikzpicture}}
\caption*{\ref{case:A}}
\end{subfigure}
\hfill
\begin{subfigure}[t]{0.235\textwidth}
\centering
\resizebox{\linewidth}{!}{\begin{tikzpicture}[every node/.style={anchor=south west,inner sep=0pt},x=1mm, y=1mm]
\node at (4,4) {\includegraphics[width=60mm]{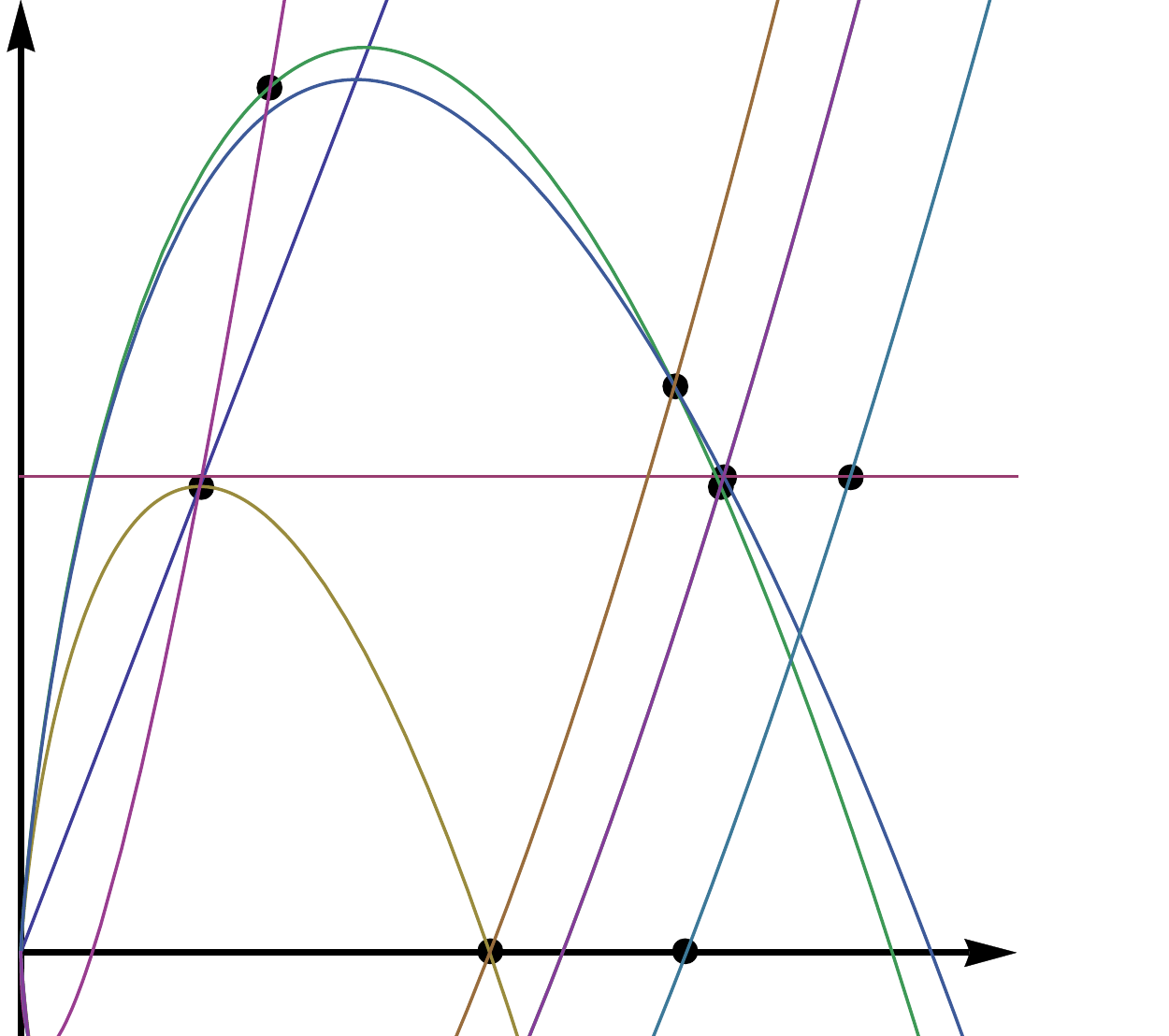}};
\node at (55,3) {\strut $\rho$};
\node at (37.5,27.5) {\strut $\tilde{u}_2$};
\node at (39,3) {\strut $u_r$};
\node at (0,55) {\strut $q$};
\node at (0,31) {\strut $q_*$};
\node at (42,32) {\strut $\hat{u}_2$};
\node at (49,32) {\strut $\check{u}_1$};
\node at (24,8) {\strut $\hat{u}_1$};
\node at (14.5,27) {\strut $u_\ell$};
\node at (12,52) {\strut $u_i$};
\node at (34,35) {\strut $\tilde{u}_1$};
\node at (22,57) {\strut $v=a$};
\node at (13.5,57) {\strut $\mathsf{FL}_2^{u_i}$};
\node at (38,57) {\strut $\mathsf{BL}_2^{\tilde{u}_1}$};
\node at (46,57) {\strut $\mathsf{BL}_2^{\hat{u}_2}$};
\node at (54,57) {\strut $\mathsf{BL}_2^{u_r}$};
\draw[latex-] (47.5,20) -- (55,25) node[right] {\strut $\mathsf{FL}_1^{\tilde{u}_1}$};
\draw[latex-] (48,15) -- (55,20) node[right] {\strut $\mathsf{FL}_1^{u_i}$};
\node at (26,-1) {\strut $\mathsf{FL}_1^{u_\ell}$};
\end{tikzpicture}}
\\
\resizebox{\linewidth}{!}{\begin{tikzpicture}[every node/.style={anchor=south west,inner sep=0pt},x=1mm, y=1mm]
\node at (4,4) {\includegraphics[width=60mm]{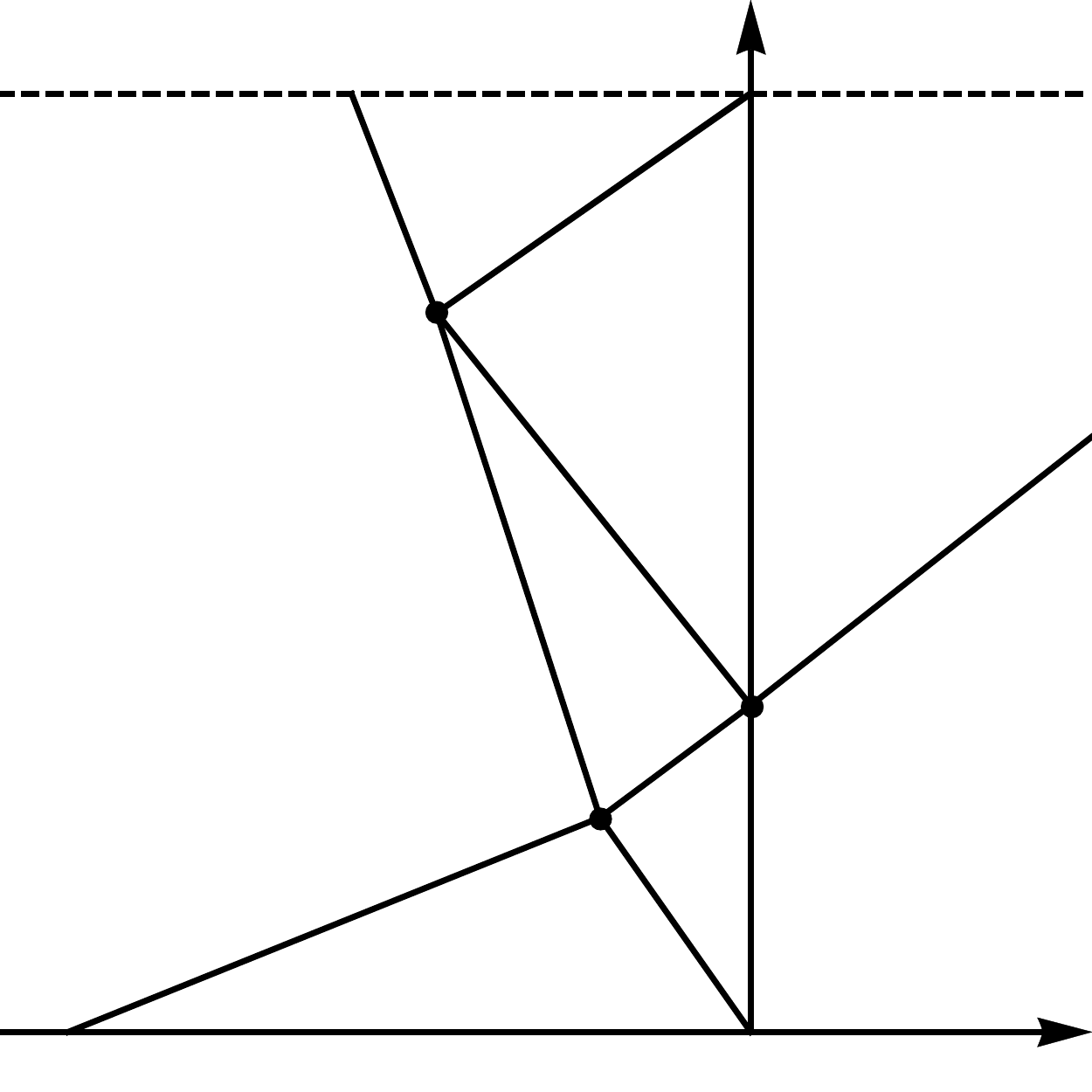}};
\node[right] at (4,61) {\strut $T_b^{q_*}$};
\node at (63,2) {\strut $x$};
\node at (5,2) {\strut $-1$};
\node at (29,52) {\strut $\tilde{u}_2$};
\node at (60,13) {\strut $u_r$};
\node at (41,61) {\strut $t$};
\node at (38,40) {\strut $\hat{u}_2$};
\node at (52,42) {\strut $\check{u}_1$};
\node at (41,15) {\strut $\hat{u}_1$};
\node at (30,10) {\strut $u_\ell$};
\node at (12,38) {\strut $u_i$};
\node at (36,26) {\strut $\tilde{u}_1$};
\node at (32,18.5) {\strut $P_1$};
\node at (47,22) {\strut $P_2$};
\node at (23,43) {\strut $P_3$};
\end{tikzpicture}}
\\
\resizebox{\linewidth}{!}{\begin{tikzpicture}[every node/.style={anchor=south west,inner sep=0pt},x=1mm, y=1mm]
\node at (4,4) {\includegraphics[width=60mm]{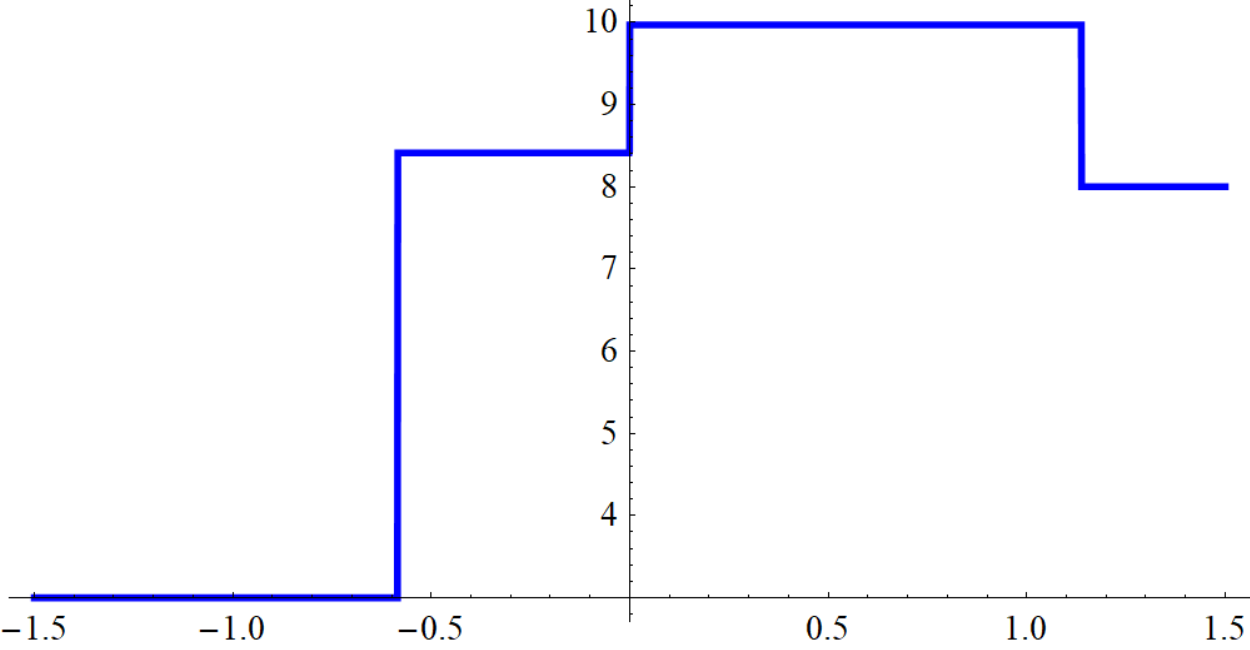}};
\node[below] at (30,0) {\strut $x\mapsto \rho(T_b^{q_*},x)$};
\end{tikzpicture}}
\\[10pt]
\resizebox{\linewidth}{!}{\begin{tikzpicture}[every node/.style={anchor=south west,inner sep=0pt},x=1mm, y=1mm]
\node at (4,4) {\includegraphics[width=60mm]{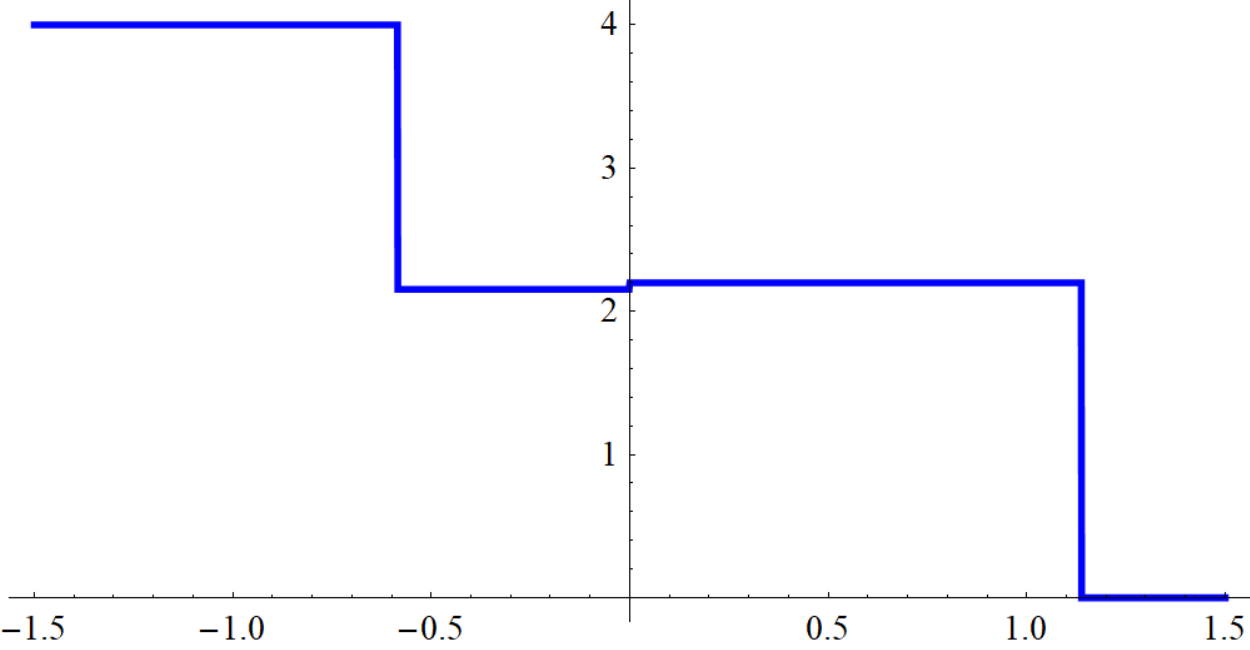}};
\node[below] at (30,0) {\strut $x\mapsto q(T_b^{q_*},x)$};
\end{tikzpicture}}
\caption*{\ref{case:B}}
\end{subfigure}
\hfill
\begin{subfigure}[t]{0.235\textwidth}
\centering
\resizebox{\linewidth}{!}{\begin{tikzpicture}[every node/.style={anchor=south west,inner sep=0pt},x=1mm, y=1mm]
\node at (4,4) {\includegraphics[width=60mm]{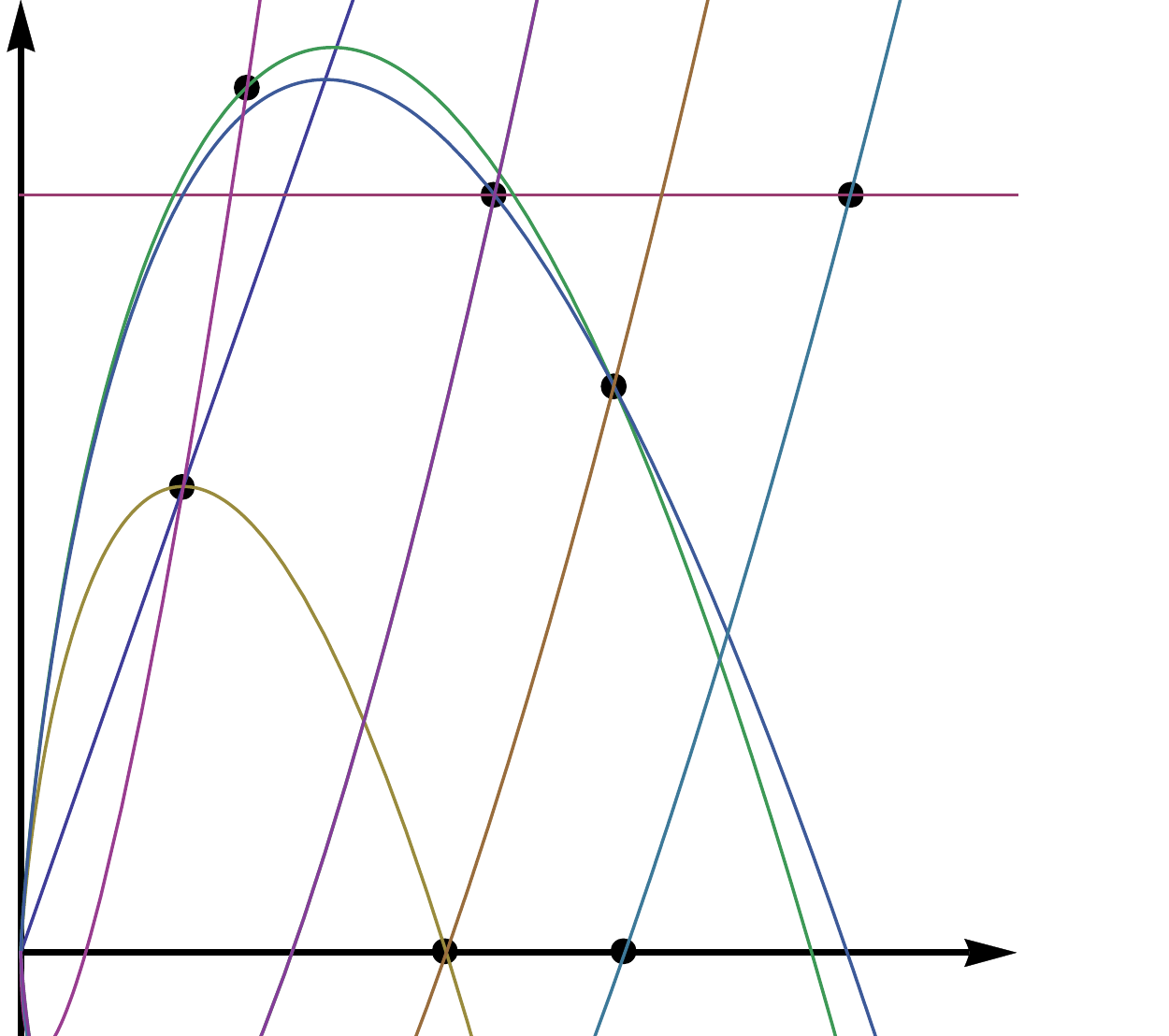}};
\node at (55,3) {\strut $\rho$};
\node at (37,3) {\strut $u_r$};
\node at (0,55) {\strut $q$};
\node at (0,45) {\strut $q_*$};
\node at (30.5,47) {\strut $\hat{u}_2$};
\node at (49,47) {\strut $\check{u}_1$};
\node at (28.5,8) {\strut $\hat{u}_1$};
\node at (14.5,31) {\strut $u_\ell$};
\node at (12,52) {\strut $u_i$};
\node at (37,35) {\strut $\tilde{u}_1$};
\node at (22,57) {\strut $v=a$};
\node at (13.5,57) {\strut $\mathsf{FL}_2^{u_i}$};
\node at (38,57) {\strut $\mathsf{BL}_2^{\tilde{u}_1}$};
\node at (48,57) {\strut $\mathsf{BL}_2^{u_r}$};
\node at (15,-1) {\strut $\mathsf{BL}_2^{\hat{u}_2}$};
\draw[latex-] (43.5,20) -- (50,25) node[right] {\strut $\mathsf{FL}_1^{\tilde{u}_1}$};
\draw[latex-] (43.8,15) -- (50,20) node[right] {\strut $\mathsf{FL}_1^{u_i}$};
\node at (26,-1) {\strut $\mathsf{FL}_1^{u_\ell}$};
\end{tikzpicture}}
\\
\resizebox{\linewidth}{!}{\begin{tikzpicture}[every node/.style={anchor=south west,inner sep=0pt},x=1mm, y=1mm]
\node at (4,4) {\includegraphics[width=60mm]{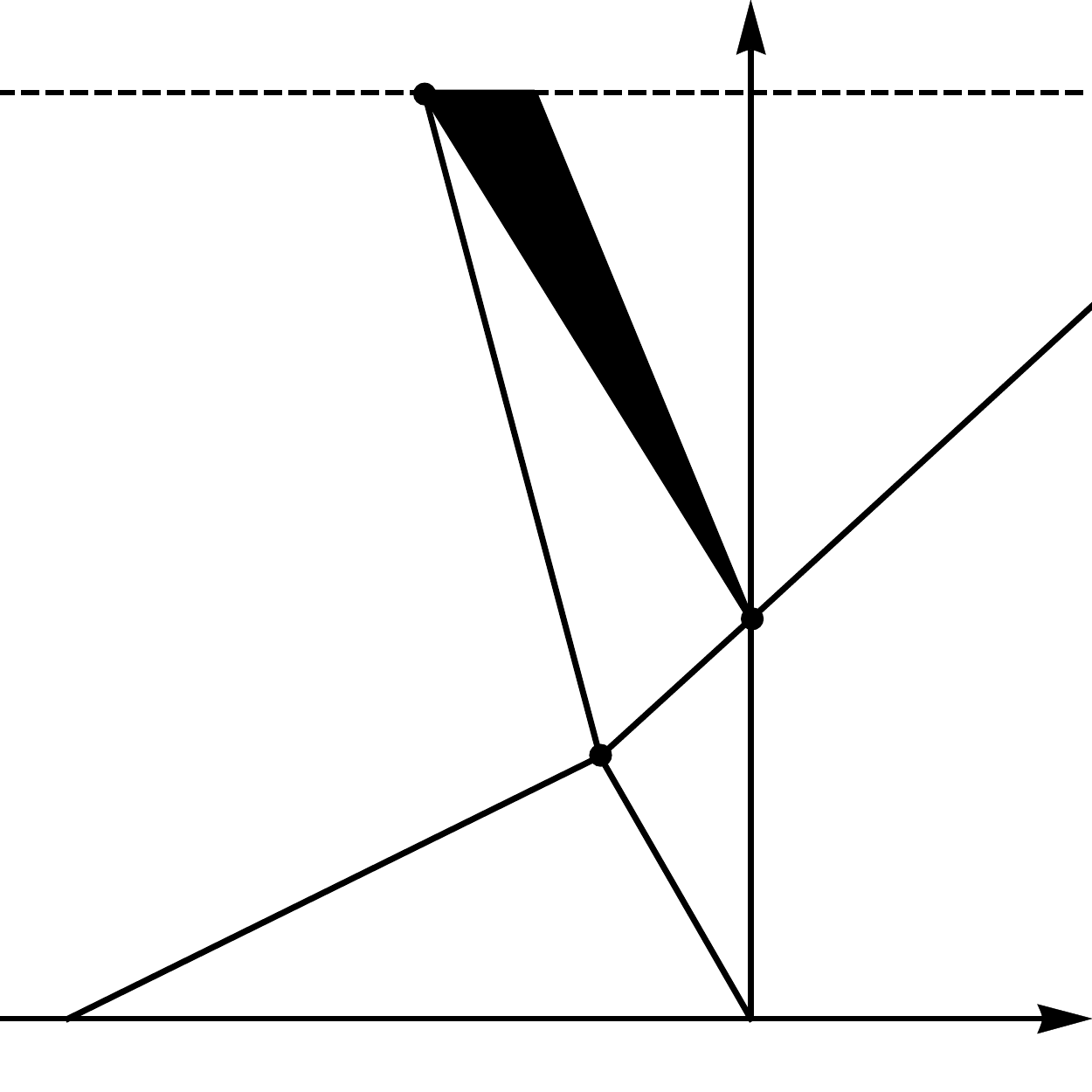}};
\node[right] at (4,61) {\strut $T_c^{q_*}$};
\node at (63,2) {\strut $x$};
\node at (5,2) {\strut $-1$};
\node at (60,13) {\strut $u_r$};
\node at (41,61) {\strut $t$};
\node at (40,50) {\strut $\hat{u}_2$};
\node at (52,50) {\strut $\check{u}_1$};
\node at (41,19) {\strut $\hat{u}_1$};
\node at (30,10) {\strut $u_\ell$};
\node at (12,38) {\strut $u_i$};
\node at (36,30) {\strut $\tilde{u}_1$};
\node at (32,22) {\strut $P_1$};
\node at (47,26) {\strut $P_2$};
\node at (23,54) {\strut $P_3$};
\end{tikzpicture}}
\\
\resizebox{\linewidth}{!}{\begin{tikzpicture}[every node/.style={anchor=south west,inner sep=0pt},x=1mm, y=1mm]
\node at (4,4) {\includegraphics[width=60mm]{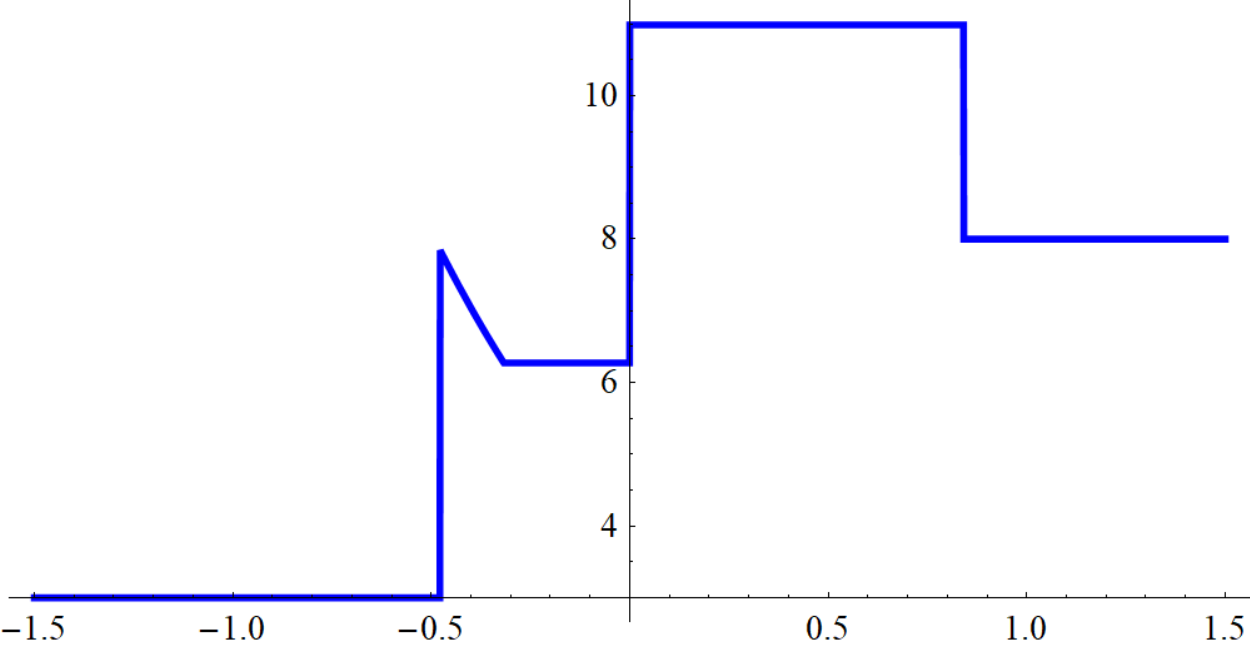}};
\node[below] at (30,0) {\strut $x\mapsto \rho(T_c^{q_*},x)$};
\end{tikzpicture}}
\\[10pt]
\resizebox{\linewidth}{!}{\begin{tikzpicture}[every node/.style={anchor=south west,inner sep=0pt},x=1mm, y=1mm]
\node at (4,4) {\includegraphics[width=60mm]{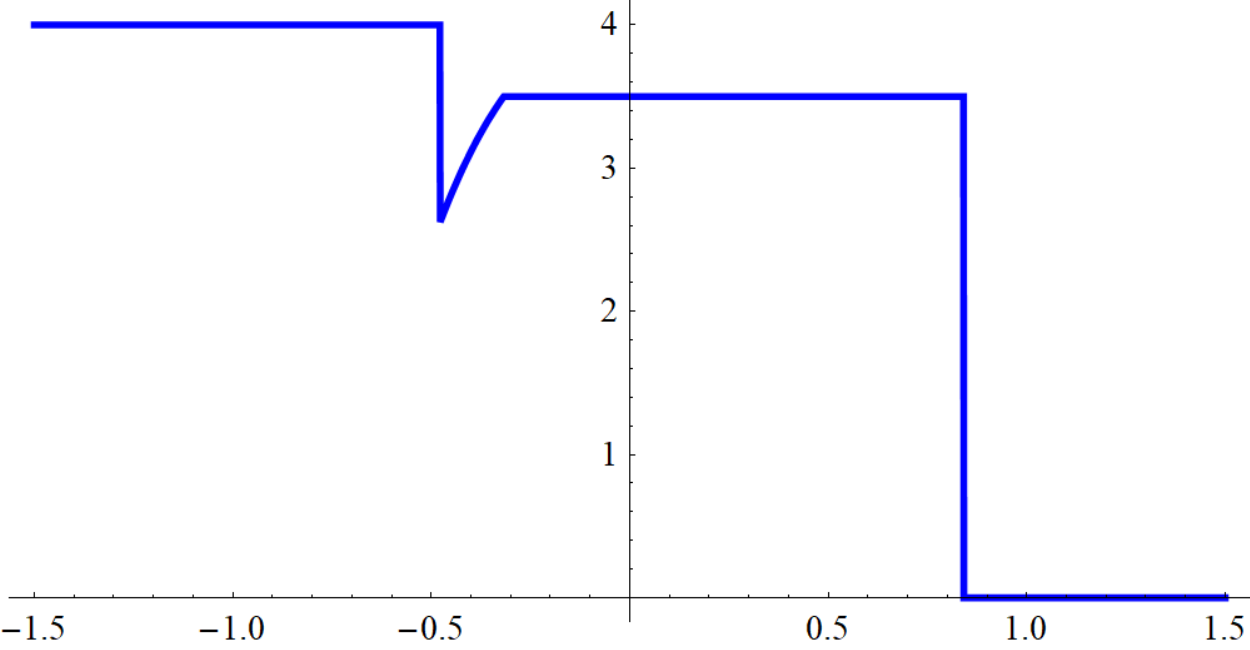}};
\node[below] at (30,0) {\strut $x\mapsto q(T_c^{q_*},x)$};
\end{tikzpicture}}
\caption*{\ref{case:C}}
\end{subfigure}
\hfill
\begin{subfigure}[t]{0.235\textwidth}
\centering
\resizebox{\linewidth}{!}{\begin{tikzpicture}[every node/.style={anchor=south west,inner sep=0pt},x=1mm, y=1mm]
\node at (4,4) {\includegraphics[width=60mm]{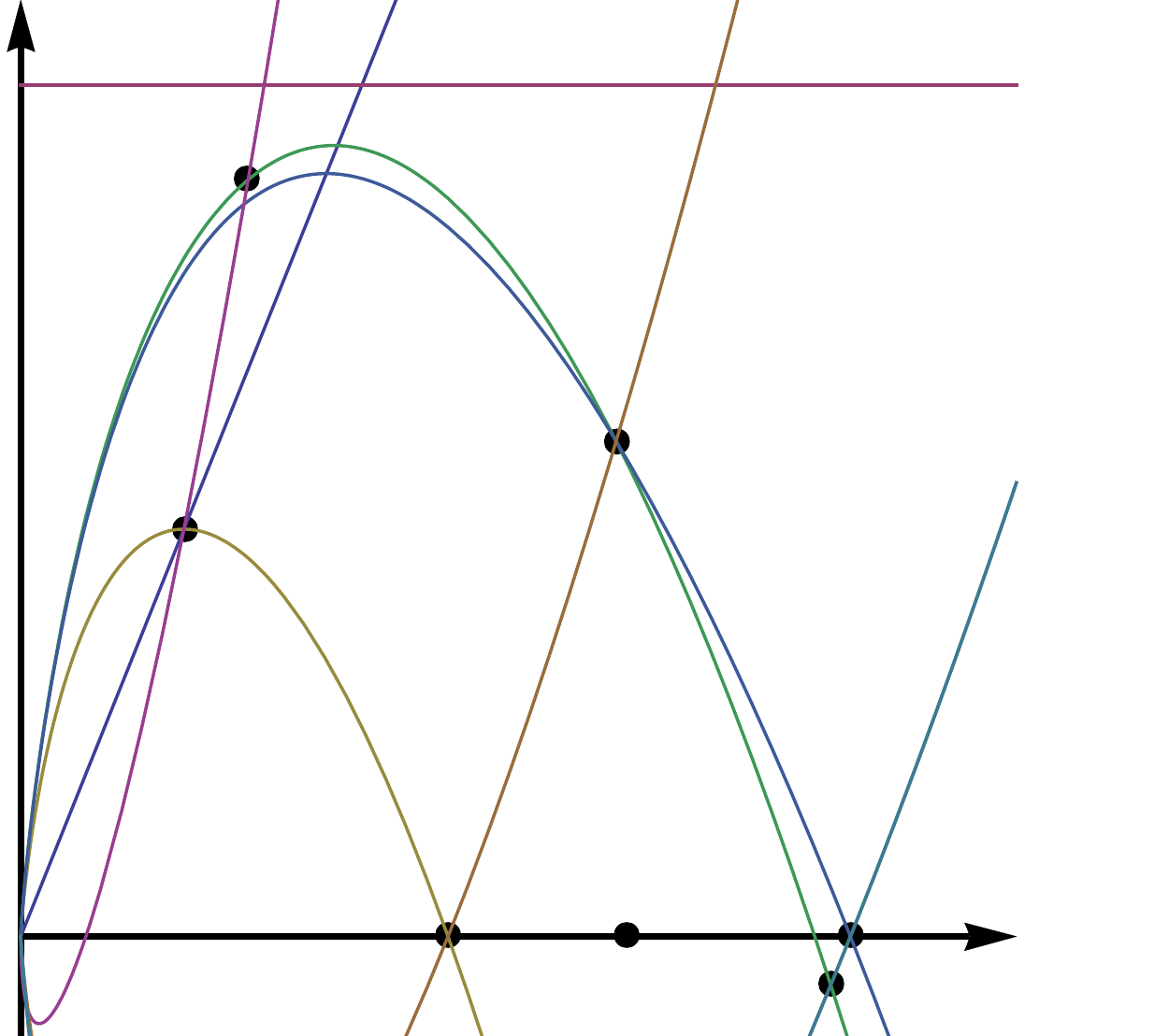}};
\node at (55,3) {\strut $\rho$};
\node at (42,4) {\strut $\tilde{u}_2$};
\node at (35,4) {\strut $u_r$};
\node at (0,55) {\strut $q$};
\node at (0,51) {\strut $q_*$};
\node at (49,8.5) {\strut $\hat{u}_2$};
\node at (22.5,8.5) {\strut $\hat{u}_1$};
\node at (15,28.5) {\strut $u_\ell$};
\node at (12,47) {\strut $u_i$};
\node at (37,32) {\strut $\tilde{u}_1$};
\node at (22,57) {\strut $v=a$};
\node at (14,57) {\strut $\mathsf{FL}_2^{u_i}$};
\node at (38,57) {\strut $\mathsf{BL}_2^{\hat{u}_1}$};
\node at (53,32) {\strut $\mathsf{BL}_2^{\hat{u}_2}$};
\draw[latex-] (43.5,20) -- (55,25) node[right] {\strut $\mathsf{FL}_1^{\tilde{u}_1}$};
\draw[latex-] (44.1,15) -- (55,20) node[right] {\strut $\mathsf{FL}_1^{u_i}$};
\node at (26,-1) {\strut $\mathsf{FL}_1^{u_\ell}$};
\end{tikzpicture}}
\\
\resizebox{\linewidth}{!}{\begin{tikzpicture}[every node/.style={anchor=south west,inner sep=0pt},x=1mm, y=1mm]
\node at (4,4) {\includegraphics[width=60mm]{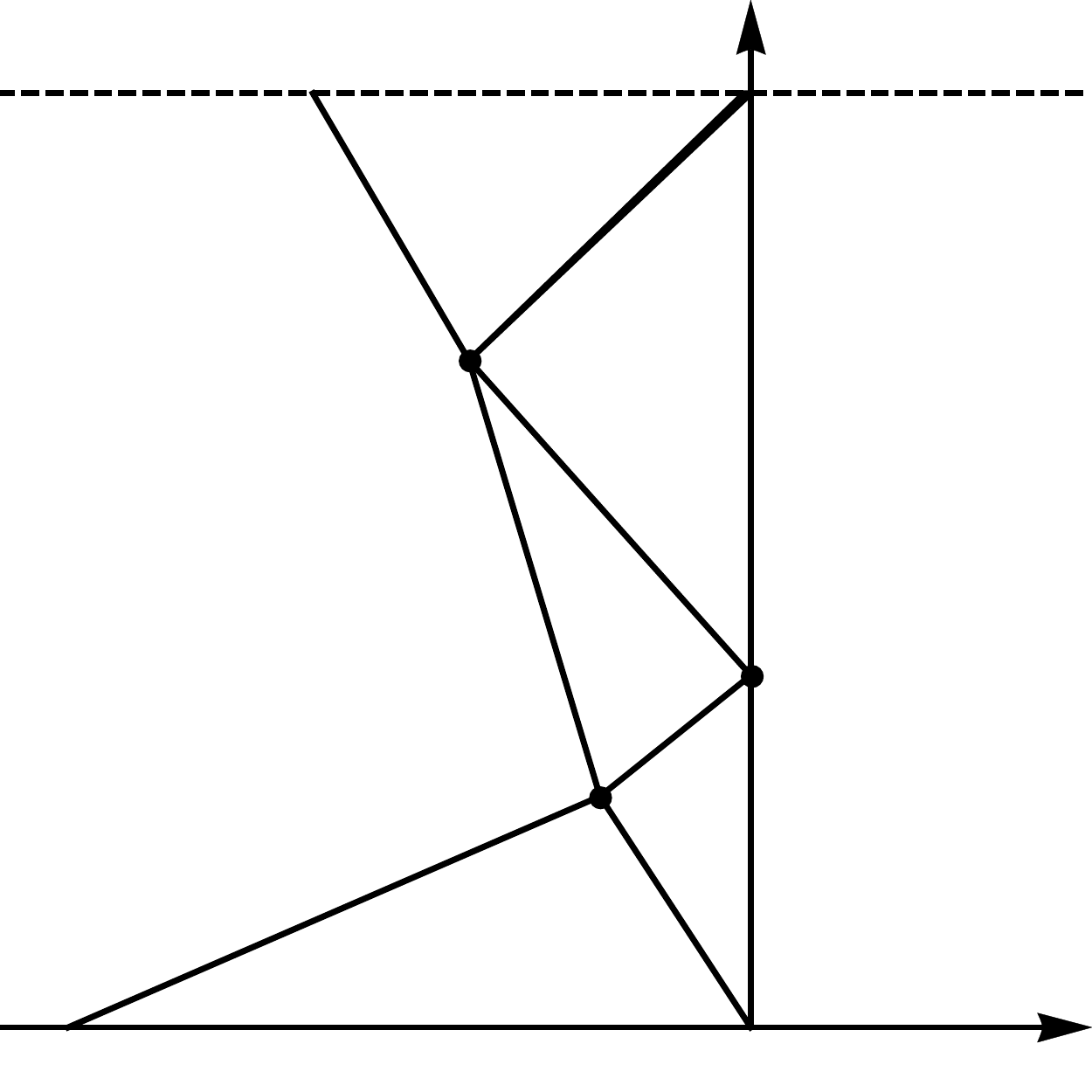}};
\node[right] at (4,61) {\strut $T_d^{q_*}$};
\node at (63,2) {\strut $x$};
\node at (5,2) {\strut $-1$};
\node at (29,52) {\strut $\tilde{u}_2$};
\node at (60,40) {\strut $u_r$};
\node at (41,61) {\strut $t$};
\node at (38,40) {\strut $\hat{u}_2$};
\node at (41,15) {\strut $\hat{u}_1$};
\node at (30,10) {\strut $u_\ell$};
\node at (12,38) {\strut $u_i$};
\node at (36,26) {\strut $\tilde{u}_1$};
\node at (32,19.5) {\strut $P_1$};
\node at (47,24) {\strut $P_2$};
\node at (24,42) {\strut $P_3$};
\end{tikzpicture}}
\\
\resizebox{\linewidth}{!}{\begin{tikzpicture}[every node/.style={anchor=south west,inner sep=0pt},x=1mm, y=1mm]
\node at (4,4) {\includegraphics[width=60mm]{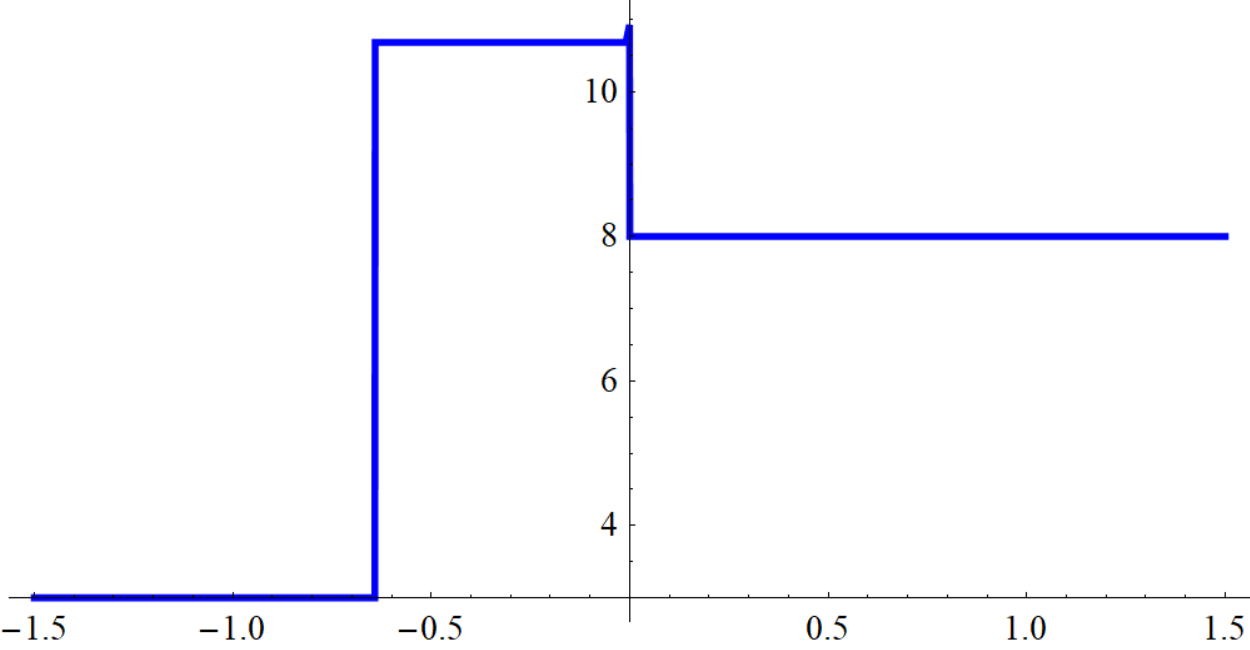}};
\node[below] at (30,0) {\strut $x\mapsto \rho(T_d^{q_*},x)$};
\end{tikzpicture}}
\\[10pt]
\resizebox{\linewidth}{!}{\begin{tikzpicture}[every node/.style={anchor=south west,inner sep=0pt},x=1mm, y=1mm]
\node at (4,4) {\includegraphics[width=60mm]{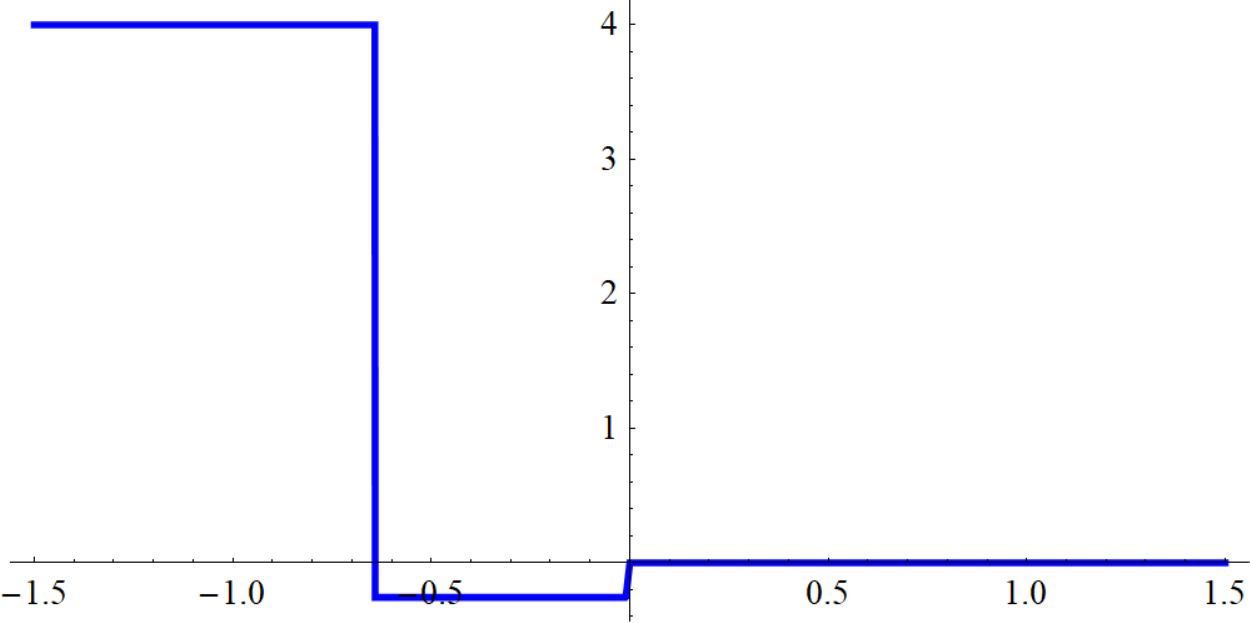}};
\node[below] at (30,0) {\strut $x\mapsto q(T_d^{q_*},x)$};
\end{tikzpicture}}
\caption*{\ref{case:D}}
\end{subfigure}
\caption{Cases considered in Subsection~\ref{subsub:las}. Notation is as in \eqref{e:ns} and in the text.
The values of the involved states are listed at the end of Subsection~\ref{subsub:las}.}
\label{f:optimization}
\end{figure}

\begin{enumerate}[label={\bf Case~(\alph*)}, wide=0pt]

\item\label{case:A}~\\
Assume $q_*\in[0,q_\ell]$. For notational simplicity, we denote
\begin{align}
\label{e:ns}
&\hat{u}_1 \doteq \hat{u}(q_*,u_\ell),& 
&\check{u}_1 \doteq \check{u}(q_*,u_r),& 
&\tilde{u}_1 \doteq \tilde{u}(u_i,\hat{u}_1),& 
&\hat{u}_2 \doteq \hat{u}(q_*,\tilde{u}_1),&
&\tilde{u}_2 \doteq \tilde{u}(u_i,\hat{u}_2).&
\end{align}
At time $t=0$ a $2$-shock with positive speed $s_{2}^{u_i}(\rho_\ell)$ starts from $x=-1$; a $1$-shock with negative speed $s_{1}^{u_\ell}(\hat{\rho}_1)$, a stationary non-classical shock and a $2$-shock with positive speed $s_{2}^{\check{u}_1}(\rho_r)$ are generated at $x=0$.
The first two shocks interact at time $t_1>0$ in $x=x_1<0$: a $1$-shock with speed $s_{1}^{u_i}(\tilde{\rho}_1)$ and a $2$-shock with positive speed $s_{2}^{\tilde{u}_1}(\hat{\rho}_1)$ are generated.
The latter shock eventually reaches $x=0$ at time $t_2$.
By applying $\rsh[\tilde{u}_1,\check{u}_1]$, we deduce that a $1$-shock with negative speed $s_{1}^{\tilde{u}_1}(\hat{\rho}_2)$ and a stationary non-classical shock start from $x=0$ at time $t_2$; the former shock eventually interacts at time $t_3$ and position $x_3$ with the $1$-shock that was generated at time $t_1$.
As a result of such interaction, a $1$-shock with negative speed $s_{1}^{u_i}(\tilde{\rho}_2)$ and a $2$-rarefaction with positive speeds ranging in $[ \lambda_2(\tilde{u}_2) , \lambda_2(\hat{u}_2) ]$ start from $x=x_3$ at time $t_3$.
Because a rarefaction showed up, we stop the construction as soon as it reaches the valve; we denote such a time by $T_a^{q_*}$. Notice that the low value of $q_*$ lets the valve open already at $t=0^+$ with flow $q_*$ on the right; the valve keeps open and the flow is $q_*$ at least until time $T_a^{q_*}$.

\item\label{case:B}~\\
Assume $q_*\in(q_\ell,\tilde{q}(u_i,\hat{u}(0,u_\ell))]$.
We still use notation \eqref{e:ns} with the exception of $\hat{u}_1 \doteq \hat{u}(0,u_\ell)$. 
As in the previous case, at time $t=0$ a $2$-shock with positive speed $s_{2}^{u_i}(\rho_\ell)$ starts from $x=-1$; however, because of the larger value of $q_*$, only a $1$-shock with negative speed $s_{1}^{u_\ell}(\hat{\rho}_1)$ and a stationary non-classical shock are generated at $x=0$, and the valve is closed. 
At time $t_1>0$ the two classical shocks interact at $x=x_1<0$: a $1$-shock with speed $s_{1}^{u_i}(\tilde{\rho}_1)$ and a $2$-shock with positive speed $s_{2}^{\tilde{u}_1}(\hat{\rho}_1)$ are generated.
The latter shock eventually reaches $x=0$ at time $t_2$.
By considering $\rsh[\tilde{u}_1,u_r]$, we deduce that a $1$-shock with negative speed $s_{1}^{\tilde{u}_1}(\hat{\rho}_2)$, a stationary non-classical shock and a $2$-shock with positive speed $s_{2}^{\check{u}_1}(\rho_r)$ leave $x=0$ at time $t_2$. Roughly speaking, the effect of the supersonic perturbation is not much damped by the shock from $x=0$ and opens the valve.
The new $1$-shock eventually interacts at time $t_3$ and position $x_3$ with the $1$-shock appeared at time $t_1$:
a $1$-shock with negative speed $s_{1}^{u_i}(\tilde{\rho}_2)$ and a $2$-rarefaction with positive speeds ranging in $[ \lambda_2(\tilde{u}_2) , \lambda_2(\hat{u}_2) ]$ are generated.
Then we stop the construction at time $T_b^{q_*}$ when the above $2$-rarefaction reaches $x=0$.

\item\label{case:C}~\\
Assume $q_*\in(\tilde{q}(u_i,\hat{u}(0,u_\ell)),\bar{q}(\tilde{u}(u_i,\hat{u}(0,u_\ell)))]$.
In this case and in the following one we omit some details, which are as in the two previous cases. Here, differently from \ref{case:B}, from $P_2$ a $1$-rarefaction appears on the left of $x=0$. This is a consequence of the higher value of $q_*$, which lets more flow pass to the right.
Then we stop the construction at time $T_c^{q_*}$ when such $1$-rarefaction interacts with the $1$-shock created at time $t_1$.

\item\label{case:D}~\\
Assume $q_*>\bar{q}(\tilde{u}(u_i,\hat{u}(0,u_\ell)))$.
The construction is analogous to that in \ref{case:B}.
The only differences are that $\hat{u}_2 \doteq \hat{u}(0,\tilde{u}_1)$ and until time $t_3$ the valve is closed, so that no waves appear in $x>0$. Notice that the very high value of $q_*$ lets the valve closed at least until time $T_d^{q_*}$ when a $2$-rarefaction reaches $x=0$.

\end{enumerate}
\medskip

Notice now that in \ref{case:A} the points $P_1$, $P_2$, $P_3$ coalesce to the point $P_*$ as $q_* \nearrow q_\ell$, where $P_*=\left(0,s_{2}^{u_i}(\rho_\ell)^{-1}\right)$ is the intersection of the $t$-axis and the line passing through the point $(-1,0)$ with slope $s_{2}^{u_i}(\rho_\ell)^{-1}$, that is the line passing through $(-1,0)$ and $P_1$ in \figurename~\ref{f:optimization}. Hence, by comparing the above constructions, see again \figurename~\ref{f:optimization}, it is now clear that the smallest time $T$ which allows an explicit construction of the solution for any $q_*\geq0$ is precisely
\begin{equation}\label{e:T}
T = T_{\min} \doteq \min\left\{ T_a, T_b, T_c, T_d \right\} = s_{2}^{u_i}(\rho_\ell)^{-1}.
\end{equation}
where
\begin{align*}
T_a &\doteq \min\left\{ T_a^{q_*} : q_*\in[0,q_\ell] \right\},&
T_b &\doteq \inf\left\{ T_b^{q_*} : q_*\in(q_\ell,\tilde{q}(u_i,\hat{u}(0,u_\ell))] \right\},
\\
T_c &\doteq \inf\left\{ T_c^{q_*} :  q_*\in(\tilde{q}(u_i,\hat{u}(0,u_\ell)),\bar{q}(\tilde{u}(u_i,\hat{u}(0,u_\ell)))] \right\},&
T_d &\doteq \inf\left\{ T_d^{q_*} : q_*>\bar{q}(\tilde{u}(u_i,\hat{u}(0,u_\ell))) \right\}.
\end{align*}
Then it is easy to see that 
\[
\max_{q^*>0}\mathfrak{Q}(q_*,T) = \overline{Q}(u_\ell)=q_\ell,
\]
and the unique maximizer is $q_*=\overline{Q}(u_\ell) = q_\ell$. In other words, the choice of reducing the maximization process only to times prior to the first interaction involving a rarefaction leads to the same result of Subsection \ref{ss:1} for the Riemann problem, even if the construction is different.

In Subsections~\ref{sub:n1} and~\ref{sub:n2} we numerically investigate two cases: $u_i\in\mathsf{CH}_{\ell,1}$ and $u_i\in\mathsf{CH}_{\ell,1}^{\scriptscriptstyle\complement}$, respectively.
The differences between these two cases are highlighted by comparison of \figurename s~\ref{f:opt-T=2} and~\ref{f:opt-T=3}.

\subsubsection{A numerical solution of the maximization problem in the case \texorpdfstring{$v_i<v_*^{\sup}$}{}}
\label{sub:n1}

It is not easy to tackle the maximization problem \eqref{e:opt1} from an analytic point of view, even under condition \eqref{e:conditions} and for short times. We provide instead a numerical simulation.

We begin by plotting the numerical solutions of Subsection \ref{subsub:las} for specific values. 
The states in \figurename~\ref{f:optimization00} and the exact solutions constructed in \figurename~\ref{f:optimization} are represented below by taking
\begin{align}\label{e:data02-A}
&a=1,&
&\rho_i=3,&
&q_i=4,&
&\rho_r=8,&
&q_r=0,
\end{align}
and the following values of $q_*$ for the corresponding cases
\begin{align}\label{e:qs}
&q_*^a=0.2,&
&q_*^b=2.2,&
&q_*^c=3.5,&
&q_*^d=4.5.
\end{align}
Notice that $v_i = q_i/\rho_i=4/3\approx1.33>1 = a$, and then $v_i$ is supersonic; however $v_i<1.63=v_*^{\sup}$. Notice moreover that the above construction and the choice in \eqref{e:data02-A} lead to 
\begin{align}\label{e:data02-B}
\rho_\ell &\approx 2.15,&
v_\ell&=1,&
q_\ell&\approx 2.15,
\\\label{e:data02-C}
\hat{\rho}(0,u_\ell) &\approx 5.64,&
v_i&=4/3,&
\tilde{q}\left(u_i,\hat{u}(0,u_\ell)\right) &\approx 2.62,
\\\label{e:data02-D}
\tilde{\rho}\left(u_i,\hat{u}(0,u_\ell)\right) &\approx 7.85,&
v_r&=0,&
\bar{q}\left( \tilde{u}\left(u_i,\hat{u}(0,u_\ell)\right) \right) &\approx 4.03,
\end{align}
while by \eqref{e:qs} and \eqref{e:T} we have
\begin{align}\label{e:data02-E}
&T_a^{q_*^a}\approx 1.37,&
&T_b^{q_*^b}\approx 1.56,&
&T_c^{q_*^c}\approx 1.26,&
&T_d^{q_*^d}\approx 1.44,&
&T_{\min} \approx 0.46.
\end{align}
In particular, the conditions listed in \eqref{e:conditions} are satisfied, see \figurename~\ref{f:optimization00}, and $q_*^a < q_\ell < q_*^b < \tilde{q}(u_i,\hat{u}(0,u_\ell)) < q_*^c < \bar{q}( \tilde{u}(u_i,\hat{u}(0,u_\ell)) ) < q_*^d$.
In \figurename~\ref{f:profiles} we show the outputs of our numerical simulations, which highlight a very good match with the exact solution and confirm the validity of the numerical scheme. Notice, both in \ref{case:A} and \ref{case:D}, the persistence of a negative left flow from the valve, as it was indeed forecast by the top pictures in \figurename~\ref{f:optimization}.
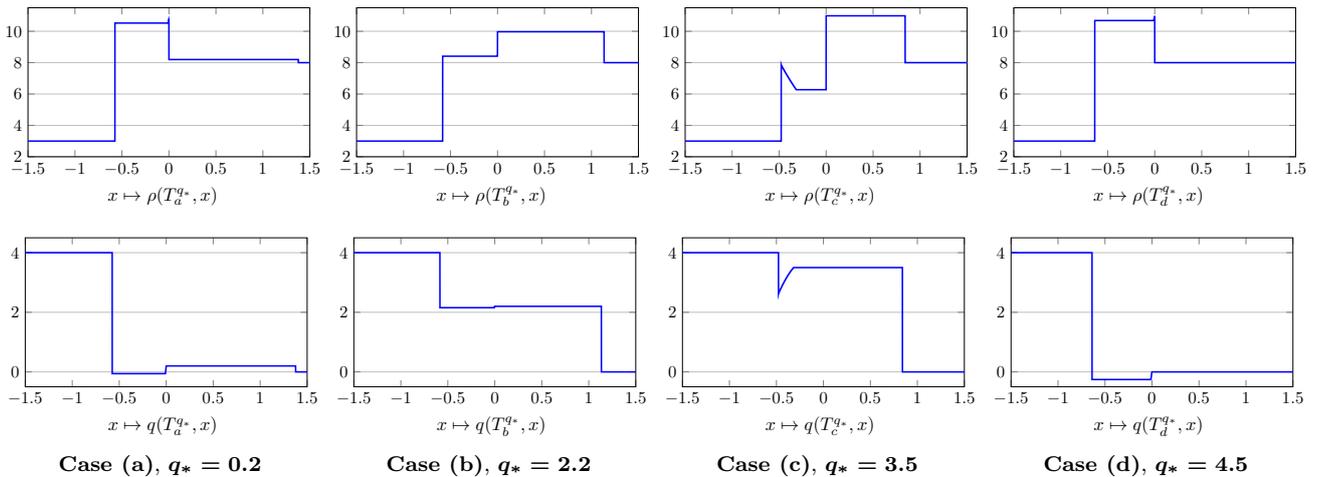
\begin{figure}[!htb]\centering
\begin{subfigure}[t]{0.235\textwidth}
\centering
\resizebox{\linewidth}{!}{
\begin{tikzpicture}[every node/.style={anchor=south west,inner sep=0pt},x=1mm, y=1mm]
\node at (0,0) {\includegraphics[width=60mm,trim=72mm 193mm 72.5mm 45.5mm, clip]{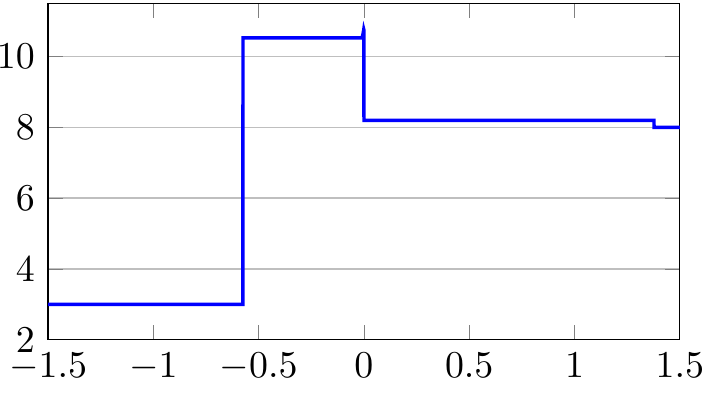}};
\node[below] at (30,0) {\strut $x\mapsto \rho(T_a^{q_*},x)$};
\end{tikzpicture}}
\\[10pt]
\resizebox{\linewidth}{!}{
\begin{tikzpicture}[every node/.style={anchor=south west,inner sep=0pt},x=1mm, y=1mm]
\node at (0,0) {\includegraphics[width=60mm,trim=72mm 193mm 72.5mm 45.5mm, clip]{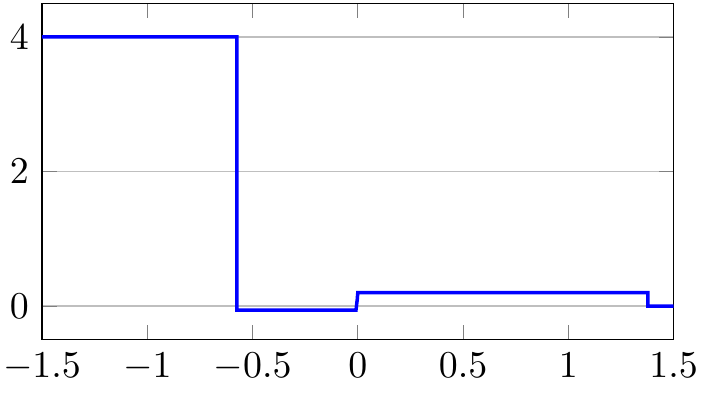}};
\node[below] at (30,0) {\strut $x\mapsto q(T_a^{q_*},x)$};
\end{tikzpicture}}
\caption*{\ref{case:A}, $\boldsymbol{q_*=0.2}$}
\end{subfigure}
\begin{subfigure}[t]{0.235\textwidth}
\centering
\resizebox{\linewidth}{!}{
\begin{tikzpicture}[every node/.style={anchor=south west,inner sep=0pt},x=1mm, y=1mm]
\node at (0,0) {\includegraphics[width=60mm,trim=72mm 193mm 72.5mm 45.5mm, clip]{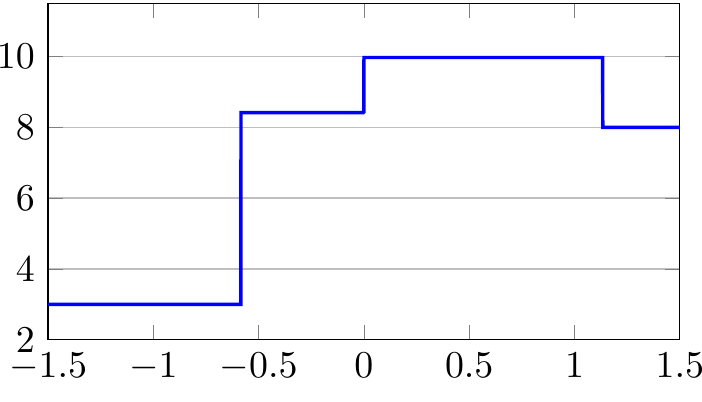}};
\node[below] at (30,0) {\strut $x\mapsto \rho(T_b^{q_*},x)$};
\end{tikzpicture}}
\\[10pt]
\resizebox{\linewidth}{!}{
\begin{tikzpicture}[every node/.style={anchor=south west,inner sep=0pt},x=1mm, y=1mm]
\node at (0,0) {\includegraphics[width=60mm,trim=72mm 193mm 72.5mm 45.5mm, clip]{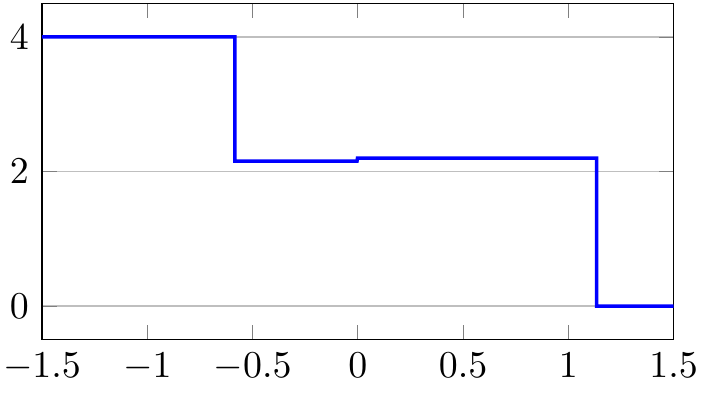}};
\node[below] at (30,0) {\strut $x\mapsto q(T_b^{q_*},x)$};
\end{tikzpicture}}
\caption*{\ref{case:B}, $\boldsymbol{q_*=2.2}$}
\end{subfigure}
\begin{subfigure}[t]{0.235\textwidth}
\centering
\resizebox{\linewidth}{!}{
\begin{tikzpicture}[every node/.style={anchor=south west,inner sep=0pt},x=1mm, y=1mm]
\node at (0,0) {\includegraphics[width=60mm,trim=72mm 193mm 72.5mm 45.5mm, clip]{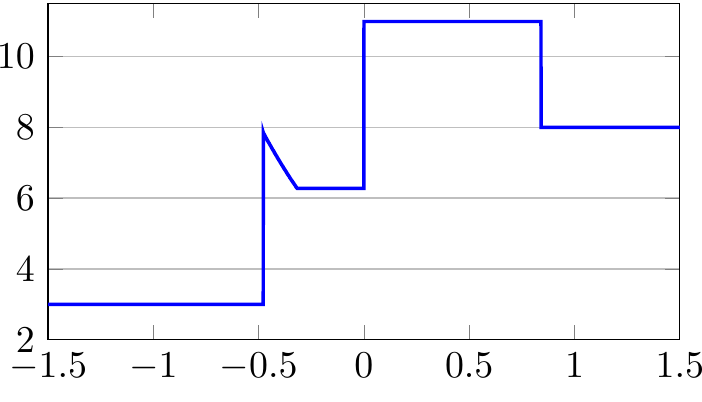}};
\node[below] at (30,0) {\strut $x\mapsto \rho(T_c^{q_*},x)$};
\end{tikzpicture}}
\\[10pt]
\resizebox{\linewidth}{!}{
\begin{tikzpicture}[every node/.style={anchor=south west,inner sep=0pt},x=1mm, y=1mm]
\node at (0,0) {\includegraphics[width=60mm,trim=72mm 193mm 72.5mm 45.5mm, clip]{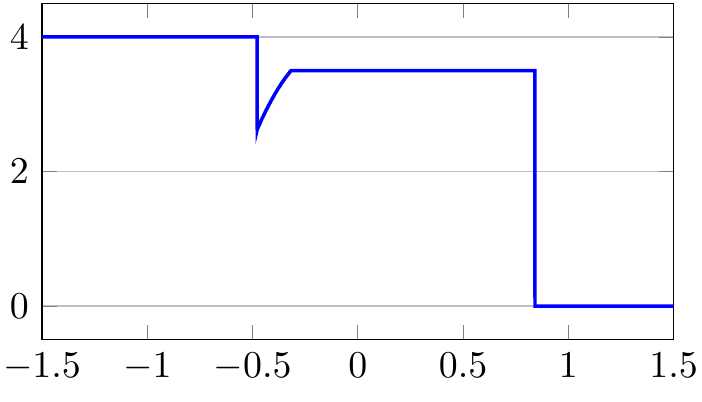}};
\node[below] at (30,0) {\strut $x\mapsto q(T_c^{q_*},x)$};
\end{tikzpicture}}
\caption*{\ref{case:C}, $\boldsymbol{q_*=3.5}$}
\end{subfigure}
\begin{subfigure}[t]{0.235\textwidth}
\centering
\resizebox{\linewidth}{!}{
\begin{tikzpicture}[every node/.style={anchor=south west,inner sep=0pt},x=1mm, y=1mm]
\node at (0,0) {\includegraphics[width=60mm,trim=72mm 193mm 72.5mm 45.5mm, clip]{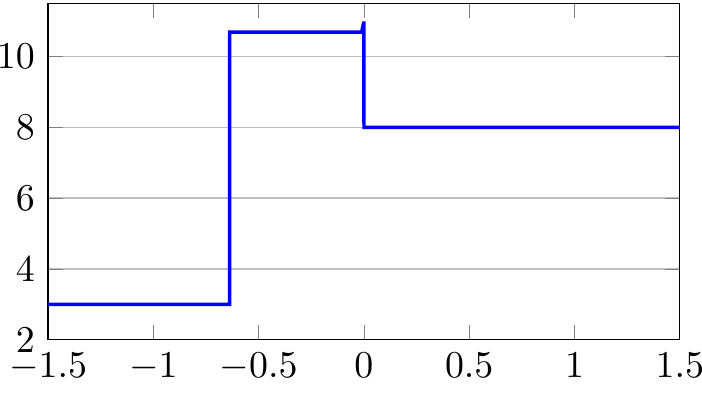}};
\node[below] at (30,0) {\strut $x\mapsto \rho(T_d^{q_*},x)$};
\end{tikzpicture}}
\\[10pt]
\resizebox{\linewidth}{!}{
\begin{tikzpicture}[every node/.style={anchor=south west,inner sep=0pt},x=1mm, y=1mm]
\node at (0,0) {\includegraphics[width=60mm,trim=72mm 193mm 72.5mm 45.5mm, clip]{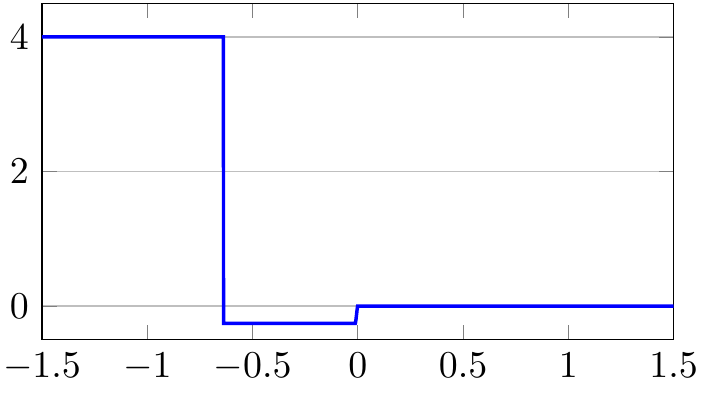}};
\node[below] at (30,0) {\strut $x\mapsto q(T_d^{q_*},x)$};
\end{tikzpicture}}
\caption*{\ref{case:D}, $\boldsymbol{q_*=4.5}$}
\end{subfigure}
\caption{Numerical simulations of the cases in Subsection~\ref{subsub:las} and corresponding to the values listed in \eqref{e:data02-A}, \eqref{e:data02-B}. 
The spikes in the $\rho$-profiles for the first and last case appear also in the exact solutions and correspond to a $2$-rarefaction followed by a non-classical stationary shock at $x=0$.}
\label{f:profiles}
\end{figure}

In \figurename~\ref{f:opt-T=2}, we show the numerical result obtained with the same values as in \eqref{e:data02-A}, \eqref{e:data02-B} and $T=2>T_{\min}$, see \eqref{e:data02-E}.
Recall that even at time $T=2$ an exact expression of the solution is not easily available.
\begin{figure}[!htb]\centering
\begin{subfigure}[t]{43mm}
\centering
\begin{tikzpicture}[every node/.style={anchor=south west,inner sep=0pt},x=1mm, y=1mm]
\node at (0,0) {\includegraphics[width=43mm,trim=73.5mm 195.5mm 74mm 45.5mm, clip]{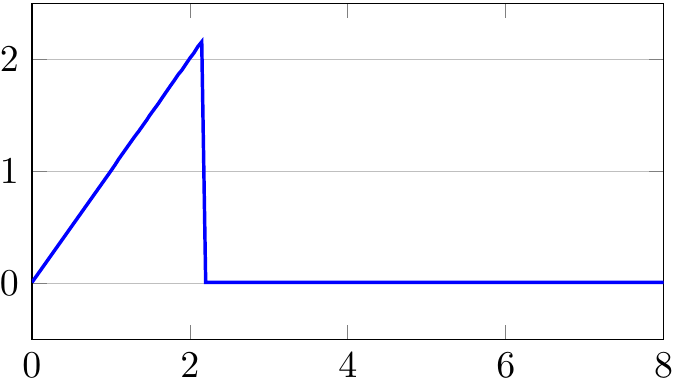}};
\end{tikzpicture}
\caption*{$q_*\mapsto\mathfrak{Q}(q_*,0.5)$}
\end{subfigure}
\hspace{5mm}
\begin{subfigure}[t]{43mm}
\centering
\begin{tikzpicture}[every node/.style={anchor=south west,inner sep=0pt},x=1mm, y=1mm]
\node at (0,0) {\includegraphics[width=43mm,trim=73.5mm 195.5mm 74mm 45.5mm, clip]{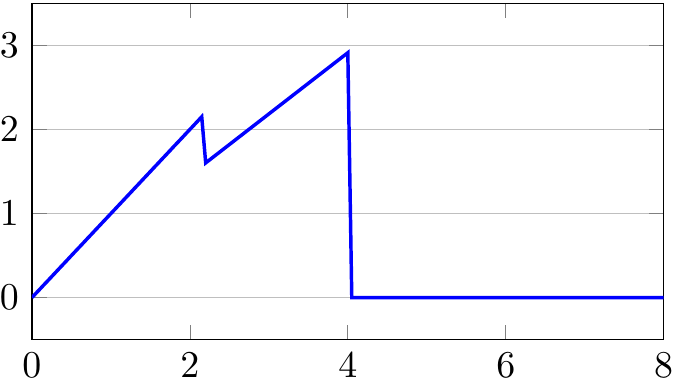}};
\end{tikzpicture}
\caption*{$q_*\mapsto\mathfrak{Q}(q_*,2)$}
\end{subfigure}
\hspace{5mm}
\begin{subfigure}[t]{43mm}
\centering
\begin{tikzpicture}[every node/.style={anchor=south west,inner sep=0pt},x=1mm, y=1mm]
\node at (0,0) {\includegraphics[width=43mm,trim=73.5mm 195.5mm 74mm 45.5mm, clip]{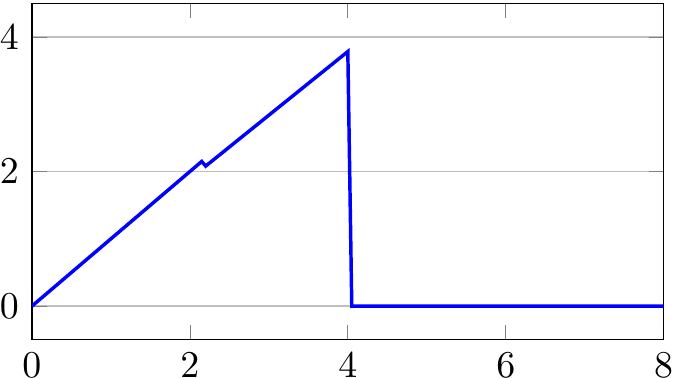}};
\end{tikzpicture}
\caption*{$q_*\mapsto\mathfrak{Q}(q_*,10)$}
\end{subfigure}
\caption{Numerical simulations corresponding to the values in \eqref{e:data02-A}, \eqref{e:data02-B} and for different values of $T$.}
\label{f:opt-T=2}
\end{figure}
We notice, in \figurename~\ref{f:opt-T=2}, that the function $q_*\mapsto \mathfrak{Q}(q_*,T)$ has up to two discontinuities, which can be interpreted as follows:
\begin{itemize}
\item if $q_*\in[0,q_\ell]$, then for any $t\geq0$ we have $u^{q_*}(t,0^-) \in \Omega \setminus \mathsf{C}_\ell^{q_*}$ and therefore $q^{q_*}(t,0) = q_*$;
\item if $q_*\in(q_\ell,\bar{q}( \tilde{u}(u_i,\hat{u}(0,u_\ell)) )]$, then for any $t\in[0,t_2)$ we have $u^{q_*}(t,0^-) \in \mathsf{C}_\ell^{q_*} \cap \mathsf{CH}_\ell^{q_*}$ and therefore $q^{q_*}(t,0) = 0$, whereas for any $t\geq t_2$ we have $u^{q_*}(t,0^-) \in \Omega \setminus \mathsf{C}_\ell^{q_*}$ and therefore $q^{q_*}(t,0) = q_*$;
\item if $q_*>\bar{q}( \tilde{u}(u_i,\hat{u}(0,u_\ell)) )$, then for any $t\geq0$ we have $u^{q_*}(t,0^-) \in \mathsf{C}_\ell^{q_*} \cap \mathsf{CH}_\ell^{q_*}$ and therefore $q^{q_*}(t,0) = 0$.
\end{itemize}
As a further check of the simulations, we plotted in \figurename~\ref{f:opt-phapla} the numerical traces $u_\Delta^{q_*}(t,0^-)$, $t\in(0,T]$, for the four different values of $q_*$ listed in \eqref{e:qs}.
\begin{figure}[!htb]\centering
\begin{subfigure}[t]{55mm}
\begin{tikzpicture}[every node/.style={anchor=south west,inner sep=0pt},x=1mm, y=1mm]
\node at (0,0) {\includegraphics[width=55mm,trim=64mm 158mm 63mm 47mm, clip]{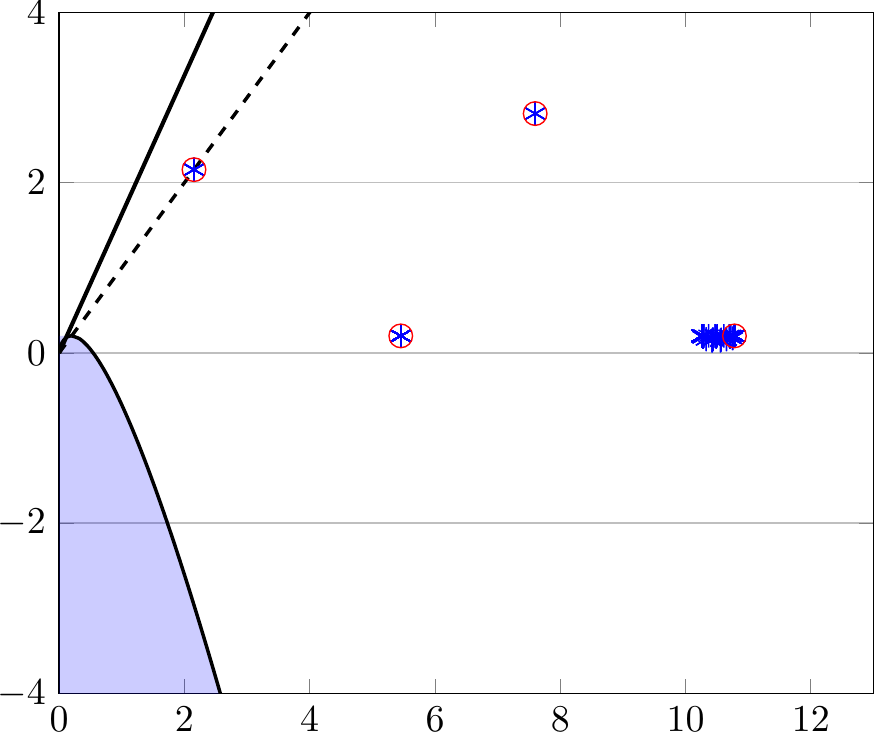}};
\node at (14,34) {\strut$u_\ell$};
\node at (27,23.5) {\strut$\hat{u}_1$};
\node at (35.5,37) {\strut$\tilde{u}_1$};
\node at (48,23.5) {\strut$\hat{u}_2$};
\end{tikzpicture}
\caption*{\ref{case:A}, $\boldsymbol{q_* = 0.2}$}
\end{subfigure}
\hspace{5mm}
\begin{subfigure}[t]{55mm}
\centering
\begin{tikzpicture}[every node/.style={anchor=south west,inner sep=0pt},x=1mm, y=1mm]
\node at (0,0) {\includegraphics[width=55mm,trim=64mm 158mm 63mm 47mm, clip]{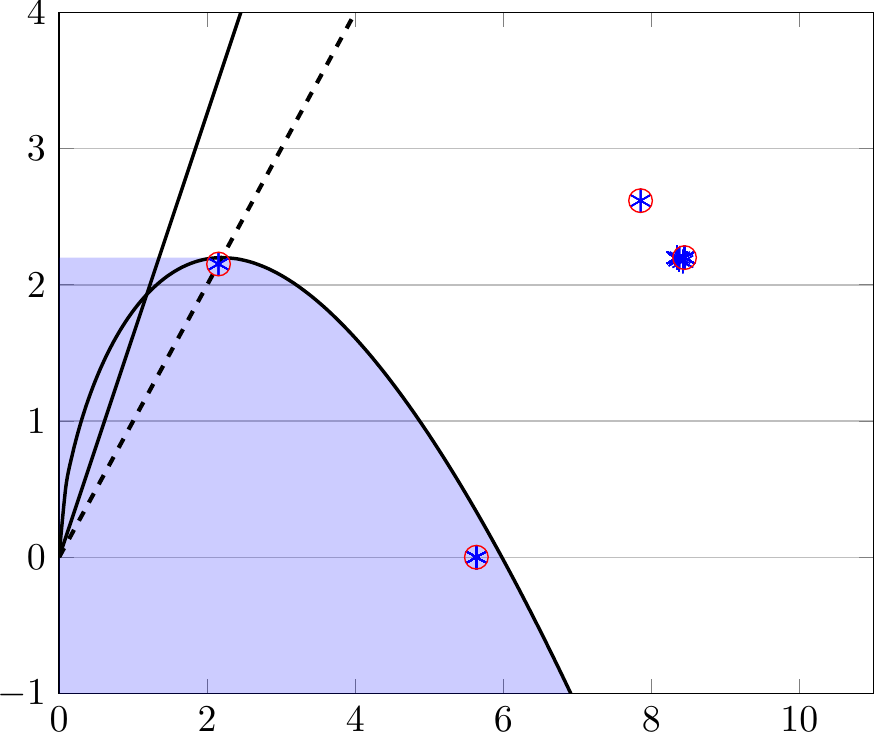}};
\node at (14.5,25.5) {\strut$u_\ell$};
\node at (25,10.5) {\strut$\hat{u}_1$};
\node at (35.5,32) {\strut$\tilde{u}_1$};
\node at (45,28) {\strut$\hat{u}_2$};
\end{tikzpicture}
\caption*{\ref{case:B}, $\boldsymbol{q_* = 2.2}$}
\end{subfigure}
\\[5pt]
\begin{subfigure}[t]{55mm}
\centering
\begin{tikzpicture}[every node/.style={anchor=south west,inner sep=0pt},x=1mm, y=1mm]
\node at (0,0) {\includegraphics[width=55mm,trim=64mm 158mm 63mm 47mm, clip]{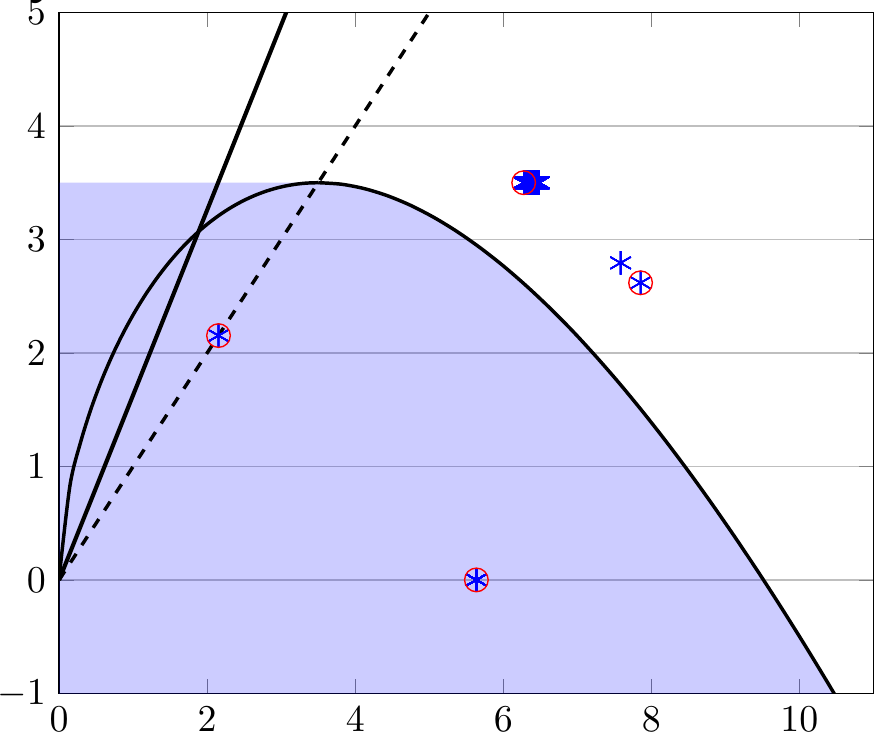}};
\node at (15,23.5) {\strut$u_\ell$};
\node at (25.5,9) {\strut$\hat{u}_1$};
\node at (28.5,32) {\strut$\tilde{u}_1$};
\node at (41.5,26) {\strut$\hat{u}_2$};
\end{tikzpicture}
\caption*{\ref{case:C}, $\boldsymbol{q_* = 3.5}$}
\end{subfigure}
\hspace{5mm}
\begin{subfigure}[t]{55mm}
\centering
\begin{tikzpicture}[every node/.style={anchor=south west,inner sep=0pt},x=1mm, y=1mm]
\node at (0,0) {\includegraphics[width=55mm,trim=64mm 158mm 63mm 47mm, clip]{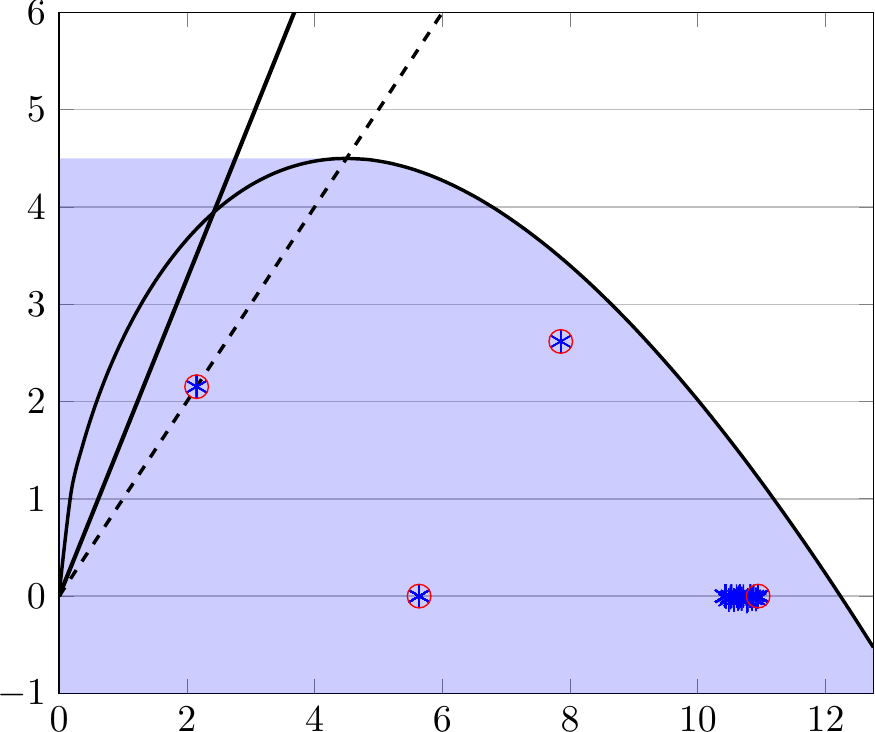}};
\node at (14,20) {\strut$u_\ell$};
\node at (25,9) {\strut$\hat{u}_1$};
\node at (31,22) {\strut$\tilde{u}_1$};
\node at (46,9) {\strut$\hat{u}_2$};
\end{tikzpicture}
\caption*{\ref{case:D}, $\boldsymbol{q_* = 4.5}$}
\end{subfigure}
\caption{Representation in the $(\rho,q)$-phase plane of the left traces at $x=0$ of the solutions constructed in Subsections~\ref{subsub:las} and~\ref{sub:n1}, and corresponding to the values listed in \eqref{e:data02-A}, \eqref{e:data02-B}.
The stars correspond to the values obtained numerically for $t\in(0,2]$, while the circles correspond to the exact values for $t$ respectively in $[0,T_a^{q_*}]$, $[0,T_b^{q_*}]$, $[0,T_c^{q_*}]$ and $[0,T_d^{q_*}]$.
The curves correspond to $\mathsf{BL}_1^{u_*^a}$, the shaded regions to $\mathsf{C}_\ell^{q_*}$, the solid lines to $q=v_*^{\sup}\,\rho$ and the dashed lines to $q=a\,\rho$, see \eqref{e:usa}, \eqref{e:CL} and \eqref{e:DevinTownsend}.
Notice that the traces $u_\ell$ and $\tilde{u}_1$ are attained only at times $t=0$ and $t=t_2$, respectively.}
\label{f:opt-phapla}
\end{figure}
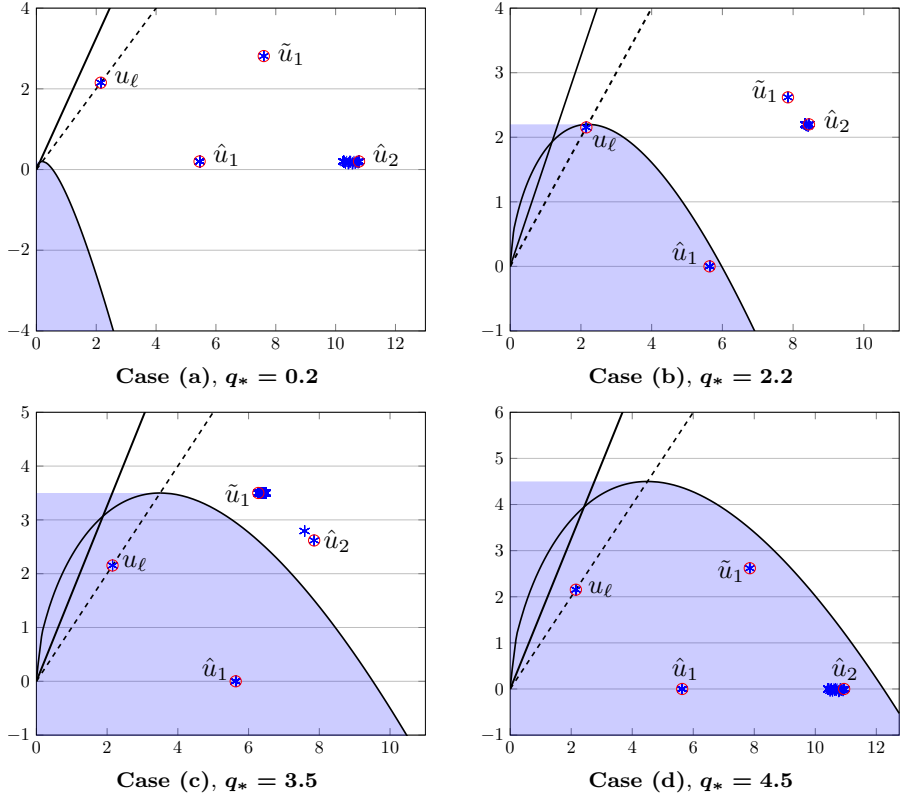
As a consequence, for $T\in[0,2]$ the solution of the maximization problem \eqref{e:opt1} is
\begin{equation}\label{e:rsl}
\mathfrak{Q}(q_*,T) =
\left\{\begin{array}{@{}l@{\quad\hbox{for }}l@{}}
q_* &q_*\in[0,q_\ell] \hbox{ and } T>0,
\\
0&q_*\in(q_\ell,\bar{q}( \tilde{u}(u_i,\hat{u}(0,u_\ell)) )] \hbox{ and } T\in[0,t_2],
\\
\frac{T-t_2}{T} \, q_* &q_*\in(q_\ell,\bar{q}( \tilde{u}(u_i,\hat{u}(0,u_\ell)) )] \hbox{ and } T>t_2,
\\
0&q_* > \bar{q}( \tilde{u}(u_i,\hat{u}(0,u_\ell)) ) \hbox{ and } T>0,
\end{array}\right.
\end{equation}
where $t_2\approx 0.54$ corresponds to $P_2$ in \figurename~\ref{f:optimization}, \ref{case:B} and \ref{case:C}.
We observe that \figurename~\ref{f:opt-T=2} highlights a good match with the expression in \eqref{e:rsl} of $\mathfrak{Q}(q_*,2)$, see \eqref{e:data02-B}, \eqref{e:data02-C} and \eqref{e:data02-D}.

\subsubsection{A numerical solution of the maximization problem in the case \texorpdfstring{$v_i>v_*^{\sup}$}{}}
\label{sub:n2}

In this subsection we consider the case when $v_i > v_*^{\sup} \approx 1.63 \cdot a$. The analytic construction is similar to the one performed in Subsection \ref{subsub:las} (which mainly aimed at checking the validity of the numerical scheme), so we do not repeat it. We guess that $u_i$ will reach the valve for $T$ sufficiently large. In this case \figurename~\ref{f:opt-T=2} will be different/richer, see \figurename~\ref{f:opt-T=3}, and we can comment it and point out the new features.
\begin{figure}[!htb]\centering
\begin{subfigure}[t]{43mm}
\centering
\begin{tikzpicture}[every node/.style={anchor=south west,inner sep=0pt},x=1mm, y=1mm]
\node at (0,0) {\includegraphics[width=43mm,trim=74.5mm 195.5mm 73.5mm 45.5mm, clip]{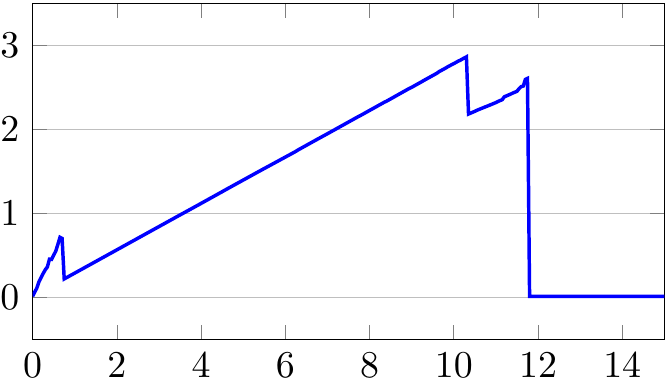}};
\end{tikzpicture}
\caption*{$q_*\mapsto\mathfrak{Q}(q_*,0.5)$}
\end{subfigure}
\hspace{5mm}
\begin{subfigure}[t]{43mm}
\centering
\begin{tikzpicture}[every node/.style={anchor=south west,inner sep=0pt},x=1mm, y=1mm]
\node at (0,0) {\includegraphics[width=43mm,trim=74.5mm 195.5mm 73.5mm 45.5mm, clip]{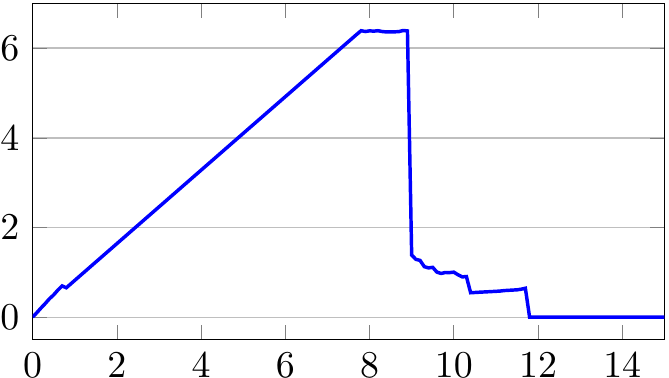}};
\end{tikzpicture}
\caption*{$q_*\mapsto\mathfrak{Q}(q_*,2)$}
\end{subfigure}
\hspace{5mm}
\begin{subfigure}[t]{43mm}
\centering
\begin{tikzpicture}[every node/.style={anchor=south west,inner sep=0pt},x=1mm, y=1mm]
\node at (0,0) {\includegraphics[width=43mm,trim=74.5mm 195.5mm 73.5mm 45.5mm, clip]{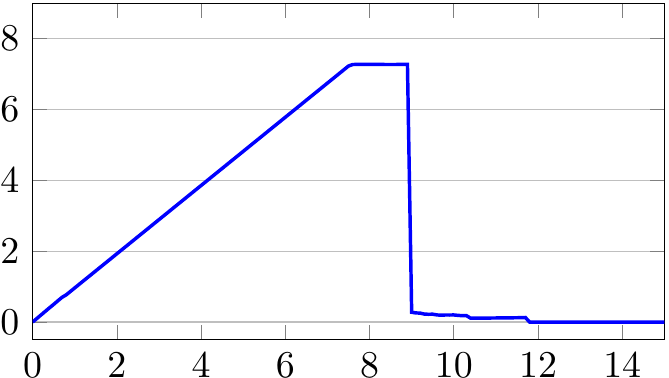}};
\end{tikzpicture}
\caption*{$q_*\mapsto\mathfrak{Q}(q_*,10)$}
\end{subfigure}
\caption{Numerical simulations corresponding to \eqref{e:data_ss-A}, \eqref{e:data_ss-B} and for different values of $T$.}
\label{f:opt-T=3}
\end{figure}

In order to have the same wave structure outlined at the beginning of Section \ref{ss:2} (namely, a supersonic perturbation $u_i$ which is separated from the state $u_\ell$ by a $2$-shock wave moving toward the valve) we now replace \eqref{e:data02-A} with
\begin{align}\label{e:data_ss-A}
&a=1,&
&\rho_i=3,&
&q_i=7.5,&
&\rho_r=8,&
&q_r=0,
\end{align}
which lead to 
\begin{align}\label{e:data_ss-B}
\rho_\ell &\approx 0.75,&
v_\ell&=1,&
q_\ell&\approx 0.75,&
\rho_r&\approx 8
\end{align}
Notice that $v_i = 2.5\cdot a > v_*^{\sup} \approx 1.63 \cdot a$.
The last picture in \figurename~\ref{f:opt-T=3} resembles the last picture in \figurename~\ref{f:numCASE1}.
This is probably due to the fact that the solution corresponding to the initial datum \eqref{e:ic3} with $u_i$, $u_\ell$ and $u_r$ given by \eqref{e:data_ss-A}, \eqref{e:data_ss-B} converges for $t\to\infty$ to the solution of the Riemann problem corresponding to the states $u_i$ and $u_r$.

\subsection{The case \texorpdfstring{$\rsp[u_i,u_\ell]$ is a $2$-rarefaction}{}}
\label{ss:3}
In this final subsection we pursue the analysis of a perturbation interacting with the valve from the left, that we began in Subsection \ref{ss:2} with the case of a $2$-shock wave, by considering the case of a $2$-rarefaction wave. In this case, as we mentioned above, analytically computations are too heavy to be provided, and therefore we focus on numerical simulations. More precisely we consider the data
\begin{align}
\label{e:data-opti-rare}
&a = 1,&
&\rho_i = 3,&
&q_i = 0,&
&\rho_\ell \approx 8.15,&
&q_\ell \approx 8.15,&
&\rho_r = 8,&
&q_r = 0,
\end{align}
Notice that the values of $a$, $\rho_i$, $\rho_r$ and $q_r$ are as in \eqref{e:data02-A} and \eqref{e:data_ss-A}. The following \figurename~\ref{f:opt-rare} shows our numerical simulations, and has to be compared with Figures \ref{f:opt-T=2} and \ref{f:opt-T=3}. We notice a similar behavior of the function $\mathfrak{Q}(q_*,T)$, which is interpreted as in the previous case, see \eqref{e:rsl}.  

\begin{figure}[!htb]\centering
\begin{subfigure}[t]{43mm}
\centering
\begin{tikzpicture}[every node/.style={anchor=south west,inner sep=0pt},x=1mm, y=1mm]
\node at (0,0) {\includegraphics[width=43mm,trim=73mm 195.5mm 73.5mm 45.5mm, clip]{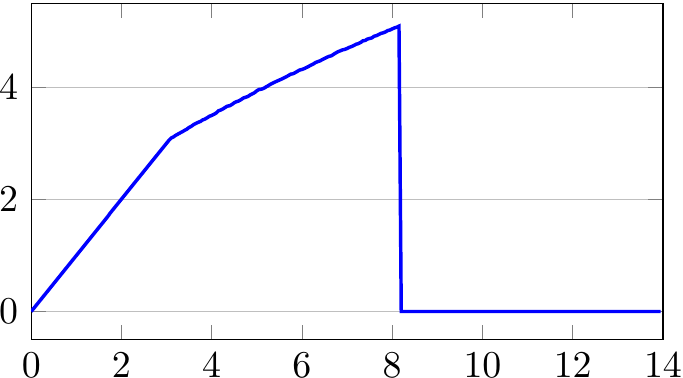}};
\end{tikzpicture}
\caption*{$q_*\mapsto\mathfrak{Q}(q_*,0.8)$}
\end{subfigure}
\hspace{5mm}
\begin{subfigure}[t]{43mm}
\centering
\begin{tikzpicture}[every node/.style={anchor=south west,inner sep=0pt},x=1mm, y=1mm]
\node at (0,0) {\includegraphics[width=43mm,trim=73mm 195.5mm 73.5mm 45.5mm, clip]{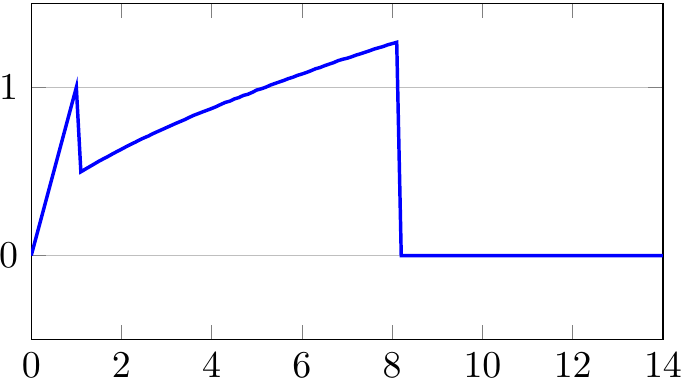}};
\end{tikzpicture}
\caption*{$q_*\mapsto\mathfrak{Q}(q_*,3.2)$}
\end{subfigure}
\hspace{5mm}
\begin{subfigure}[t]{43mm}
\centering
\begin{tikzpicture}[every node/.style={anchor=south west,inner sep=0pt},x=1mm, y=1mm]
\node at (0,0) {\includegraphics[width=43mm,trim=73mm 195.5mm 73.5mm 45.5mm, clip]{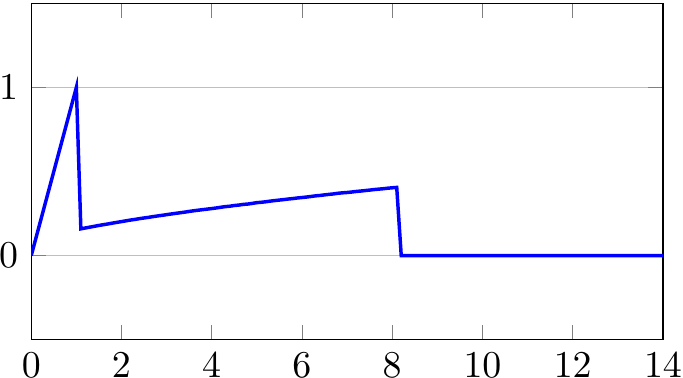}};
\end{tikzpicture}
\caption*{$q_*\mapsto\mathfrak{Q}(q_*,10)$}
\end{subfigure}
\caption{Numerical simulations corresponding to \eqref{e:data-opti-rare} for different values of $T$.}
\label{f:opt-rare}
\end{figure}

\section*{Acknowledgment}
A.~Corli and M.D.~Rosini are members of the {\em Gruppo Nazionale per l'{Ana\-li\-si} Matematica, la Probabilit\`{a} e le loro Applicazioni} (GNAMPA) of the {\em Istituto Nazionale di Alta Matematica} (INdAM) and acknowledge financial support from this institution.
The second author acknowledges the support of the R\'egion Bourgogne Franche-Comt\'e, projet 2017–2020 \lq\lq Analyse math\'ematique et simulation num\'erique d’EDP issus de probl\`emes de contr{\^o}le et du trafic routier\rq\rq, and Instytut Matematyki, Uniwersytet Marii Curie-Skłodowskiej, and Dipartimento di Matematica e Informatica, Università degli Studi di Ferrara for the hospitality during the preparation of the paper.
The last author acknowledges the support of the National Science Centre, Poland, Project \lq\lq Mathematics of multi-scale approaches in life and social sciences\rq\rq\ No.~2017/25/B/ST1/00051 and by University of Ferrara, FIR Project 2019 \lq\lq Leggi di conservazione di tipo iperbolico: teoria ed applicazioni\rq\rq.

{\bibliography{refe}
\bibliographystyle{abbrv}}


\end{document}